\documentclass[11pt,twoside]{article}
\usepackage{times}
\usepackage{CJK}
\usepackage{mathptmx}
\usepackage{amsmath,amssymb,amsthm}
\usepackage{enumerate}
\usepackage{cite}
\usepackage{mathrsfs,graphicx}
\usepackage{float}
\usepackage{stmaryrd}
\usepackage[margin=1in]{geometry}  % set the margins to 1in on all sides
\usepackage{graphicx}
\usepackage{amssymb}             % to include figures
\usepackage{amsmath}
\usepackage{slashed}            % great math stuff
\usepackage{amsfonts}              % for blackboard bold, etc
\usepackage{amsthm}
\usepackage{breqn}
\usepackage{verbatim}              % better theorem environments
\usepackage{enumitem}
\usepackage{fancyhdr}% http://ctan.org/pkg/fancyhdr
\usepackage{color}
\usepackage[pagewise]{lineno}%\linenumbers
\usepackage{amssymb}
\usepackage{slashed}
\begin{document}

\footskip=0pt
\footnotesep=2pt

\allowdisplaybreaks

\newtheorem{claim}{Claim}[section]

\theoremstyle{definition}
\newtheorem{thm}{Theorem}[section]
\newtheorem*{thmm}{Theorem}
\newtheorem{mydef}{Definition}[section]
\newtheorem{lem}[thm]{Lemma}
\newtheorem{prop}[thm]{Proposition}
\newtheorem{remark}[thm]{Remark}
\newtheorem*{propp}{Proposition}
\newtheorem{cor}[thm]{Corollary}
\newtheorem{conj}[thm]{Conjecture}
\def\F{\mathcal{F}}
\def\la{\lambda}
\def\tR{\tilde{R}}
\def\ds{\displaystyle}
\def\ve{\varepsilon}
\def\ls{\lesssim}
\newcommand{\bd}[1]{\mathbf{#1}}  % for bolding symbols
\newcommand{\R}{\mathbb{R}}      % for Real numbers
\newcommand{\ZZ}{\mathbb{Z}}      % for Integers
\newcommand{\Om}{\Omega}
\newcommand{\bv}{\mathbf{v}}
\newcommand{\bn}{\mathbf{n}}
\newcommand{\bU}{\mathbf{U}}
\newcommand{\q}{\quad}
\newcommand{\p}{\partial}
\newcommand{\n}{\nabla}
\newcommand{\f}{\frac}

\numberwithin{equation}{section}

\title{Almost global smooth solutions of the 3D quasilinear Klein-Gordon equations on the product space~$\mathbb{R}^{2}\times \mathbb{T}$}
\author{}

\author{Li Jun$^{1,*}$,\quad Tao Fei$^{1,*}$, \quad Yin Huicheng$^{2,}$
\footnote{Li Jun (lijun@nju.edu.cn) is supported by the NSFC (No.11871030),
Tao Fei (dg1721009@smail.nju.edu.cn) and  Yin Huicheng
(huicheng@nju.edu.cn, 05407@njnu.edu.cn) are
supported by the NSFC (No.11731007).}\vspace{0.5cm}\\
\small  1. Department of Mathematics, Nanjing University, Nanjing, 210093, China.
\\
\small 2. School of Mathematical Sciences and Mathematical Institute, \\
\small Nanjing Normal University, Nanjing 210023, China.
}
\date{}
\maketitle

\date{}
\maketitle
\thispagestyle{empty}
\begin{abstract}
In the paper, for the 3D quasilinear Klein-Gordon equation with the small initial data
posed on the product space~$\mathbb{R}^{2}\times \mathbb{T}$, we focus on the lower bound of the
lifespan of the smooth solution. When the size of initial data is bounded by $\ve_0>0$, by the  space-time resonance method,
it is shown that smooth solution exists up to the time $e^{c_{0}/\epsilon_{0}^2}$ with $\ve_0$ being
sufficiently small and $c_0>0$ being some suitable constant.
\end{abstract}

\vskip 0.2cm

{\bf Keywords:}  Quasilinear Klein-Gordon equation, almost global solution, Z-norm,

\qquad \qquad  \quad Littlewood-Paley decomposition,  space-time resonance method, energy estimate

{\bf Mathematical Subject Classification 2000:} 35L70, 35L65,
35L67

\vskip 0.2 true cm

\centerline{\bf Contents}
\contentsline {section}{\numberline {1}Introduction}{2}
\contentsline {section}{\numberline {2}Definitions~and~function spaces}{5}
\contentsline {section}{\numberline {3}Some preliminary estimates}{7}
\contentsline {section}{\numberline {4}Two crucial propositions and proof of~Theorem~1.1}{17}
\contentsline {section}{\numberline {5}Proof of Proposition~4.1}{20}
\contentsline {section}{\numberline {6}Proof of Proposition~4.2}{26}
\contentsline {subsection}{\numberline {6.1}Reformulation of $U$}{26}
\contentsline {subsection}{\numberline {6.2}Preliminary for the proof of Proposition \ref{prop4-2}}{27}
\contentsline {subsection}{\numberline {6.3}Proof of Lemma~6.1.}{28}
\contentsline {subsection}{\numberline {6.4}Proof of Lemma~6.2.}{32}
\contentsline {section}{\numberline {7}Appendix~A. Bilinear and trilinear estimates}{42}
\contentsline {section}{\numberline {8}Appendix~B. Temporal decay estimate for the 3D linear Klein-Gordon problem}{46}
\contentsline {section}{\numberline {}References}{49}

\section{Introduction}\label{I}

We consider the following Cauchy problem of the 3D quasilinear Klein-Gordon equation
\begin{equation}\label{0-1}
\begin{cases}
(\partial_{t}^2-\Delta_{x,y}+1)u=F(u,\partial u,\partial^{2} u),\\[2mm]
(u(0),\dot{u}(0))=(u_{0},u_{1})(x,y),
\end{cases}
\end{equation}
where $x=(x_1,x_2)\in\Bbb R^2$, $y\in\mathbb{T}=\mathbb{R}/(2\pi\mathbb{Z})$,
$\dot u=\p_t u$, $\partial=(\partial_{0},\partial_{1},\partial_{2},\partial_{3})$
with $\partial_{0}=\partial_{t},~\partial_{j}=
\partial_{x_{j}}$ $(j=1,2)$ and $\partial_{3}=\partial_{y}$,
$\Delta_{x,y}=\Delta_{x}+\p_{y}^2=\sum\limits_{j=1}^{3}\partial_{j}^{2}$,
$F(u,\partial u,\partial^{2} u)=\ds\sum_{j,k=0}^{3}G^{jk}(u,\partial u)\partial_{jk}^2u+Q(u,\partial u)$
with $G^{jk}(u,\partial u)=\ds\sum_{l=0}^{3}g^{jkl}\partial_{l}u+h^{jk}u$, $g^{jkl}, h^{jk}\in \mathbb{R}$,
$g^{jkl}=g^{kjl}$, $h^{jk}=h^{kj}$, and
$Q(u,\partial u)$ is a quadratic form of $(u,\partial u)$. Without loss of generality, $G^{00}(u,\p u)\equiv 0$
can be assumed.

In order to state the main result, we define the normalized initial data with respect to problem \eqref{0-1} as
\begin{equation}\label{0-31}
U_{0}(x,y):=u_{1}(x,y)-i\Lambda u_{0}(x,y),
\end{equation}
where the operator $\Lambda=(1-\Delta_{x,y})^{1/2}$ and $i=\sqrt{-1}$.

Hereafter, we also use the following notations: for two non-negative quantities $A, B$ and some generic
positive constant $C>1$, $A\ls B$ means $A\le CB$, and $A\thickapprox B$ means $\frac{1}{C}B\leq A\leq CB$.

Based on the preparations above, our main result can be stated as:
\begin{thm}\label{mainthm}
{\it For $N\geq 220$, there exist small positive constants $\bar\ve$ and $c_0$ such that if
$U_0$ in \eqref{0-31} satisfies the smallness condition
\begin{equation}\label{0-3}
\|U_{0}\|_{H^{N}(\mathbb{R}^{2}\times \mathbb{T})}+\|U_{0}\|_{Z}
=\varepsilon_{0}\leq\overline{\varepsilon},
\end{equation}
where the $Z$-norm of $U_0$ will be defined in Definition 2.1 below, then problem (1.1) admits a unique almost global
solution $u\in C([0,e^{c_{0}/\epsilon_{0}^{2}}], H^{N+1}(\mathbb{R}^{2}\times \mathbb{T}))
\cap C^{1}([0,e^{c_{0}/\epsilon_{0}^{2}}], H^{N}(\mathbb{R}^{2}\times \mathbb{T}))$.
Moreover, for any $t\in [0,e^{c_{0}/\epsilon_{0}^{2}}]$,
\begin{subequations}\label{0-4}
\begin{align}
&\sum\limits_{|\rho|\leq2}\|\partial_{x,y}^{\rho}u(t)\|_{L^{\infty}(\mathbb{R}^{2}\times \mathbb{T})}+\sum\limits_{|\rho|\leq1}\|\partial_{x,y}^{\rho}\dot{u}(t)\|_{L^{\infty}(\mathbb{R}^{2}\times \mathbb{T})}
\lesssim \varepsilon_{0}(1+t)^{-1},\label{0-4-1}\\
&\|u(t)\|_{H^{N+1}(\mathbb{R}^{2}\times \mathbb{T})}
+\|\dot{u}(t)\|_{H^{N}(\mathbb{R}^{2}\times \mathbb{T})}\lesssim \varepsilon_{0}(1+t)^{A_{0}\varepsilon_{0}},\label{0-4-2}
\end{align}
\end{subequations}
where $A_{0}>0$ is a universal constant.}
\end{thm}

\begin{remark}
If $(u_{0},u_{1})=\epsilon (\widetilde{u}_{0},\widetilde{u}_{1})$ is chosen, where $\epsilon>0$ is
sufficiently small and~$(\widetilde{u}_{0},\widetilde{u}_{1})\in C_{0}^{\infty}(\mathbb{R}^{2}\times \mathbb{T})$, then
it follows from direct computation that condition \eqref{0-3} is fulfilled.
\end{remark}

\begin{remark}\label{rem1-2}
{\it For the 1D quadratic semilinear Klein-Gordon equation
\begin{equation}\label{Y0Y-5}
\begin{cases}
(\partial_{t}^2-\p_x^2+1)u=F(u,\partial u),\\
(u(0),\dot{u}(0))=\ve (u_{0},u_{1})(x),
\end{cases}
\end{equation}
where $x\in\Bbb R$, $(u_0, u_1)(x)\in C_0^{\infty}(\Bbb R)$, the authors in \cite{MTT}
establish that the solution $u$ of \eqref{Y0Y-5} exists at least on the time interval
$[0, T_{\ve})$ with $T_{\ve}\ge Ae^{B/\ve^2}$, where the positive constants A and B depend on the suitable norms of $u_0$
and $u_1$. On the other hand, the lower bound of $T_{\ve}$ is in general optimal (see Proposition 7.8.8 of \cite{H}).
It is noted that the lifespan obtained in Theorem 1.1 is consistent with the lifespan $T_{\ve}$
for the 1D problem \eqref{Y0Y-5}.}
\end{remark}

\begin{remark}\label{rem1-3}
{\it Consider the 2D case of problem \eqref{0-1}
\begin{equation}\label{0Y-2}
\begin{cases}
(\partial_{t}^2-\Delta_{x,y}+1)u=F(u,\partial u,\partial^{2} u),\\
(u(0),\dot{u}(0))=(u_{0},u_{1})(x,y),
\end{cases}
\end{equation}
where $x\in\Bbb R$, $y\in\mathbb{T}=\mathbb{R}/(2\pi\mathbb{Z})$.
With minor modifications of the proof in the current paper, we may show that there exist
small positive constants $\bar\ve$ and $c_0$ such that
if the corresponding condition \eqref{0-3} holds, then
problem \eqref{0Y-2} has a unique solution~$u\in C([0, c_{0}\epsilon_{0}^{-2}], H^{N+1}(\mathbb{R}\times \mathbb{T}))\cap C^{1}([0, c_{0}\epsilon_{0}^{-2}], H^{N}(\mathbb{R}\times \mathbb{T}))$. Moreover, for $t\in [0, c_{0}\epsilon_{0}^{-2}]$,
\begin{equation}\label{0Y-3}
\sum\limits_{|\rho|\leq2}\|\partial_{x,y}^{\rho}u(t)\|_{L^{\infty}(\mathbb{R}\times \mathbb{T})}+\sum\limits_{|\rho|\leq1}\|\partial_{x,y}^{\rho}\dot{u}(t)\|_{L^{\infty}(\mathbb{R}\times \mathbb{T})}
\lesssim~\varepsilon_{0}(1+t)^{-1/2}.
\end{equation}}
\end{remark}

\begin{remark}\label{rem1-4}
{\it For the 3D linear homogeneous Klein-Gordon equation
\begin{equation}\label{0-5-1}
\begin{cases}
(\partial_{t}^2-\Delta_{x,y}+1)V=0,\\[2mm]
(V(0),\dot{V}(0))=(V_{0},V_{1})(x,y)\in C_{0}^{\infty}(\mathbb{R}^{2}\times \mathbb{T}),
\end{cases}
\end{equation}
we have proved that in Theorem \ref{8-1} of Appendix B,
\begin{equation}\label{0Y-1}
\sup\limits_{|\rho|\leq2}\|\partial_{x,y}^{\rho}V(t)\|_{L^{\infty}(\mathbb{R}^{2}\times \mathbb{T})}+\sup\limits_{|\rho|\leq1}\|\partial_{x,y}^{\rho}\dot{V}(t)\|_{L^{\infty}(\mathbb{R}^{2}\times \mathbb{T})}
\lesssim (1+t)^{-1}.
\end{equation}
From the estimates \eqref{0-4-1} and \eqref{0Y-1},
we know that the time-decay rate of the solution $u$ to the quasilinear problem \eqref{0-1}
is consistent with that of the solution $V$ to the corresponding linear problem \eqref{0-5-1}.}
\end{remark}

\begin{remark}\label{rem1-5}
{\it With respect to the 4D case of \eqref{0-1}
\begin{equation}\label{Y0Y-4}\begin{cases}
(\partial_{t}^2-\Delta_{x,y}+1)u=F(u,\partial u,\partial^{2} u),\\[2mm]
(u(0),\dot{u}(0))=(u_{0},u_{1})(x,y),
\end{cases}
\end{equation}
where $x\in\Bbb R^3$, $y\in\mathbb{T}=\mathbb{R}/(2\pi\mathbb{Z})$,
under the corresponding condition \eqref{0-3}, we have established the global existence of
$u$ to \eqref{Y0Y-4} in \cite{Tao-Yin}. Since problem \eqref{Y0Y-4} has the higher space dimensions
than problem \eqref{0-1}, the solution $u$ of problem \eqref{Y0Y-4} can be shown to admit the better
time-decay rate as $(1+t)^{-3/2}$ and further  derive the global existence of $u$.}
\end{remark}

\begin{remark}\label{rem1-6}
{\it It is expected that problem \eqref{0-1} will have a global small data solution as for the 2D
quasilinear Klein-Gordon equation with the smooth and sufficiently fast decaying initial data (see \cite{OTT}
and \cite{ST}).}
\end{remark}

Let us recall some remarkable results on the nonlinear Klein-Gordon equation in $\Bbb R^{1+n}$ ($n\ge 1$)
\begin{equation}\label{0-E}\begin{cases}
\Box u+m^2u=F(u,\partial u,\partial^{2} u),\\[2mm]
(u(0),\dot{u}(0))=\ve (u_{0},u_{1})(x),
\end{cases}
\end{equation}
where $\Box=\p_0^2-\p_1^2-\cdots-\p_n^2$,
$\ve>0$ is sufficiently small, $m$ is a non-zero real number, $F$ is at least quadratic with respect to its
arguments and linear in $\partial^{2} u$.

\vskip 0.2 true cm

{\bf $\bullet$} The cases of $(u_{0},u_{1})\in C_0^{\infty}(\Bbb R^n)$ or  $(u_{0},u_{1})\in C^{\infty}(\Bbb R^n)$
with the fairly rapid decay at infinity

\vskip 0.1 true cm

When $n\ge 2$,  it is known that problem \eqref{0-E} admits a global smooth solution $u$ from the works of
\cite{K}, \cite{S}, \cite{OTT}
and \cite{ST}; When $n=1$ and the nonlinearity $F$ satisfies the corresponding null
condition, problem \eqref{0-E} has a global smooth solution $u$ (see \cite{De-01}).

\vskip 0.2 true cm

{\bf $\bullet$} The cases of $(u_{0},u_{1})$ with  slow decay  at infinity just as $H^s(\Bbb R^n)$ functions

\vskip 0.1 true cm
When $F=F(u,\p u)$ in problem \eqref{0-E} satisfies the suitable null condition,
the lifespan of $u$ are at least $Ce^{C\epsilon^{-1}}$ for $n\ge 3$, ~$Ce^{C\epsilon^{-2/3}}$ for
$n=2$ and $C\ve^{-4}|ln\ve|^{-6}$ for $n=1$ respectively,
where $C>0$ is some appropriate constant
(see \cite{DF} and \cite{JDX}).

\vskip 0.2 true cm

{\bf $\bullet$} The cases of the periodic initial data $(u_{0},u_{1})$

\vskip 0.1 true cm

For $n=1$ and $F=F(u)$ (or even $F(x,u)$), the authors in \cite{BJ}, \cite{Bam} and \cite{Bam-G} have shown
that: For any $N>0$, when $(u_0,u_1)\in H^{s+1}\times H^s$ for some $s$ depending on $N$, if $m$ stays
outside an exceptional subset of zero measure, then the solution $u$ of problem \eqref{0-E} exists at least
on an interval of length $C_N\ve^{-N}$.

For $n\ge 1$ and $F=F(u,\p u)$, it follows from Theorem 1.1.1 of \cite{JMD4} that the lifespan $T_{\ve}$ is at least $C\ve^{-2}$
(when $F$ vanishes of order $r$ at $0$ for $r\ge 3$, then $T_{\ve}\ge C\ve^{-(r-1)}|ln\ve|^{-(r-3)}$). In addition, when
$F=F(u)$ (or even $F(x,u)$), by \cite{JMD1},
there exists a subset $Q\subset (0,\infty)$ of zero measure with the
property that for every $m\in (0,\infty)\setminus Q$ and any $A>1$, there exists
$s_0>0$ such that problem \eqref{0-E} has a unique solution $u\in C([0, T_{\ve}], H^{s_0+1})\cap  C^1([0, T_{\ve}], H^{s_0})$
with $T_{\ve}\ge C\ve^{-(r-1)(1+2/n)}|ln\ve|^{-A}$ when $F$ vanishes of order $r$ at $0$ for $r\ge 2$.
For more general $F=F(u,\p u, \p^2u)$,
the related results can be referred to \cite{JMD3}.

\vskip 0.2 true cm

So far, to our best knowledge, there are few results related to
the long time behaviours of the general Klein-Gordon equation of problem \eqref{0-1} with
the initial data $(u_0,u_1)$ posed on the product spaces~$\mathbb{R}^{m_1}\times \mathbb{T}^{m_2}$. For the semilinear
Klein-Gordon equation $(\partial_{t}^2-\Delta_{x,y}+1)u=\pm |u|^{p-1}u$ with $p>1$, $(x,y)\in\Bbb R^{m_1}\times
\mathbb{T}^{m_2}$, $m_1\ge 1$, $m_2=1,2$, and $m_1+m_2\ge 3$, under some suitable restrictions on $m_1, m_2$ and $p$,
the authors in \cite{HV1,HV2} prove the global existence of small data solution
$u\in C([0, \infty), H^1_{x,y})\cap C^1([0, \infty), L^2_{x,y})$ and
$u\in L^p([0, \infty), L^{2p}_{x,y})$, meanwhile, the scattering of the large solution $u$ is also obtained in \cite{FH1}.
In our recent paper \cite{Tao-Yin}, we have shown the global existence of smooth small data solution to
the 4D case of problem \eqref{0-1}
on $\mathbb{R}^{3}\times \mathbb{T}$, where the solution admits the better time-decay rate
like $(1+t)^{-3/2}$. In the present paper,
the  almost global smooth solution of problem \eqref{0-1} will be established.

Let us give the comments on the proof of Theorem 1.1. Motivated by \cite{IP2},
where the global smooth and small  perturbed solution of the 2D Euler-Poisson system for the electrons
is established, we will study problem \eqref{0-1} by the  space-time resonance method (more detailed introductions of
space-time resonance can be referred to \cite{GMS}-\cite{GNT}). Note that we are working on the product space
$\mathbb{R}^{2}\times \mathbb{T}$ for problem \eqref{0-1}, then the estimate of $u$ on the periodic direction $y$
should be specially concerned. For this purpose, except for carrying out space-frequency decomposition
on $\mathbb{R}^{2}$, we require to localize the Fourier modes by the resolution of
unity in $y$ direction (see \eqref{2-7} below). This leads to establish various dispersive
estimates of the resulting terms on $\mathbb{R}^{2}\times \mathbb{T}$. In addition, according to
the structure of the equation in
problem~(1.1), it is natural to define the modified energy for $N\in\Bbb N$ as:
\begin{equation}\label{0-666}
\begin{aligned}
\mathcal{E}_{N+1}(t):=&\sum\limits_{|\rho|\leq N}
\Big\{\int_{\mathbb{R}^{2}}\int_{\mathbb{T}}\Big[(\partial_{t}\partial_{x,y}^{\rho}u)^{2}
+\sum\limits_{j=1}^{3}(\partial_{j}\partial_{x,y}^{\rho}u)^{2}
+(\partial_{x,y}^{\rho}u)^{2}\Big]dydx
\\&~~~~~~~~~~~~~~~~\quad +\int_{\mathbb{R}^{2}}\int_{\mathbb{T}}
\sum\limits_{j,k=1}^{3}G^{jk}(u,\partial u)
\partial_{j}\partial_{x,y}^{\rho}u\partial_{k}\partial_{x,y}^{\rho}udydx\Big\}.
\end{aligned}
\end{equation}
Obviously, $\|u(t)\|_{H^{N+1}(\mathbb{R}^{2}\times \mathbb{T})}+\|\dot{u}(t)\|_{H^{N}(\mathbb{R}^{2}\times \mathbb{T})}\thickapprox\mathcal{E}^{1/2}_{N+1}(t)$ for the small solution $u$ of problem \eqref{0-1}.
As in \cite{IP}, by the standard energy estimate, one can derive the following estimate for problem \eqref{0-1},
\begin{equation}\label{0-7}
\begin{aligned}
\mathcal{E}_{N+1}(t)-\mathcal{E}_{N+1}(0)
&\lesssim \int_{0}^{t}\mathcal{E}_{N+1}(s)\cdot
\Big(\sum\limits_{|\rho|\leq2}\|\partial_{x,y}^{\rho}u(s)\|_{L^{\infty}(\mathbb{R}^{2}\times \mathbb{T})}
+\sum\limits_{|\rho|\leq1}\|\partial_{x,y}^{\rho}\dot{u}(s)
\|_{L^{\infty}(\mathbb{R}^{2}\times \mathbb{T})}\Big)ds\\
&\leq \int_{0}^{t}\mathcal{E}_{N+1}(s)\cdot (1+s)^{-1}\cdot\theta(s)ds,
\end{aligned}
\end{equation}
where
\begin{equation}\label{0-7-1}
\theta(s):=\sup\limits_{\tau\in [0,s]}\Big[(1+\tau)\Big(\sum\limits_{|\rho|\leq2}
\|\partial_{x,y}^{\rho}u(\tau)\|_{L^{\infty}(\mathbb{R}^{2}\times \mathbb{T})}
+\sum\limits_{|\rho|\leq1}\|\partial_{x,y}^{\rho}\dot{u}(\tau)
\|_{L^{\infty}(\mathbb{R}^{2}\times \mathbb{T})}\Big)\Big].
\end{equation}
Together with Gronwall's inequality, this yields
\begin{equation}\label{0-8}
\begin{aligned}
\mathcal{E}_{N+1}^{1/2}(t)\leq \mathcal{E}_{N+1}^{1/2}(0)(1+t)^{C\cdot\theta(t)}.\nonumber
\end{aligned}
\end{equation}
Therefore, once $\theta(t)\lesssim\ve_0$ is derived, then the estimates in \eqref{0-4} are obtained.
To control $\theta(t)$,  as in \cite{IP2},  we shall construct an appropriate $Z$-norm (see \eqref{2-1A} below) for the
profile $V(t)$, where $V(t):=e^{it\Lambda}U(t)$ with $U(t):=\dot{u}(t)-i\Lambda u(t)$.
In this case, $\theta(t)$ can be controlled by the dispersive estimate and the
uniform bound of $\|V(t)\|_{Z}$~(see \eqref{4-12ASS-1}-\eqref{4-12ASS-2}).
Based on this, it follows from the standard continuation argument that  the proof of Theorem 1.1 is finished.

This paper is organized as follows. In Section 2, the main notations and definitions are collected,
meanwhile the required Z-norm is introduced. In Section 3, some preliminary results are derived.
In Section 4, the related estimates for $Z$-norm and bootstrap argument are arranged in Proposition 4.1
and Proposition 4.2 respectively.
As a consequence of Proposition 4.1, Proposition 4.2 and the local well-posedness
of problem \eqref{0-1}, the proof of Theorem 1.1 is completed. The proofs of Proposition 4.1 and Proposition 4.2
are put in Section 5 and Section 6, respectively. In Appendix A, two lemmas on the bilinear and
trilinear estimates are established.
The temporal decay rate of solutions to the linear homogeneous Klein-Gordon
equations on $\mathbb{R}^{2}\times \mathbb{T}$ is given in Appendix B.

\section{Definitions and function spaces}

In this section, we introduce the notations of Fourier transformation, the tools of space-frequency
decomposition as well as $Z$-norm and the necessary pseudo-differential operators.

\subsection{Notations for Fourier transformation}

The partial or full Fourier transformations of $f(x,y)\in L^{1}(\mathbb{R}^{2}\times \mathbb{T})$ are denoted by
\begin{equation*}
f_{n}(x):=\frac{1}{2\pi}\int_{\mathbb{T}}e^{-iny}f(x,y)dy,~~
\mathcal{F}_{x,y}(f)(\xi,n)=\widehat{f_{n}}(\xi):=
\frac{1}{2\pi}\int_{\mathbb{R}^{2}}\int_{\mathbb{T}}e^{-i(\xi\cdot x+ny)}f(x,y)dydx,
\end{equation*}
where~$x\in \mathbb{R}^{2},~y\in \mathbb{T},~\xi\in \mathbb{R}^{2},~n\in \mathbb{Z}$. The inverse of the
Fourier transformation of the sequence~$\{g(\xi,n)\}_{n\in \mathbb{Z}}$ is defined as
\begin{equation*}
(\mathcal{F}_{\xi,n}^{-1}g)(x,y):=\frac{1}{(2\pi)^{2}}
\sum_{n\in \mathbb{Z}}e^{iny}\int_{\mathbb{R}^{2}}e^{ix\cdot\xi}g(\xi,n)d\xi.
\end{equation*}
In the same way, for $f_{\mu}(x,y)\in L^{1}(\mathbb{R}^{2}\times \mathbb{T})$, we set
\begin{equation*}
f_{\mu;n}(x):=\frac{1}{2\pi}\int_{\mathbb{T}}e^{-iny}f_{\mu}(x,y)dy,~~~\widehat{f_{\mu;n}}(\xi)
:=\frac{1}{2\pi}\int_{\mathbb{R}^{2}}\int_{\mathbb{T}}e^{-i(\xi\cdot x+ny)}f_{\mu}(x,y)dydx.
\end{equation*}\vspace{0.2cm}

\subsection{Tools for space-frequency decomposition and $Z$-norm}

For a smooth function $\varphi:\mathbb{R}\rightarrow[0,1]$, which is supported in $[-8/5,8/5]$ and equals to 1 in~$[-5/4,5/4]$, we set
\begin{equation*}
\begin{aligned}
&\varphi_{k}(\cdot):=\varphi(|\cdot|/2^{k})-\varphi(|\cdot|/2^{k-1})~~(k\in \mathbb{Z}),\\
&\psi_{-1}(\cdot):=1-\sum\limits_{j\geq0}\varphi_{j}(\cdot)=\varphi(2|\cdot|),\ \psi_{l}(\cdot):=\varphi_{l}(\cdot)~~ (l\in \mathbb{N}_0),\\
&\psi_{I}:=\sum_{m\in I\bigcap\mathbb{ Z}}\psi_{m}~\mathrm{and}~\varphi_{I}:=\sum_{m\in I\bigcap\mathbb{ Z}}\varphi_{m}
~~\mathrm{for~any~interval}~I\subset\mathbb{R}.
\end{aligned}
\end{equation*}

The frequency projection operators~$R_{k},R_{I}$~on~$\mathbb{R}^{2}$~are formally defined for $f=f(x)$ as
\begin{equation*}
\widehat{R_{k}f}(\xi):=\psi_{k}(\xi)\widehat{f}(\xi)~~(k\geq -1),\ \widehat{R_{I}f}(\xi):=\psi_{I}(\xi)\widehat{f}(\xi).
\end{equation*}

For $g(y)\in L^{2}(\mathbb{T})$, one has that
\begin{equation}
g(y)=\sum_{n\in \mathbb{Z}}\widehat{g}(n)e^{iny}=\sum_{l\geq-1}\sum_{n\in \mathbb{Z}}
\psi_{l}(n)\widehat{g}(n)e^{iny}:=\sum_{l\geq-1}S_{l}g(y),\nonumber
\end{equation}
where
\begin{equation*}
\begin{aligned}
S_{l}g(y):=\sum_{n\in \mathbb{Z}}\psi_{l}(n)\widehat{g}(n)e^{iny}~~~(l\geq-1).
\end{aligned}
\end{equation*}
Moreover, we define
\begin{equation*}
\begin{aligned}
S_{I}g(y):=\sum_{n\in \mathbb{Z}}\psi_{I}(n)\widehat{g}(n)e^{iny}~~~(I\subset\mathbb{R}).
\end{aligned}
\end{equation*}

The space-frequency projections are defined as
\begin{equation*}\label{2-5}
\varphi_{j}(x)\cdot R_{k}S_{l}f(x,y)=
\left\{
\begin{aligned}
&\varphi_{j}(x)\cdot R_{k}f_{0}(x),~~~~~~~~~~~~~~~~~~~~~~~~~~~~~~~\mathrm{if}~l=-1,\\[2mm]
&\sum_{n\in\mathbb{Z}}\varphi_{j}(x)\cdot R_{k}f_{n}(x)\psi_{l}(n)e^{iny},~~~~~~~~~~\mathrm{if}~l\geq0.\\
\end{aligned}
\right.
\end{equation*}

Under the above notations, we have the following atom decomposition
\begin{equation}\label{2-7}
\begin{aligned}
f(x,y)=\sum_{j\in \mathbb{Z}}\sum_{k\geq -1}\sum_{l\geq-1}\varphi_{j}(x)\cdot R_{k}S_{l}f.
\end{aligned}
\end{equation}
Based on this decomposition, it is natural to introduce the $Z$-norm required in Theorem \ref{mainthm}:

\vskip 0.2 true cm

\noindent{\bf Definition 2.1.}\label{Def-2-1} {\it According to the atom decomposition \eqref{2-7}, we introduce
\begin{equation}\label{2-11}
Z:=\Big\{f\in L^{2}(\mathbb{R}^{2}\times\mathbb{T}):~\|f\|_{Z}~<+\infty\Big\},
\end{equation}
where
\begin{equation}\label{2-1A}
\|f\|_{Z}:=
\sup\limits_{k,l\geq-1}2^{9(k+l)}
\Big[\|R_{k}S_{l}f\|_{L^{2}(\mathbb{R}^{2}\times\mathbb{T})}
+\sum\limits_{j\in \mathbb{Z}^{+}}2^{j}\|\varphi_{j}(x)\cdot R_{k}S_{l}f\|_{L^{2}(\mathbb{R}^{2}\times\mathbb{T})}\Big].
\end{equation}}

It is pointed out that the introduction of $Z$-norm in Definition 2.1 is strongly motivated by the
definition (2.5) in \cite{IP2}. Moreover, compared with \cite{IP2}, we additionally introduce
the Littlewood-Paley decomposition in the periodic direction.

\begin{remark}\label{rem2-1} {\it With the definition of $Z$-norm, one has
\begin{equation}\label{2-1-0}
\|R_k S_l f\|_{L^1(\mathbb{R}^2\times\mathbb{T})}\lesssim 2^{-9(k+l)}\|f\|_Z~~(k, l\geq -1).
\end{equation}
}
\end{remark}

\vskip 0.3cm

\subsection{Necessary pseudo-differential operators}

To study problem \eqref{0-1} on $\mathbb{R}^2\times\mathbb{T}$, with the notations in the above Fourier transformations, we
naturally introduce the operator $\Lambda=(1-\Delta_{x,y})^{1/2}$ as
\begin{equation}\label{2-1}
(\Lambda f)(x,y):=\sum_{n\in \mathbb{Z}}\sqrt{1-\Delta_{x}+n^{2}}f_{n}(x)e^{iny}
:=\sum_{n\in \mathbb{Z}}\mathcal{F}^{-1}_{\xi}\Big(\sqrt{1+|\xi|^{2}+n^{2}}\widehat{f}_{n}(\xi)\Big)(x)e^{iny}.
\end{equation}
In addition, define
\begin{equation}\label{2-2}
\begin{aligned}
(e^{it\Lambda}f)(x,y):=\sum_{n\in \mathbb{Z}}e^{it\sqrt{1-\Delta_{x}+n^{2}}}f_{n}(x)e^{iny}
:=\sum_{n\in \mathbb{Z}}\mathcal{F}^{-1}_{\xi}\Big(e^{it\sqrt{1+|\xi|^{2}+n^{2}}}\widehat{f}_{n}(\xi)\Big)(x)e^{iny}.
\end{aligned}
\end{equation}

The \textit{normalized solution}~$U$ and its associated \textit{profile}~$V$ for the solution~$u$
of problem \eqref{0-1} are defined as
\begin{equation}\label{2-8}
U(t):=\dot{u}(t)-i\Lambda u(t),\ V(t):=e^{it\Lambda}U(t).\\
\end{equation}
Moreover, we set
\begin{equation}\label{2-9}
U_{+}(t):=U(t),~~U_{-}(t):=\overline{U},~~V_{+}(t):=e^{it\Lambda}U(t),~~V_{-}(t):=\overline{V(t)}
\end{equation}
and
\begin{subequations}\label{2-1-1}\begin{align}
&\Lambda_{\mu ;n}(\xi):=\mu\Lambda_{n}(\xi):=\mu\sqrt{1+|\xi|^{2}+n^{2}},\label{2-10}\\[2mm]
&\Phi^{n, m}_{\mu,\nu}(\xi,\eta)
:=\Lambda_{n}(\xi)-\Lambda_{\mu; n-m}(\xi-\eta)-\Lambda_{\nu; m}(\eta)\label{2-11}
\end{align}
\end{subequations}
 with $n, m\in\mathbb{Z}, \xi, \eta\in\mathbb{R}^2$ and $\mu,\nu\in\{+,-\}$,

\section{Some preliminary estimates}

In order to prove Theorem \ref{mainthm}, some preliminary estimates will be established in this section.
Although most of the results are similar to the ones in Appendix A of \cite{Tao-Yin} for the 4D case, we still give the
details due to the different space dimensions, different expressions of solutions
between the 4D and 3D linear Klein-Gordon equations, and slightly different Littlewood-Paley decompositions.

\begin{lem}\label{HC-1}
{\it For $n\in \mathbb{Z}$ and $t>0$, the following estimate holds
\begin{equation}\label{7-0}
\Big\|\int_{\mathbb{R}^{2}}e^{ix\cdot\xi}e^{\pm it\sqrt{1+|\xi|^{2}+n^{2}}}
\psi_{-1}(\xi)d\xi\Big\|_{L^{\infty}(\mathbb{R}^{2})}\lesssim (1+|n|)(1+t)^{-1}.
\end{equation}}
\end{lem}

\begin{proof}
Since
\begin{align*}
\Big|\int_{\mathbb{R}^{2}}e^{ix\cdot\xi}e^{\pm it\sqrt{1+|\xi|^{2}+n^{2}}}
\psi_{-1}(\xi)d\xi\Big|\le \int_{\mathbb{R}^{2}}|\psi_{-1}(\xi)|d\xi\le 4,
\end{align*}
we may assume $t\geq 1$ in \eqref{7-0}. We rewrite
\begin{equation}
\int_{\mathbb{R}^{2}}e^{ix\cdot\xi}e^{\pm it\sqrt{1+|\xi|^{2}+n^{2}}}
\psi_{-1}(\xi)d\xi=\int_{\mathbb{R}^{2}}e^{i\frac{t}{(1+|n|)}\cdot\phi_{\pm}(\xi;~t,x,n)}
\psi_{-1}(\xi)d\xi,\nonumber
\end{equation}
where~$\phi_{\pm}(\xi;~t,x,n)=(1+|n|)((x\cdot\xi)/t\pm\sqrt{1+|\xi|^{2}+n^{2}})$.~The possible
critical point $\xi$ of $\phi_{\pm}$ is determined by
\begin{equation}
x/t\pm\xi/\sqrt{1+|\xi|^{2}+n^{2}}=0,\nonumber
\end{equation}
which has no real solution when $t\leq |x|$ and a unique solution~$\xi_{0}=\mp x\sqrt{\frac{1+n^2}{t^2-|x|^2}}$ when $t> |x|$.

If~$t\leq |x|$ or $|\xi_{0}|\geq 1$,
one can apply Lemma~2.34 of ~\cite{NS}~($\mathrm{Page}$ 65) to get \eqref{7-0} directly.

If $|\xi_{0}|\leq 1,$ by direct calculation
\begin{equation*}
\mathrm{det}~D^{2}_{\xi}\phi_{\pm}(\xi;~t,x,n)=(1+|\xi_{0}|^{2}+n^{2})^{-2}(1+n^{2})(1+|n|)^{2},
\end{equation*}
then the Hessian $D^{2}\phi_{\pm}(\xi;~t,x,n)$ is nondegenerate and
is uniformly bounded with respect to $\xi_{0}$. Hence,
\\Lemma 2.35 of \cite{NS} ($\mathrm{Page}$ 66) implies \eqref{7-0}.
\end{proof}

\begin{lem}\label{HC-1-1}
{\it Let $\sigma_{\mathbb{S}^{d-1}}$ be the surface measure of the unit sphere in $\mathbb{R}^{d}~(d\geq 1)$. Then one has
\begin{equation}\label{7-0-AAA}
\widehat{\sigma_{\mathbb{S}^{d-1}}}(x):=\int_{\mathbb{S}^{d-1}}e^{ix\cdot\xi}\sigma_{\mathbb{S}^{d-1}}(d\xi)
=e^{i|x|}\omega_{+}(|x|)+e^{-i|x|}\omega_{-}(|x|),\nonumber
\end{equation}
where $\omega_{\pm}$ are smooth and satisfy
\begin{equation}\label{7-0-BBB}
|\partial_{r}^{k}\omega_{\pm}(r)|\leq C_{k}\cdot r^{-\frac{d-1}{2}-k},~~~~~\forall~r>0,~~~k\geq 0.
\end{equation}}
\end{lem}

\begin{proof}
See Corollary 2.37 of \cite{NS}~($\mathrm{Page}$ 68). Here we point out that by examining
the proof for Corollary 2.37, the restrictions of $|x|\geq 1$ and $r\geq 1$ there can be removed.
\end{proof}

\begin{lem}\label{HC-2}
{\it For $n\in \mathbb{Z}$, $t>0$~and~$k\geq0$, there holds
\begin{equation}\label{7-1}
\Big\|\int_{\mathbb{R}^{2}}e^{ix\cdot\xi}e^{\pm it\sqrt{1+|\xi|^{2}+n^{2}}}
\varphi_{k}(\xi)d\xi\Big\|_{L^\infty(\mathbb{R}^{2})}\lesssim \frac{(2^{2k}+n^{2})}{\sqrt{1+n^{2}}}\cdot t^{-1}.
\end{equation}}
\end{lem}

\begin{proof}
Let $\ds\Phi_{k, n}^{\pm}(t, x):=\int_{\mathbb{R}^{2}}e^{ix\cdot\xi}e^{\pm it\sqrt{1+|\xi|^{2}+n^{2}}}
\varphi_{k}(\xi)d\xi$ and $\lambda:=2^{k}$. We only consider the
estimate of~$\Phi_{k,n}^{+}(t,x)$ since $\Phi_{k,n}^{-}(t,x)$ can be treated in a similar way.

It follows from Lemma~\ref{HC-1-1} and direct computation that
\begin{equation}\label{7-2}
\begin{aligned}
\Phi_{k,n}^{+}(t,x)&=\lambda^{2}\int_{\mathbb{R}^{2}}e^{i\lambda x\cdot\xi}e^{it\sqrt{1+|\lambda\xi|^{2}+n^{2}}}
\varphi_{0}(\xi)d\xi\\
&=\lambda^{2}\int_{0}^{\infty}\int_{\mathbb{S}^{1}}e^{i\lambda rx\cdot\theta}e^{it\sqrt{1+(\lambda r)^{2}+n^{2}}}
\varphi_{0}(r)rdrd\theta\\
&=\lambda^{2}\int_{0}^{\infty}e^{it\sqrt{1+(\lambda r)^{2}+n^{2}}}
\varphi_{0}(r)rdr\int_{\mathbb{S}^{1}}e^{i\lambda rx\cdot\theta}d\theta\\
&=\lambda^{2}\int_{0}^{\infty} e^{it\sqrt{1+(\lambda r)^{2}+n^{2}}}
\varphi_{0}(r)r\Big[\sum\limits_{\pm}e^{\pm i\lambda r|x|}\omega_{\pm}(\lambda r|x|)\Big]dr
\\&=\sum\limits_{\pm}\lambda^{2}\int_{0}^{\infty} e^{i\lambda\phi^{\pm}_{\lambda,n}(r;~t,|x|)}
\omega_{\pm}(\lambda r|x|)\varphi_{0}(r)rdr,
\end{aligned}
\end{equation}
where $\phi^{\pm}_{\lambda,n}(r;~t,|x|)=\frac{t}{\lambda}\sqrt{1+(\lambda r)^{2}+n^{2}}\pm r|x|$. In what follows,
we only consider the phase $\phi^{-}_{\lambda,n}(r;~t,|x|)$ in \eqref{7-2} since the treatment for $\phi^{+}_{\lambda,n}(r;~t,|x|)$
is easier.

\vskip 0.1 true cm

\textbf{Case 1: $t\thickapprox |x|(1+|n|/\lambda)\gtrsim (1+n^{2})/(\lambda^{3}+\lambda^{2}|n|)$ and $t>|x|$}

\vskip 0.1 true cm

Direct calculation shows
\begin{equation*}
\partial_{r}\phi^{-}_{\lambda,n}(r;~t,|x|)=\frac{t\lambda r}{\sqrt{1+(\lambda r)^{2}+n^{2}}}-|x|,\ \partial_{r}^{2}\phi^{-}_{\lambda,n}(r;~t,|x|)=\frac{t\lambda(1+n^{2})}{(1+(\lambda r)^{2}+n^{2})^{3/2}}.
\end{equation*}
Due to $\frac{5}{8}\leq r\leq \frac{8}{5}$ in the support of $\varphi_0(r)$, then
\begin{equation}\label{7-3}
\partial_{r}^{2}\phi^{-}_{\lambda,n}(r;~t,|x|)\thickapprox \frac{t\lambda(1+n^{2})}{(1+\lambda^{2}+n^{2})^{3/2}}:=\tau~~~~~(\frac{5}{8}\leq r\leq \frac{8}{5}).
\end{equation}
Note that $r=r_{0}=\frac{|x|\sqrt{1+n^{2}}}{\lambda\sqrt{t^{2}-|x|^{2}}}$ is the unique
solution of~$\partial_{r}\phi^{-}_{\lambda,n}(r;~t,|x|)=0$,
we now take the following decomposition
\begin{equation}
\lambda^{2}\int_{0}^{\infty} e^{i\lambda \phi^{-}_{\lambda,n}}
\omega_{-}(\lambda r|x|)\varphi_{0}(r)rdr:=H_1+H_2,\nonumber
\end{equation}
where
\begin{align*}
\ds H_1&=\lambda^{2}\int_{0}^{\infty} e^{i\lambda \phi^{-}_{\lambda,n}}
\omega_{-}(\lambda r|x|)\varphi_{0}(r)r\chi_{0}((\lambda \tau)^{1/2}(r-r_{0}))dr,\\
\ds H_2&=\lambda^{2}\int_{0}^{\infty} e^{i\lambda \phi^{-}_{\lambda,n}}
\omega_{-}(\lambda r|x|)\varphi_{0}(r)r\{1-\chi_{0}((\lambda \tau)^{1/2}(r-r_{0}))\}dr,
\end{align*}
and $\chi_{0}(r)$ is a cut-off function near $r=0$.
By \eqref{7-0-BBB}, one gets that when $\frac{5}{8}\leq r\leq \frac{8}{5}$,
\begin{equation}\label{7-4}
|\partial_{r}^{k}\omega_{-}(\lambda r|x|)|\leq C_{k}(\lambda r|x|)^{-1/2}\leq C_{k}(\lambda|x|)^{-1/2},~~\forall k\geq 0.
\end{equation}

With \eqref{7-4} for $k=0$, then
\begin{equation}\label{7-A2}
\begin{aligned}
|H_1|&\lesssim\lambda^{2}\int_{0}^{\infty} \varphi_0(r)|\omega_{-}(\lambda r|x|)\chi_{0}((\lambda \tau)^{1/2}(r-r_{0}))|dr
\lesssim\lambda^{2}(\lambda \tau)^{-1/2}(\lambda|x|)^{-1/2}.
\end{aligned}
\end{equation}

In addition, we rewrite $H_2$ as
\begin{equation*}
H_2=\lambda^{2}\int_{0}^{\infty} e^{i\lambda\tau \phi^{-}_{\lambda,n,\tau}}
\omega_{-}(\lambda r|x|)\varphi_{0}(r)r\{1-\chi_{0}((\lambda \tau)^{1/2}(r-r_{0}))\}dr,
\end{equation*}
where $\phi^{-}_{\lambda,n,\tau}=\phi^{-}_{\lambda,n}/\tau$. In order to estimate $H_2$ via the method of
non-stationary phase, due to $\partial_r\phi_{\lambda, n}^{-}(r; t, |x|)\neq 0$ in the
integrand of $H_2$, it is necessary to define the operator $\mathfrak{L}$ and then its adjoint operator $\mathfrak{L}^{\ast}$ as
\begin{equation}\label{3-3-3}
\mathfrak{L}:=\frac{1}{i\lambda\tau}\frac{\partial_{r}\phi^{-}_{\lambda,n,\tau}}{|\partial_{r}\phi^{-}_{\lambda,n,\tau}|^{2}}\partial_{r},~~~~ \mathfrak{L}^{\ast}(\cdot):=\frac{i}{\lambda\tau}\partial_{r}
\Big(\frac{\partial_{r} \phi^{-}_{\lambda,n,\tau}}{|\partial_{r} \phi^{-}_{\lambda,n,\tau}|^{2}}\cdot\Big).
\end{equation}
Here, it is easy to verify $\mathfrak{L} e^{i\lambda\tau \phi^{-}_{\lambda,n,\tau}}=e^{i\lambda\tau \phi^{-}_{\lambda,n,\tau}}$.

Due to $\partial_{s}^{2} \phi^{-}_{\lambda,n,\tau}(s;~t,|x|)\simeq1$ on the support of $\varphi_{0}(s)$, then $|\partial_{r}\phi^{-}_{\lambda,n,\tau}(r)|\gtrsim|r-r_{0}|$ and
\begin{equation}\label{7-5}
\Big|\partial_{r}^{l}\Big(\frac{\partial_{r} \phi^{-}_{\lambda,n,\tau}}
{|\partial_{r} \phi^{-}_{\lambda,n,\tau}|^{2}}\Big)(r)\Big|
\leq~C_{l}|r-r_{0}|^{-1-l}~~(l\in \mathbb{N}).
\end{equation}
By the method of non-stationary phase, it derives from \eqref{3-3-3}, \eqref{7-4} and \eqref{7-5} that
\begin{equation}\label{7-A3}
\begin{aligned}
|H_2|&=\Big|\lambda^{2}\int_{0}^{\infty} (\mathfrak{L}^{2} e^{i\lambda\tau \phi^{-}_{\lambda,n,\tau}})
\omega_{-}(\lambda r|x|)\varphi_{0}(r)r\{1-\chi_{0}((\lambda \tau)^{1/2}(r-r_{0}))\}dr\Big|\\
&=\Big|\lambda^{2}\int_{0}^{\infty}e^{i\lambda\tau \phi^{-}_{\lambda,n,\tau}}
(\mathfrak{L}^{\ast})^{2}\Big[\omega_{-}(\lambda r|x|)\varphi_{0}(r)r
\{1-\chi_{0}((\lambda \tau)^{1/2}(r-r_{0}))\}\Big]dr\Big|\\
&\lesssim\lambda^{2}\int_{0}^{\infty} \Big|(\mathfrak{L}^{\ast})^{2}\Big[\omega_{-}(\lambda r|x|)\varphi_{0}(r)r
\{1-\chi_{0}((\lambda \tau)^{1/2}(r-r_{0}))\}\Big]\Big|dr\\
&\lesssim\lambda^{2}(\lambda|x|)^{-1/2}(\lambda\tau)^{-2}\int_{|r-r_{0}|
\gtrsim(\lambda\tau)^{-1/2}}((\lambda\tau)(r-r_{0})^{-2}+(r-r_{0})^{-4})dr
\\&\lesssim\lambda^{2}(\lambda|x|)^{-1/2}(\lambda \tau)^{-1/2}.
\end{aligned}
\end{equation}
Thus \eqref{7-1} comes from \eqref{7-A2} and \eqref{7-A3} directly.

\textbf{Case 2: $t\thickapprox |x|(1+|n|/\lambda)\gtrsim (1+n^{2})/(\lambda^{3}+\lambda^{2}|n|)$ and $t\leq |x|$}

\vskip 0.1 true cm

In this case, by direct calculation, we have
\begin{equation*}
|\partial_{r}\phi^{-}_{\lambda,n}(r;~t,|x|)|\gtrsim \frac{t(1+n^{2})}{1+\lambda^{2}+n^{2}}.
\end{equation*}
Then we can apply Lemma~2.34 of~\cite{NS}~($\mathrm{Page}$ 65), \emph{i.e.}, the method of non-stationary phase, to obtain
\begin{equation*}
\Big|\lambda^{2}\int_{0}^{\infty} e^{i\lambda \phi^{-}_{\lambda,n}}
\omega_{-}(\lambda r|x|)\varphi_{0}(r)rdr\Big|\lesssim \lambda^{2}
\Big(\frac{\lambda t(1+n^{2})}{1+\lambda^{2}+n^{2}}\Big)^{-1}(\lambda|x|)^{-1/2}.
\end{equation*}
On the other hand, by \eqref{7-4}, one has
\begin{equation*}
\Big|\lambda^{2}\int_{0}^{\infty} e^{i\lambda \phi^{-}_{\lambda,n}}
\omega_{-}(\lambda r|x|)\varphi_{0}(r)rdr\Big|\lesssim \lambda^{2}(\lambda|x|)^{-1/2}.
\end{equation*}
Hence,
\begin{equation*}
\Big|\lambda^{2}\int_{0}^{\infty} e^{i\lambda \phi^{-}_{\lambda,n}}
\omega_{-}(\lambda r|x|)\varphi_{0}(r)rdr\Big|\lesssim \lambda^{2}
\Big(\frac{\lambda t(1+n^{2})}{1+\lambda^{2}+n^{2}}\Big)^{-1/2}(\lambda|x|)^{-1/2},
\end{equation*}
which suffices to get \eqref{7-1}.

\vskip 0.1 true cm

\textbf{Case 3: $t\thickapprox |x|(1+|n|/\lambda)\lesssim (1+n^{2})/(\lambda^{3}+\lambda^{2}|n|)$}

\vskip 0.1 true cm

\noindent By \eqref{7-2}, we have
\begin{equation}
\Phi_{k,n}^{+}(t,x)=\lambda^{2}\int_{\mathbb{R}^{2}}e^{i\lambda x\cdot\xi}e^{ it\sqrt{1+|\lambda\xi|^{2}+n^{2}}}
\varphi_{0}(\xi)d\xi.\nonumber
\end{equation}
This derives
\begin{equation}
|\Phi_{k,n}^{+}(t,x)|\lesssim\lambda^{2}\lesssim~\frac{(2^{2k}+n^{2})}{\sqrt{1+n^{2}}}\cdot t^{-1}.\nonumber
\end{equation}

\vskip 0.1 true cm

\textbf{Case 4}: $t\ncong|x|(1+|n|/\lambda)$

\vskip 0.1 true cm

In this case, we write $\Phi_{k,n}^{+}(t,x)$ as
\begin{equation}\label{7-6}
\Phi_{k,n}^{+}(t,x)
=\lambda^{2}\int_{\mathbb{R}^{2}}e^{i \lambda t\phi_{\lambda, n}(\xi;\ t, x)}\varphi_{0}(\xi)d\xi,
\end{equation}
where $\phi_{\lambda,n}(\xi;~t,x)=\frac{1}{\lambda}\sqrt{1+|\lambda\xi|^{2}+n^{2}}+\frac{x\cdot\xi}{t}$.

When  $t\geq|x|(1+|n|/\lambda)$, we may assume $t\geq100|x|(1+|n|/\lambda)$. Then one has
\begin{equation}\label{3-3-1}
|\partial_{\xi}\phi_{\lambda,n}(\xi;~t,x)|
\geq\frac{\lambda |\xi|}{\sqrt{1+|\lambda \xi|^{2}+n^{2}}}-\frac{|x|}{t}\geq \frac{\frac{\lambda}{2}}{6(\lambda+|n|)}-\frac{|x|}{\frac{100|x|}{\lambda}(\lambda+|n|)}\geq\frac{1}{24}\frac{\lambda}{\lambda+|n|}
\end{equation}
since $\sqrt{1+|\lambda \xi|^{2}+n^{2}}\leq 6(\lambda+|n|)$ holds for $5/8\leq|\xi|\leq8/5$ in the support of $\varphi_0(\xi)$.

When $t\leq|x|(1+\frac{|n|}{\lambda})$, we may assume $t\leq\frac{1}{100}\frac{|x|}{\lambda}(\lambda+|n|)$. Then
\begin{equation}\label{3-3-2}
|\partial_{\xi}\phi_{\lambda,n}(\xi;~t,x)|\geq\frac{|x|}{t}
-\frac{\lambda |\xi|}{\sqrt{1+|\lambda \xi|^{2}+n^{2}}}\geq\frac{\lambda}{\lambda+|n|}.
\end{equation}

Using the method of non-stationary phase (Lemma 2.34 in \cite{NS},\ Page 65), it derives from \eqref{7-6}
and \eqref{3-3-1}-\eqref{3-3-2} that
\begin{equation*}
\Big|\Phi_{k,n}^{+}(t,x)\Big|
\lesssim \lambda^{2}\Big(\lambda t\frac{\lambda}
{\lambda+|n|}\Big)^{-1}.
\end{equation*}
Therefore, \eqref{7-1} is proved.\end{proof}

Based on Lemma \ref{HC-1} and Lemma \ref{HC-2}, we have the following conclusion:
\begin{lem}\label{HC-3} {\it
For $t>0, n\in \mathbb{Z}$ and $k\in\mathbb{Z}\cap [-1, +\infty)$, then
\begin{equation}\label{7-7}
\Big\|\int_{\mathbb{R}^{2}}e^{ix\cdot\xi}e^{\pm it\sqrt{1+|\xi|^{2}+n^{2}}}
\psi_{[-1,k]}(\xi)d\xi\Big\|_{L^{\infty}(\mathbb{R}^{2})}\lesssim2^{2k}(1+|n|)\cdot t^{-1}.
\end{equation}}
\end{lem}
\begin{proof}
We denote
\begin{equation*}
\int_{\mathbb{R}^{2}}e^{ix\cdot\xi}e^{\pm it\sqrt{1+|\xi|^{2}+n^{2}}}
\psi_{[-1,k]}(\xi)d\xi=H_3+H_4,
\end{equation*}
where
\begin{equation*}
\ds H_3=\int_{\mathbb{R}^{2}}e^{ix\cdot\xi}e^{\pm it\sqrt{1+|\xi|^{2}+n^{2}}}
\psi_{-1}(\xi)d\xi\ \text{and~}\ \ds H_4
=\sum\limits_{l=0}^{k}\int_{\mathbb{R}^{2}}e^{ix\cdot\xi}e^{\pm it\sqrt{1+|\xi|^{2}+n^{2}}}
\varphi_{l}(\xi)d\xi.
\end{equation*}
Applying Lemma \ref{HC-1} and Lemma \ref{HC-2} to $H_3$ and $H_4$ respectively yields \eqref{7-7}.\end{proof}

\begin{lem}{\bf (Main Dispersive Estimates)}\label{HC-4} {\it
For $t\in\mathbb{R}, k, l\in \mathbb{Z}\cap [-1, +\infty)$, and $g\in L^{1}(\mathbb{R}^{2}\times\mathbb{T})$, one has
\begin{equation}\label{7-8}
\Big\|R_{[-1,k]}e^{it\Lambda}S_{l}g\Big\|_{L^{\infty}(\mathbb{R}^{2}\times\mathbb{T})}
\lesssim 2^{2(k+l)}(1+|t|)^{-1}\|g\|_{L^{1}(\mathbb{R}^{2}\times\mathbb{T})}.
\end{equation}}
\end{lem}
\begin{proof} Note that
\begin{equation}\label{3-5-0}
(R_{[-1,k]}e^{it\Lambda}S_{l}g)(x,y):=
\sum_{n\in\mathbb{Z}}R_{[-1,k]}e^{it\sqrt{1-\Delta_{x}+n^{2}}}g_{n}(x)\psi_{l}(n)e^{iny}.
\end{equation}

Thus, when $|t|<1$,
\begin{equation}\label{3-5-1}\begin{aligned}
\Big\|R_{[-1,k]}e^{it\Lambda}S_{l}g\Big\|_{L^{\infty}(\mathbb{R}^{2}\times\mathbb{T})}
\lesssim& \sum\limits_{n\in\mathbb{Z}}\Big\|\psi_{[-1,k]}(\xi)e^{it\sqrt{1+|\xi|^{2}+n^{2}}}\widehat{g_{n}}(\xi)\psi_{l}(n)\Big\|_{L^{1}_{\xi}}\\
\lesssim& 2^{2k+l}\|g\|_{L^{1}(\mathbb{R}^{2}\times\mathbb{T})}.
\end{aligned}
\end{equation}

When $|t|\geq 1$, by Young's inequality, we have
\begin{equation}\label{3-5-2}
\begin{aligned}
&\Big\|\sum_{n\in\mathbb{Z}}R_{[-1,k]}e^{it\sqrt{1-\Delta_{x}+n^{2}}}g_{n}(x)\psi_{l}(n)e^{iny}
\Big\|_{L^{\infty}(\mathbb{R}^{2}\times\mathbb{T})}\\
\lesssim& \sum_{n\in\mathbb{Z}}\Big\|R_{[-1,k]}e^{it\sqrt{1-\Delta_{x}+n^{2}}}g_{n}(x)\psi_{l}(n)e^{iny}
\Big\|_{L^{\infty}(\mathbb{R}^{2}\times\mathbb{T})}\\
\lesssim& \sum_{n\in\mathbb{Z}}\Big\|R_{[-1,k]}e^{it\sqrt{1-\Delta_{x}+n^{2}}}
g_{n}(x)\psi_{l}(n)\Big\|_{L^{\infty}(\mathbb{R}^{2})}\\
\lesssim& \sum_{|n|\leq 2^{l+1}}\Big\|\int_{\mathbb{R}^{2}}e^{ix\cdot\xi}e^{it\sqrt{1+|\xi|^{2}+n^{2}}}
\psi_{[-1,k]}(\xi)\widehat{g_{n}}(\xi)d\xi\Big\|_{L^{\infty}(\mathbb{R}^{2})}\\
\lesssim&\sum_{|n|\leq 2^{l+1}}\Big\|\mathcal{F}^{-1}_{\xi}(e^{it\sqrt{1+|\xi|^{2}+n^{2}}}
\psi_{[-1,k]}(\xi))(x)\Big\|_{L^{\infty}(\mathbb{R}^{2})}\|g_{n}(x)\|_{L^{1}(\mathbb{R}^{2})}\\
\lesssim& 2^{2(k+l)}(1+|t|)^{-1}\|g\|_{L^{1}(\mathbb{R}^{2}\times\mathbb{T})},
\end{aligned}
\end{equation}
where the last inequality comes from Lemma \ref{HC-3}. Then \eqref{7-8} is obtained
from \eqref{3-5-0}-\eqref{3-5-2}.
\end{proof}

\begin{lem}\label{HC-5} {\it For $\|\cdot\|_{\mathcal{H}^{N_0}}:=\|\cdot\|_{H^{N_0}(\mathbb{R}^2\times\mathbb{T})}
+\|\cdot\|_{Z}$ and $\kappa>0$ is a constant, when
\begin{equation}\label{3-6-0}
\sup\limits_{t\in [0,T_{0}]}(1+t)^{-\kappa}\|U(t)\|_{H^{N}(\mathbb{R}^{2}\times \mathbb{T})}+
\sup\limits_{t\in [0,T_{0}]}\|e^{it\Lambda}U(t)\|_{\mathcal{H}^{N_{0}}}~\leq~\varepsilon_{0},
\end{equation}
then for any $t\in [0,T_{0}]$, one has
\begin{subequations}\label{3-6-1}\begin{align}
&\|R_{k}S_{l}U_{\pm}(t)\|_{L^{\infty}(\mathbb{R}^{2}\times\mathbb{T})}
\lesssim \varepsilon_{0}(1+t)^{-1}2^{-7(k+l)},\label{7-12AA}\\[2mm]
&\|\mathcal{F}_{x,y}(R_{k}S_{l}U_{\pm}(t))(\xi,n)\|_{L^{\infty}_{\xi,n}}
=\|\mathcal{F}_{x,y}(R_{k}S_{l}V_{\pm}(t))(\xi,n)\|_{L^{\infty}_{\xi,n}}
\lesssim \varepsilon_{0}2^{-9(k+l)},\label{7-13AA}\\[2mm]
&\|R_{k}S_{l}U_{\pm}(t)\|_{L^{2}(\mathbb{R}^{2}\times\mathbb{T})}
+\|R_{k}S_{l}V_{\pm}(t)\|_{L^{2}(\mathbb{R}^{2}\times\mathbb{T})}
\lesssim \varepsilon_{0}\mathrm{min}
\{2^{-N_{0}(k+l)/2},~(1+t)^{\kappa}2^{-N(k+l)/2}\}.\label{7-14AA}
\end{align}
\end{subequations}}
\end{lem}

\begin{proof}
By the definition of $V_{\pm}(t)$, Lemma \ref{HC-4}, \eqref{2-1-0} and the assumption \eqref{3-6-0}, we arrive at
\begin{equation*}
\begin{aligned}
&\|R_{k}S_{l}U_{\pm}(t)\|_{L^{\infty}(\mathbb{R}^{2}\times\mathbb{T})}\\
=&\|e^{\mp it\Lambda}R_{k}S_{l}(V_{\pm}(t))\|_{L^{\infty}(\mathbb{R}^{2}\times\mathbb{T})}\\
=&\|R_{[-1,k+2]}e^{\mp it\Lambda}S_{[l-2,l+2]}(R_{k}S_{l}V_{\pm}(t))\|_{L^{\infty}(\mathbb{R}^{2}\times\mathbb{T})}\\
\lesssim& 2^{2(k+l)}(1+t)^{-1}\|R_{k}S_{l}V_{\pm}(t)\|_{L^{1}(\mathbb{R}^{2}\times\mathbb{T})}\\
\lesssim& (1+t)^{-1} 2^{-7(k+l)}\|e^{it\Lambda}U(t)\|_{Z}\lesssim \varepsilon_{0}(1+t)^{-1}2^{-7(k+l)}.
\end{aligned}
\end{equation*}
This yields \eqref{7-12AA}.

In the similar way, \eqref{7-13AA} can be derived as follows
\begin{equation*}
\begin{aligned}
\|\mathcal{F}_{x,y}(R_{k}S_{l}U_{\pm}(t))(\xi,n)\|_{L^{\infty}_{\xi,n}}
=&\|\mathcal{F}_{x,y}(R_{k}S_{l}V_{\pm}(t))(\xi,n)\|_{L^{\infty}_{\xi,n}}\\
\lesssim& \|R_{k}S_{l}V_{\pm}(t)\|_{L^{1}(\mathbb{R}^{2}\times\mathbb{T})}
\lesssim 2^{-9(k+l)}\|e^{i\Lambda}U(t)\|_{Z}\lesssim \varepsilon_{0}2^{-9(k+l)}.
\end{aligned}
\end{equation*}

For \eqref{7-14AA}, together with \eqref{3-6-0}, one has
\begin{equation*}
\begin{aligned}
&\|R_{k}S_{l}V_{\pm}(t)\|_{L^{2}(\mathbb{R}^{2}\times\mathbb{T})}\\=
&\|R_{k}S_{l}U_{\pm}(t)\|_{L^{2}(\mathbb{R}^{2}\times\mathbb{T})}\\
=&\|\widehat{U_{\pm;n}}(t,\xi)\psi_{k}(\xi)\psi_{l}(n)\|_{L^{2}_{\xi}l^{2}_{n}}\\
\lesssim& 2^{-N_{0}(k+l)/2}\|(1+|\xi|)^{N_{0}/2}(1+|n|)^{N_{0}/2}\widehat{U_{\pm;n}}(t,\xi)\psi_{k}(\xi)\psi_{l}(n)\|_{L^{2}_{\xi}l^{2}_{n}}\\
\lesssim& 2^{-N_{0}(k+l)/2}\|(1+|\xi|^{2}+n^{2})^{N_{0}/2}\widehat{U_{\pm;n}}(t,\xi)\psi_{k}(\xi)\psi_{l}(n)\|_{L^{2}_{\xi}l^{2}_{n}}\\
\lesssim& 2^{-N_{0}(k+l)/2}\|U(t)\|_{H^{N_{0}}(\mathbb{R}^{2}\times \mathbb{T})}\lesssim \varepsilon_0 2^{-N_0(k+l)/2}
\end{aligned}
\end{equation*}
and
\begin{equation}
\begin{aligned}
\|R_{k}S_{l}U_{\pm}(t)\|_{L^{2}(\mathbb{R}^{2}\times\mathbb{T})}&\lesssim 2^{-N(k+l)/2}\|U(t)\|_{H^{N}(\mathbb{R}^{2}\times \mathbb{T})}
\lesssim \varepsilon_{0}(1+t)^{\kappa}2^{-N(k+l)/2}.\nonumber
\end{aligned}
\end{equation}
Then  \eqref{7-14AA} is shown by collecting the two estimates above.
\end{proof}

\begin{lem}{\bf (Finite band property)}\label{HC-6}
{\it With $R_{k},S_{l}$ defined in Section 2, there hold  that for $1\leq p\leq\infty$,
\begin{equation}\label{3-7-0}
\|\partial_{x}R_{k}S_{l}f\|_{L^{p}(\mathbb{R}^{2}\times \mathbb{T})}\lesssim
2^{k}\|R_{k}S_{l}f\|_{L^{p}(\mathbb{R}^{2}\times \mathbb{T})},~~
\|\partial_{y}R_{k}S_{l}f\|_{L^{p}(\mathbb{R}^{2}\times \mathbb{T})}\lesssim
2^{l}\|R_{k}S_{l}f\|_{L^{p}(\mathbb{R}^{2}\times \mathbb{T})}.
\end{equation}}
\end{lem}

\begin{proof} At first, we have
\begin{equation}\label{3-7-1}
\begin{aligned}
\|\partial_{x}R_{k}S_{l}f\|_{L^{p}(\mathbb{R}^{2}\times \mathbb{T})}
&=\Big(\int_{\mathbb{T}}\Big\|\partial_{x}R_{k}\Big(\sum\limits_{n\in \mathbb{Z}}f_{n}(x)\psi_{l}(n)e^{iny}\Big)\Big\|^{p}_{L^{p}(\mathbb{R}^{2})}dy\Big)^{1/p}\\
&\lesssim2^{k}\Big(\int_{\mathbb{T}}\Big\|R_{k}\Big(\sum\limits_{n\in \mathbb{Z}}f_{n}(x)\psi_{l}(n)e^{iny}\Big)\Big\|^{p}_{L^{p}(\mathbb{R}^{2})}dy\Big)^{1/p}\\
&\lesssim2^{k}\|R_{k}S_{l}f\|_{L^{p}(\mathbb{R}^{2}\times \mathbb{T})}.
\end{aligned}
\end{equation}

In addition, set $g_l(z)=\sum\limits_{n\in \mathbb{Z}}
\frac{n}{2^{l}}\cdot\psi_{[l-2, l+2]}(n)e^{inz}$, then
\begin{equation}\label{3-7-2}
\begin{aligned}
\|\partial_{y}R_{k}S_{l}f\|_{L^{p}(\mathbb{R}^{2}\times \mathbb{T})}&=\Big(\int_{\mathbb{R}^{2}}\Big\|\sum\limits_{n\in \mathbb{Z}}(R_{k}f_{n}(x)n\psi_{l}(n))e^{iny}\Big\|^{p}_{L^{p}(\mathbb{T})}dx\Big)^{1/p}\\
&=2^{l}\Big(\int_{\mathbb{R}^{2}}\Big\|\int_{\mathbb{T}}(R_{k}S_{l}f)(x,y-z)g_l(z)dz\Big\|^{p}_{L^{p}(\mathbb{T})}dx\Big)^{1/p}\\
&\lesssim 2^{l}\|R_{k}S_{l}f\|_{L^{p}(\mathbb{R}^{2}\times \mathbb{T})}\|g_l\|_{L^{1}(\mathbb{T})}.
\end{aligned}
\end{equation}

In the following estimate of $\|g_l\|_{L^1(\mathbb{T})}$, $l\geq2$ is assumed since the case $l<2$
can be treated similarly. By Poisson summation formula, we have
\begin{equation}\label{3-7-3}
\begin{aligned}
g_l(z)&=\sum\limits_{n\in \mathbb{Z}}\int_{-\infty}^{\infty}e^{i(z+2n\pi)\cdot x} \varphi_{[l-2, l+2]}(x)\cdot\frac{x}{2^{l}}dx\\
&=2^{l}\sum\limits_{n\in \mathbb{Z}}\int_{-\infty}^{\infty}e^{i(z+2n\pi)\cdot2^{l} \xi} \varphi_{[-2, 2]}(\xi)\cdot \xi d\xi.
\end{aligned}
\end{equation}
It follows  from the method of non-stationary phase that
\begin{equation}\label{3-7-4}
\Big|\int_{-\infty}^{\infty}e^{i(z+2n\pi)\cdot2^{l} \xi} \varphi_{[-2, 2]}(\xi)\cdot \xi d\xi\Big|\lesssim(1+2^{l}|z+2n\pi|)^{-6}.
\end{equation}
Then
\begin{equation}\label{3-7-5}
\begin{aligned}
\|g_l\|_{L^{1}(\mathbb{T})}&\lesssim\sum\limits_{n\in \mathbb{Z}}\int_{0}^{2\pi}2^{l}
\Big|\int_{-\infty}^{\infty}e^{i(z+2n\pi)\cdot2^{l} \xi} \varphi_{[-2, 2]}(\xi)\cdot \xi d\xi\Big|dz
\\&\lesssim\sum\limits_{n\in \mathbb{Z}}\int_{0}^{2\pi}2^{l}(1+2^{l}|z+2n\pi|)^{-6}dz\lesssim1.
\end{aligned}
\end{equation}
Therefore, the estimates in \eqref{3-7-0} come from \eqref{3-7-1}, \eqref{3-7-2} and \eqref{3-7-5}.
\end{proof}

\begin{lem}{\bf (Bernstein-type inequality)}\label{HC-7}
{\it With $R_{k}, S_{l}$ defined in Section 2, one has
\begin{equation}\label{7-6-123aAA}
\|R_{k}S_{l}f\|_{L^{r}(\mathbb{R}^{2}\times \mathbb{T})}\lesssim
2^{(k+l/2)(1-2/r)}\|f\|_{L^{2}(\mathbb{R}^{2}\times \mathbb{T})}~~ (2\leq r\leq +\infty).
\end{equation}}
\end{lem}

\begin{proof} Note that
\begin{equation}\label{7-7-123aAA}
\begin{aligned}
\|R_{k}S_{l}f\|_{L^{r}(\mathbb{R}^{2}\times \mathbb{T})}&=\Big(\int_{\mathbb{R}^{2}}\Big\|\sum\limits_{n\in \mathbb{Z}}(R_{k}f_{n}(x)\psi_{l}(n))e^{iny}\Big\|^{r}_{L^{r}(\mathbb{T})}dx\Big)^{1/r}\\&
=\Big(\int_{\mathbb{R}^{2}}\Big\|\int_{\mathbb{T}}\int_{\mathbb{R}^{2}}f(x-x_{1},y-y_{1})w_k(x_{1})h_l(y_{1})dx_{1}dy_{1}
\Big\|^{r}_{L^{r}(\mathbb{T})}dx\Big)^{1/r}
\\&\lesssim \|f\|_{L^{2}(\mathbb{R}^{2}\times \mathbb{T})}\|w_k(x)h_l(y)\|_{L^{q}(\mathbb{R}^{2}\times \mathbb{T})}
\\&= \|f\|_{L^{2}(\mathbb{R}^{2}\times \mathbb{T})}\|w_k\|_{L^{q}(\mathbb{R}^{2})}\|h_l\|_{L^{q}(\mathbb{T})},
\end{aligned}
\end{equation}
where $1/r+1=1/2+1/q$ and
\begin{equation*}
\begin{aligned}
w_k(x)=\int_{\mathbb{R}^{2}}e^{ix\cdot\xi}\psi_{k}(\xi)d\xi,\
h_l(y)=\sum\limits_{n\in \mathbb{Z}}\psi_{l}(n)e^{iny}.
\end{aligned}
\end{equation*}
Direct calculation shows
\begin{equation}\label{7-8-123aAA}
\|w_k\|_{L^{q}(\mathbb{R}^{2})}\lesssim 2^{k(1-2/r)}.
\end{equation}

In the following estimate of $\|h_l\|_{L^{q}(\mathbb{T})}$, $l\geq 0$ is assumed
since the case of $l=-1$ can be done analogously. By the
similarly way to deal with $g_l(z)$ in \eqref{3-7-3}-\eqref{3-7-5},  we have
\begin{equation*}
\begin{aligned}
h_l(y)=&\sum\limits_{n\in \mathbb{Z}}\int_{-\infty}^{\infty}e^{i(y+2n\pi)\cdot x}\psi_{l}(x)dx\\
=&\sum\limits_{n\in \mathbb{Z}}\int_{-\infty}^{\infty}e^{i(y+2n\pi)\cdot x}\varphi_{0}(x/2^{l})dx=2^{l}\sum\limits_{n\in \mathbb{Z}}\int_{-\infty}^{\infty}e^{i(y+2n\pi)\cdot2^{l} x}\varphi_{0}(x)dx
\end{aligned}
\end{equation*}
and
\begin{equation*}
\Big|\int_{-\infty}^{\infty}e^{i(y+2n\pi)\cdot2^{l} x}\varphi_{0}(x)dx\Big|\lesssim (1+2^{l}|y+2n\pi|)^{-6}.
\end{equation*}
Therefore,
\begin{equation}\label{7-9-123aAA}
\begin{aligned}
\|h_l\|_{L^{q}(\mathbb{T})}\lesssim& 2^{l}\sum\limits_{n\in \mathbb{Z}}\Big(\int_{0}^{2\pi}(1+2^{l}|y+2n\pi|)^{-6q}dy\Big)^{1/q}\\
\lesssim& 2^{l(1-1/q)}\sum\limits_{n\in \mathbb{Z}}\Big(\int_{0}^{2^{l}\cdot2\pi}(1+|w+2^{l}\cdot2n\pi|)^{-6q}dw\Big)^{1/q}\\
\lesssim& 2^{l(1-1/q)}\sum\limits_{n\in \mathbb{Z},~|n|\leq3}\Big(\int_{0}^{\infty}(1+|w|)^{-6q}dw\Big)^{1/q}
\\
&+2^{l(1-1/q)}\sum\limits_{n\in \mathbb{Z},~|n|\geq 4}\Big(\int_{0}^{2^{l}\cdot2\pi}(2^{l}|n|\pi)^{-6q}dw\Big)^{1/q}\\
\lesssim& 2^{l(1-1/q)}+2^{l(1-1/q)}\sum\limits_{n\in \mathbb{Z},~|n|\geq 4}\Big((2^{l}|n|\pi)^{-5q}\Big)^{1/q}\lesssim 2^{l(1-2/r)/2}.
\end{aligned}
\end{equation}
Substituting \eqref{7-8-123aAA} and \eqref{7-9-123aAA} into \eqref{7-7-123aAA} yields \eqref{7-6-123aAA}.\end{proof}

\begin{lem}\label{HC-10} {\it
For $1\leq p\leq \infty$, we have
\begin{equation}\label{7-24AAaa}
\Big\|R_{k}S_{l}\Big(\sum\limits_{n\in \mathbb{Z}}\frac{\partial_{x}}{\sqrt{1-\Delta_{x}+n^{2}}}f_{n}(x)e^{iny}\Big)
\Big\|_{L^{p}(\mathbb{R}^{2}\times\mathbb{T})}\lesssim \|f\|_{L^{p}(\mathbb{R}^{2}\times\mathbb{T})}
\end{equation}
and
\begin{equation}\label{7-24AAaa-1}
\Big\|R_{k}S_{l}\Big(\sum\limits_{n\in \mathbb{Z}}\frac{n}{\sqrt{1-\Delta_{x}+n^{2}}}f_{n}(x)e^{iny}\Big)
\Big\|_{L^{p}(\mathbb{R}^{2}\times\mathbb{T})}\lesssim \|f\|_{L^{p}(\mathbb{R}^{2}\times\mathbb{T})}.
\end{equation}}
\end{lem}

\begin{proof}
We only prove \eqref{7-24AAaa} since the estimate \eqref{7-24AAaa-1} can be treated by a similar way. With the
definitions of $R_k$ and $S_l$, one has
\begin{equation*}\label{7-25Aaa}
\begin{aligned}
&R_{k}S_{l}\Big(\sum\limits_{n\in \mathbb{Z}}\frac{\partial_{x}}{\sqrt{1-\Delta_{x}+n^{2}}}f_{n}(x)e^{iny}\Big)\\
=&\int_{\mathbb{R}^{2}}e^{ix\cdot \xi}\psi_{k}(\xi)\Big(\sum\limits_{n\in \mathbb{Z}}
\frac{i\xi}{\sqrt{1+|\xi|^{2}+n^{2}}}\widehat{f_{n}}(\xi)\psi_{l}(n)e^{iny}\Big)d\xi\\
=&\int_{\mathbb{R}^{2}}e^{ix\cdot \xi}\psi_{k}(\xi)\Big[\int_{\mathbb{T}}\mathcal{F}_{x}(f)(\xi,y-y_{1})
\Big(\sum\limits_{n\in \mathbb{Z}}\frac{i\xi}{\sqrt{1+|\xi|^{2}+n^{2}}}\psi_{l}(n)e^{iny_{1}}\Big)dy_{1}\Big]d\xi\\
=&\int_{\mathbb{T}}\Big[\int_{\mathbb{R}^{2}}e^{ix\cdot \xi}\psi_{k}(\xi)\mathcal{F}_{x}(f)(\xi,y-y_{1})
\Big(\sum\limits_{n\in \mathbb{Z}}\frac{i\xi}{\sqrt{1+|\xi|^{2}+n^{2}}}\psi_{l}(n)e^{iny_{1}}\Big)d\xi\Big]dy_{1}\\
=&\int_{\mathbb{T}}\Big[\int_{\mathbb{R}^{2}}f(x-x_{1},y-y_{1})\Big(\int_{\mathbb{R}^{2}}e^{ix_{1}\cdot \xi}\psi_{k}(\xi)
\sum\limits_{n\in \mathbb{Z}}\frac{i\xi}{\sqrt{1+|\xi|^{2}+n^{2}}}\psi_{l}(n)e^{iny_{1}}d\xi\Big)dx_{1}\Big]dy_{1}.
\end{aligned}
\end{equation*}
Then
\begin{equation}\label{7-26aAAaa}
\Big\|R_{k}S_{l}\Big(\sum\limits_{n\in \mathbb{Z}}\frac{\partial_{x}}{\sqrt{1-\Delta_{x}+n^{2}}}f_{n}(x)e^{iny}\Big)
\Big\|_{L^{p}(\mathbb{R}^{2}\times\mathbb{T})}\lesssim \|f\|_{L^{p}(\mathbb{R}^{2}\times\mathbb{T})}
\|h_{kl}\|_{L^{1}(\mathbb{R}^{2}\times\mathbb{T})},
\end{equation}
where
\begin{equation*}\label{7-27aAAaa}
h_{kl}(x,y):=\int_{\mathbb{R}^{2}}e^{ix\cdot \xi}\psi_{k}(\xi)
\Big(\sum\limits_{n\in \mathbb{Z}}\frac{\xi}{\sqrt{1+|\xi|^{2}+n^{2}}}
\psi_{l}(n)e^{iny}\Big)d\xi.
\end{equation*}

On the other hand, by Poisson summation formula, we have
\begin{equation}\label{3-9-1}
\begin{aligned}
h_{kl}(x,y)&=\int_{\mathbb{R}^{2}}e^{ix\cdot \xi}\psi_{k}(\xi)
\Big(\sum\limits_{n\in \mathbb{Z}}\int_{-\infty}^{\infty}e^{i(y+2n\pi)\cdot z}
\frac{\xi}{\sqrt{1+|\xi|^{2}+z^{2}}}\psi_{l}(z)dz\Big)d\xi
\\&=\sum\limits_{n\in \mathbb{Z}}\int_{\mathbb{R}^{2}}
\int_{-\infty}^{\infty}e^{ix\cdot \xi}e^{i(y+2n\pi)\cdot z}
\frac{\xi}{\sqrt{1+|\xi|^{2}+z^{2}}}
\psi_{k}(\xi)\psi_{l}(z)dzd\xi.
\end{aligned}
\end{equation}
It follows from the method of non-stationary phase that
\begin{equation}\label{YH-1}
\begin{aligned}
&\Big|\int_{\mathbb{R}^{2}}
\int_{-\infty}^{\infty}e^{ix\cdot \xi}e^{i(y+2n\pi)\cdot z}\frac{\xi}{\sqrt{1+|\xi|^{2}+z^{2}}}
\psi_{k}(\xi)\psi_{l}(z)dzd\xi\Big|\\
\lesssim& 2^{2k}(1+2^{k}|x|)^{-4}2^{l}(1+2^{l}|y+2n\pi|)^{-6}.
\end{aligned}
\end{equation}
Substituting \eqref{YH-1} into \eqref{3-9-1} shows
\begin{equation}\label{7-27aAAaabb}
\begin{aligned}
\|h_{kl}(x, y)\|_{L^{1}(\mathbb{R}^{2}\times\mathbb{T})}
&\lesssim \int_{\mathbb{R}^{2}}2^{2k}(1+2^{k}|x|)^{-4}dx
\cdot\sum\limits_{n\in \mathbb{Z}}\int_{0}^{2\pi}2^{l}(1+2^{l}|y+2n\pi|)^{-6}dy
\lesssim1.
\end{aligned}
\end{equation}
Therefore, \eqref{7-24AAaa} comes from \eqref{7-26aAAaa} and \eqref{7-27aAAaabb}.
And  the proof of Lemma \ref{HC-10} is completed.
\end{proof}

\begin{lem}\label{HC-11} {\it
With $\Lambda_{n}(\xi)=\sqrt{1+|\xi|^{2}+n^{2}}$,~$\xi\in \mathbb{R}^{2},n\in \mathbb{Z}$, we have
\begin{equation*}
\frac{1}{|\Lambda_{n}(\xi)\pm\Lambda_{n-m'}(\xi-\eta)\pm\Lambda_{m'}(\eta)|}\leq 2\sqrt{1+\mathrm{min}(a,b,c)},
\end{equation*}
where $a=|\xi|^{2}+n^{2},~b=|\xi-\eta|^{2}+(n-m')^{2},~c=|\eta|^{2}+{m'}^{2}$.}
\end{lem}

\begin{proof}
The proof is similar to that of Lemma 5.1 in \cite{IP2}, we omit the details here.
\end{proof}

\begin{lem}\label{HC-12} {\it
Under the notations in \eqref{2-10} and \eqref{2-11}, we have
\begin{equation}\label{7-20AA}
|\nabla_{\xi,\eta}\Phi^{n, m}_{\mu,\nu}(\xi,\eta)|\lesssim |\Phi^{n, m}_{\mu,\nu}(\xi,\eta)|
\end{equation}
and
\begin{equation}\label{7-16BAB}
\Big|\partial_{\xi}^{k}\partial_{\eta}^{l}\Big(\frac{1}{\Phi^{n, m}_{\mu,\nu}(\xi,\eta)}\Big)\Big|\lesssim
\Big|\frac{1}{\Phi^{n, m}_{\mu,\nu}(\xi,\eta)}\Big|~~ (k+l\geq1).
\end{equation}}
\end{lem}

\begin{proof}
The estimate (5.4) of Lemma 5.1 in \cite{IP2} can be easily extended into 3D case by the same argument, namely,
\begin{equation}\label{7-21AA}
|\nabla_{\zeta,\upsilon}\Phi_{\mu, \nu}(\zeta,\upsilon)|\lesssim |\Phi_{\mu, \nu}(\zeta,\upsilon)|\ (\zeta,\upsilon\in \mathbb{R}^{3}),
\end{equation}
where $\Phi_{\mu, \nu}(\zeta,\upsilon)=\Lambda(\zeta)-\mu\Lambda(\zeta-\upsilon)-\nu\Lambda(\upsilon)$~and $\Lambda(\zeta)=\sqrt{1+|\zeta|^{2}}$.
Then \eqref{7-20AA} comes from \eqref{7-21AA} immediately by the choice $\Phi_{\mu, \nu}(\xi, n, \eta, m)=\Phi_{\mu,\nu}^{n, m}(\xi,\eta)$.

The estimate \eqref{7-16BAB} follows directly from the induction argument and the following inequality
\begin{equation}\label{7-18BAB}
\Big|\partial_{\xi}^{k}\partial_{\eta}^{l}\Phi_{\mu,\nu}^{n, m}(\xi,\eta)\Big|\lesssim \mathrm{min}\Big(1,
\Big|\Phi_{\mu,\nu}^{n, m}(\xi,\eta)\Big|\Big)\ (k+l\geq1).
\end{equation}

Next we prove \eqref{7-18BAB}: If $k+l=1$, \eqref{7-18BAB} comes from \eqref{7-20AA} directly.
If $(k,l)=(1,1)$, it follows from direct calculation and Lemma \ref{HC-11} that
\begin{equation}
\begin{aligned}
\Big|\partial_{\xi}\partial_{\eta}\Phi_{\mu,\nu}^{n, m}(\xi,\eta)\Big|\lesssim 1/\Lambda_{n-m}(\xi-\eta) \lesssim 1/\sqrt{1+\mathrm{min}(a,b,c)} \lesssim \Big|\Phi_{\mu,\nu}^{n, m}(\xi,\eta)\Big|,\nonumber
\end{aligned}
\end{equation}
where the notations $a, b~\mathrm{and}~c$~are defined as in Lemma \ref{HC-11}.
When $k+l\geq 2$,  \eqref{7-18BAB} can be analogously proved. Thus the
proof of Lemma \ref{HC-12} is completed.
\end{proof}

\section{Two crucial propositions and the proof of Theorem 1.1}

At first, we state two crucial propositions which will be utilized to prove Theorem 1.1.

\begin{prop}\label{prop4-1}{\bf (Continuity and boundedness of $Z$-norm)} {\it Assume
that $u\in C([0,T_{0}], H^{N+1}(\mathbb{R}^{2}\times\mathbb{T}))
\cap C^1([0,T_{0}], H^{N}(\mathbb{R}^{2}\times\mathbb{T}))$ is a solution of problem \eqref{0-1}.
Define $U$ as in \eqref{2-8} with the property of $U_{0}=U(0)\in Z$.
Then one has
\begin{equation}\label{3-2}
\sup\limits_{t\in [0,T_{0}]}\|e^{it\Lambda}U(t)\|_{Z}~\leq~C\Big(T_{0},\|U_{0}\|_{Z},\sup\limits_{t\in [0,T_{0}]}\|U(t)\|_{H^{N}(\mathbb{R}^{2}\times\mathbb{T})}\Big)
\end{equation}
and the mapping $t\mapsto e^{it\Lambda}U(t)$ is continuous from $[0,T_{0}]$~to~$Z$.}
\end{prop}

\begin{prop}\label{prop4-2} {\bf (Bootstrap estimate)} {\it Under the assumptions in Proposition \ref{prop4-1}, for $N_{0}:=50$ and $0< T_{0}\leq e^{c_{0}/\epsilon_{0}^2}$,
if
\begin{equation}\label{3-3}
\sup\limits_{t\in [0,T_{0}]}(1+t)^{-\kappa}\|U(t)\|_{H^{N}(\mathbb{R}^{2}\times \mathbb{T})}+
\sup\limits_{t\in [0,T_{0}]}\|e^{it\Lambda}U(t)\|_{\mathcal{H}^{N_{0}}}\leq\varepsilon_{0}<1,
\end{equation}
then
\begin{equation}\label{3-4}
\sup\limits_{t\in [0,T_{0}]}\|e^{it\Lambda}U(t)-U(0)\|_{\mathcal{H}^{N_{0}}}\lesssim \varepsilon_{0}^{2}+c_{0}\varepsilon_{0},
\end{equation}
where $0<\kappa\ll1$, and $\|\cdot\|_{\mathcal{H}^{N_{0}}}:=\|\cdot\|_{H^{N_{0}}(\mathbb{R}^{2}\times \mathbb{T})}+\|\cdot\|_{Z}$.\\
}
\end{prop}

The proofs of Proposition 4.1 and Proposition 4.2 will be postponed to Section~5 and Section~6, respectively.
In addition, by the standard local existence theory and energy estimate for the second order quasilinear wave
equations (see Chapter 6 of \cite{H}), we have

\begin{prop} {\bf (Local existence)}\label{prop4-3} {\it  Under the assumptions in Theorem \ref{mainthm},
problem \eqref{0-1} has a unique solution $u\in C([0, 1], H^{N+1}(\mathbb{R}^{2}\times \mathbb{T}))\cap C^{1}([0, 1], H^{N}(\mathbb{R}^{2}\times \mathbb{T}))$. Moreover
\begin{equation}\label{4-2}
\sup\limits_{t\in [0,1]}\|u(t)\|_{H^{N+1}(\mathbb{R}^{2}\times \mathbb{T})}+\sup\limits_{t\in [0,1]}\|\dot{u}(t)\|_{H^{N}(\mathbb{R}^{2}\times \mathbb{T})}\lesssim\|u_{0}\|_{H^{N+1}(\mathbb{R}^{2}\times \mathbb{T})}+\|u_{1}\|_{H^{N}(\mathbb{R}^{2}\times \mathbb{T})}
\end{equation}
and
\begin{equation}\label{4-3}
\begin{aligned}
\mathcal{E}_{N+1}(t')-\mathcal{E}_{N+1}(t)\lesssim\int_{t}^{t'}\mathcal{E}_{N+1}(s)
\cdot\Big[\sum\limits_{|\rho|\leq2}\|\partial_{x,y}^{\rho}u(s)\|_{L^{\infty}(\mathbb{R}^{2}\times \mathbb{T})}+\sum\limits_{|\rho|\leq1}\|\partial_{x,y}^{\rho}\dot{u}(s)\|_{L^{\infty}(\mathbb{R}^{2}\times \mathbb{T})}\Big]ds
\end{aligned}
\end{equation}
for any $t\leq t'\in[0,1]$,~where the definition of $\mathcal{E}_{N+1}(t)$ is given in \eqref{0-666}.}
\end{prop}

\vskip 0.1 true cm

Based on the preparations above, we now prove Theorem 1.1.

\vskip 0.1 true cm

\noindent {\bf \textit{Proof of Theorem 1.1.}} By Proposition 4.3 and Proposition 4.1, there is $T_{0}>0$ such that the unique
solution $u\in C([0,T_{0}],H^{N+1}(\mathbb{R}^{2}\times \mathbb{T}))\cap C^{1}([0,T_{0}],H^{N}(\mathbb{R}^{2}\times \mathbb{T}))$
of problem \eqref{0-1} exists. Moreover, there exists a suitably large constant $C_1 >2$ such that
\begin{equation}\label{4-9}
\sup\limits_{t\in [0,T_{0}]}\|e^{it\Lambda}U(t)\|_{\mathcal{H}^{N_{0}}}~\leq~C_{1}\varepsilon_{0}.
\end{equation}

Under \eqref{4-9}, we next show the following bootstrap estimate
\begin{equation}\label{4-10}
\sup\limits_{t\in [0,T_{0}]}\|e^{it\Lambda}U(t)\|_{\mathcal{H}^{N_{0}}} \leq \f{1}{2}C_1\varepsilon_{0}.
\end{equation}

Note that
\begin{equation*}
\begin{aligned}
U(t)&=\sum_{k\geq-1}\sum_{l\geq-1}R_{k}S_{l}U\\
&=\sum_{k\geq-1}\sum_{l\geq-1}e^{-it\Lambda}R_{k}S_{l}(e^{it\Lambda}U)\\
&=\sum_{k\geq-1}\sum_{l\geq-1}R_{[-1,k+2]}e^{-it\Lambda}S_{[l-2,l+2]}(R_{k}S_{l}(e^{it\Lambda}U)).
\end{aligned}
\end{equation*}
With the definition $u=\frac{i}{2\Lambda}(U-\overline{U})$, by Lemma \ref{HC-6} and Lemma \ref{HC-10}, we arrive at
\begin{equation}\label{4-12ASS-1}
\begin{aligned}
&(1+t)\sup\limits_{|\rho|\leq 2}\|\partial_{x,y}^{\rho}u(t)\|_{L^{\infty}(\mathbb{R}^{2}\times \mathbb{T})}\\
\lesssim& (1+t)\sup\limits_{|\rho|\leq 2}\sum_{k\geq-1}\sum_{l\geq-1}
\|\partial_{x,y}^{\rho}(R_{k}S_{l}u(t))\|_{L^{\infty}(\mathbb{R}^{2}\times \mathbb{T})}\\
\lesssim& (1+t)\sum_{\mbox{\tiny$\begin{array}{c}a,b\in \mathbb{N}\\
a+b\leq 2\end{array}$}}\sum_{k\geq-1}\sum_{l\geq-1}2^{ak+bl}\|R_{k}S_{l}u(t)\|_{L^{\infty}(\mathbb{R}^{2}\times \mathbb{T})}\\
\lesssim& (1+t)\sum_{\mbox{\tiny$\begin{array}{c}a,b\in \mathbb{N}\\
a+b\leq 2\end{array}$}}\sum_{k\geq-1}\sum_{l\geq-1}2^{ak+bl}\|R_{k}S_{l}U(t)\|_{L^{\infty}(\mathbb{R}^{2}\times \mathbb{T})}.
\end{aligned}
\end{equation}

On the other hand, by \eqref{7-8}, \eqref{2-1-0} and \eqref{4-9},  one has that for $t\in [0,T_{0}]$,
\begin{equation}\label{4-12ASS-2}
\begin{aligned}
&\|R_{k}S_{l}U\|_{L^{\infty}(\mathbb{R}^{2}\times\mathbb{T})}\\
=&\|R_{[-1,k+2]}e^{-it\Lambda}S_{[l-2,l+2]}(R_{k}S_{l}(e^{it\Lambda}U))\|_{L^{\infty}(\mathbb{R}^{2}\times\mathbb{T})}\\
\lesssim& (1+t)^{-1}2^{2(k+l)}\Big\|R_{k}S_{l}(e^{it\Lambda}U)\Big\|_{L^{1}(\mathbb{R}^{2}\times\mathbb{T})}\\
\lesssim& (1+t)^{-1}2^{2(k+l)}2^{-9(k+l)}\|e^{it\Lambda}U(t)\|_{Z}\\
\lesssim& C_{1}\varepsilon_{0}(1+t)^{-1}2^{-7(k+l)}.
\end{aligned}
\end{equation}
Combining \eqref{4-12ASS-1} with \eqref{4-12ASS-2} yields
\begin{equation*}
\begin{aligned}
\sup\limits_{t\in [0,T_{0}]}[(1+t)\sup\limits_{|\rho|\leq2}
\|\partial_{x,y}^{\rho}u(t)\|_{L^{\infty}(\mathbb{R}^{2}\times \mathbb{T})}]\lesssim C_{1}\varepsilon_{0}.
\end{aligned}
\end{equation*}
\vskip 0.1 true cm
\noindent Similarly, $\sup\limits_{t\in [0,T_{0}]}[(1+t)\sup\limits_{|\rho|\leq1}\|\partial_{x,y}^{\rho}\dot{u}(t)\|_{L^{\infty}(\mathbb{R}^{2}\times \mathbb{T})}]\lesssim C_{1}\varepsilon_{0}$ holds. Thus, for $t\in [0,T_{0}]$,
\begin{equation}\label{4-12A}
\sum\limits_{|\rho|\leq2}\|\partial_{x,y}^{\rho}u(t)\|_{L^{\infty}(\mathbb{R}^{2}\times \mathbb{T})}+\sum\limits_{|\rho|\leq1}\|\partial_{x,y}^{\rho}\dot{u}(t)\|_{L^{\infty}(\mathbb{R}^{2}\times \mathbb{T})}
\leq~C_{1}C_{2}\varepsilon_{0}\cdot (1+t)^{-1}.
\end{equation}
By \eqref{4-3} and \eqref{4-12A}, we get
\begin{equation*}
\begin{aligned}
\mathcal{E}_{N+1}(t)-\mathcal{E}_{N+1}(0)\leq C_{3}\cdot\int_{0}^{t}\mathcal{E}_{N+1}(s)\cdot C_{1}C_{2}\varepsilon_{0}\cdot (1+s)^{-1}ds.
\end{aligned}
\end{equation*}
Together with Gronwall's inequality, this yields
\begin{equation}\label{4-14}
\mathcal{E}_{N+1}(t)\leq \mathcal{E}_{N+1}(0)e^{C_{1}C_{2}C_{3}\varepsilon_{0}\cdot\int_{0}^{t}(1+s)^{-1}ds}
\leq \mathcal{E}_{N+1}(0)(1+t)^{2A_{0}\epsilon_{0}}.
\end{equation}
where $A_{0}:=C_{1}C_{2}C_{3}/2$. By \eqref{4-9}, \eqref{0-666} and Sobolev embedding theorem, as long as $\varepsilon_{0}$
is sufficiently small, then
\begin{equation*}
\frac{1}{2}\Big(\|\dot{u}(t)\|^{2}_{H^{N}(\mathbb{R}^{2}\times \mathbb{T})}
+\|u(t)\|^{2}_{H^{N+1}(\mathbb{R}^{2}\times \mathbb{T})}\Big)
\leq \mathcal{E}_{N+1}(t)
\leq 2\Big(\|\dot{u}(t)\|^{2}_{H^{N}(\mathbb{R}^{2}\times \mathbb{T})}
+\|u(t)\|^{2}_{H^{N+1}(\mathbb{R}^{2}\times \mathbb{T})}\Big),~~~t\in [0,T_{0}].
\end{equation*}
Therefore, combing this with \eqref{0-3} and \eqref{4-14} derives
\begin{equation}\label{4-13AD}
\|\dot{u}(t)\|_{H^{N}(\mathbb{R}^{2}\times \mathbb{T})}
+\|u(t)\|_{H^{N+1}(\mathbb{R}^{2}\times \mathbb{T})}\leq 4\varepsilon_{0}(1+t)^{A_{0}\epsilon_{0}}.
\end{equation}
In addition, by $U(t):=\dot{u}(t)-i\Lambda u(t)$, then there is a constant $C_{4}>0$, such that
\begin{equation}\label{4-13BE}
\|U(t)\|_{H^{N}(\mathbb{R}^{2}\times \mathbb{T})}\leq C_{4}(\|\dot{u}(t)\|_{H^{N}(\mathbb{R}^{2}\times \mathbb{T})})
+\|u(t)\|_{H^{N+1}(\mathbb{R}^{2}\times \mathbb{T})}\leq 4C_{4}\varepsilon_{0}(1+t)^{A_{0}\epsilon_{0}}.
\end{equation}
Furthermore, based on \eqref{4-9} and \eqref{4-13BE}, applying Proposition 4.2 concludes that
\begin{equation*}
\sup\limits_{t\in [0,T_{0}]}\|e^{it\Lambda}U(t)-U(0)\|_{\mathcal{H}^{N_{0}}}\leq
(\varepsilon_{0}+c_{0})\cdot (4C_{4}+C_{1})C_{5}\varepsilon_{0}.
\end{equation*}
This implies that if $c_{0}$ and $\varepsilon_{0}$ are sufficiently small, then
\begin{equation*}
\begin{aligned}
\sup\limits_{t\in [0,T_{0}]}\|e^{it\Lambda}U(t)\|_{\mathcal{H}^{N_{0}}}
\leq& \|U(0)\|_{\mathcal{H}^{N_{0}}}+(\varepsilon_{0}+c_{0})\cdot (4C_{4}+C_{1})C_{5}\varepsilon_{0}\\
\leq& [1+(\varepsilon_{0}+c_{0})\cdot (4C_{4}+C_{1})C_{5}]\varepsilon_{0}\\
\leq& \f{1}{2}C_1\varepsilon_{0}.
\end{aligned}
\end{equation*}

Therefore, we can extend the local solution $u$ of of \eqref{0-1} at least on $[0,e^{c_{0}/\epsilon_{0}^{2}}]$
by the continuation argument. Moreover, the estimates in \eqref{0-4} follow from \eqref{4-12A} and \eqref{4-13AD}.

\section{\textbf{Proof of Proposition 4.1}}

In this section, $C>0$ denotes the sufficiently large generic constant that depends only
on $T_0$, $\|U_{0}\|_{Z}$ and  $\sup\limits_{t\in[0,T_{0}]}
\|U(t)\|_{H^{N}(\mathbb{R}^{2}\times\mathbb{T})}$, whose value may change from line to line.

For integer $J\geq0$~and~$f\in H^{N}(\mathbb{R}^{2}\times\mathbb{T})$, define
\begin{equation}\label{5-1}
\|f\|_{Z_{J}}:=\sup\limits_{k,l \in \mathbb{Z},~k,l\geq-1}2^{9(k+l)}
\Big[\|R_{k}S_{l}f\|_{L^{2}(\mathbb{R}^{2}\times\mathbb{T})}
+\sum\limits_{j\in \mathbb{Z}^{+}}2^{\mathrm{min}(j,2J-j)}
\|\varphi_{j}(x)\cdot R_{k}S_{l}f\|_{L^{2}(\mathbb{R}^{2}\times\mathbb{T})}\Big].
\end{equation}
This obviously means
\begin{equation*}
\|f\|_{Z_{J}}\leq\|f\|_{Z},~~~~~~\|f\|_{Z_{J}}\lesssim_{J}\|f\|_{H^{N}(\mathbb{R}^{2}\times\mathbb{T})}.
\end{equation*}

As in (3.20) of \cite{IP2}, we shall show that if~$t, t'\in [0,T_{0}]$ with $0\leq t'-t\leq 1$, and any $J\in \mathbb{Z}_{+}$, then
\begin{equation}\label{5-2}
\begin{aligned}
\|e^{it'\Lambda}U(t')-e^{it\Lambda}U(t)\|_{Z_{J}}&\leq C|t'-t|\Big(1+\sup\limits_{s\in [t,t']}\|e^{is\Lambda}U(s)\|_{Z_{J}}\Big).
\end{aligned}
\end{equation}
Note that under \eqref{5-2}, one easily has that for any $t,t'\in [0,T_{0}]$,
\begin{equation}\label{5-3}
\sup\limits_{t\in [0,T_{0}]}\|e^{it\Lambda}U(t)\|_{Z_{J}}\leq C,\ \|e^{it'\Lambda}U(t')-e^{it\Lambda}U(t)\|_{Z_{J}}\leq C|t'-t|
\end{equation}
hold uniformly in $J$. Subsequently, letting $J\rightarrow\infty$ in \eqref{5-3} yields the results in Proposition 4.1.

Since
\begin{equation}\label{5-4}
\left\{
\begin{aligned}
&\partial_{t}\widehat{u_{n}}(t,\xi)=\frac{\widehat{U_{+;n}}(t,\xi)+\widehat{U_{-;n}}(t,\xi)}{2},\\
&\widehat{u_{n}}(t,\xi)=\frac{\widehat{U_{+;n}}(t,\xi)-\widehat{U_{-;n}}(t,\xi)}{-2i\Lambda_{n}(\xi)},
\end{aligned}
\right.
\end{equation}
and
\begin{equation}\label{5-5}
(\partial_{t}+i\Lambda)U=(\partial_{t}^2-\Delta_{x,y}+1)u=F(u,\partial u,\partial^{2} u),
\end{equation}
then taking the Fourier transform on the variable $(x,y)$ in \eqref{5-5} and utilizing \eqref{5-4} together
with the form of $F(u,\partial u,\partial^{2} u)$, we arrive at
\begin{equation}\label{5-6}
\Big[\partial_{t}+i\Lambda_{n}(\xi)\Big]\widehat{U_{n}}(t,\xi)=\sum_{\mu,\nu\in\{+,-\}}\int_{\mathbb{R}^{2}}\sum_{m\in \mathbb{Z}}
\mathcal{M}_{n,m}(\xi,\eta)\widehat{U_{\mu;n-m}}(t,\xi-\eta)\widehat{U_{\nu;m}}(t,\eta)d\eta,
\end{equation}
where~$\mathcal{M}_{n,m}(\xi,\eta)$ is the linear combination of the products of \eqref{5-7-0} and \eqref{5-7-1} below:
\begin{subequations}\label{5-7-00}
\begin{align}
&1,~\frac{1}{\Lambda_{n-m}(\xi-\eta)},~\frac{\xi_{l}-\eta_{l}}{\Lambda_{n-m}(\xi-\eta)},~\frac{n-m}{\Lambda_{n-m}(\xi-\eta)}\ (1\leq l\leq 2),\label{5-7-0}\\
&1,~\eta_{i},~m,\frac{1}{\Lambda_{m}(\eta)},~\frac{\eta_{i}}{\Lambda_{m}(\eta)},~\frac{m}{\Lambda_{m}(\eta)},
~\frac{\eta_{i}\eta_{j}}{\Lambda_{m}(\eta)},~
\frac{m\eta_{i}}{\Lambda_{m}(\eta)},~\frac{m^{2}}{\Lambda_{m}(\eta)}\ (1\leq i, j\leq 2).\label{5-7-1}
\end{align}
\end{subequations}

By \eqref{2-9}, \eqref{2-2}, \eqref{2-10} and \eqref{2-11}, the equation \eqref{5-6} is actually equivalent to
\begin{equation}\label{5-8}
\frac{d}{dt}[\widehat{V_{+;n}}(t,\xi)]=\sum_{\mu,\nu\in\{+,-\}}\int_{\mathbb{R}^{2}}\sum_{m\in \mathbb{Z}}
e^{it \Phi_{\mu, \nu}^{n, m}(\xi, \eta)}\mathcal{M}_{n,m}(\xi,\eta)\widehat{V_{\mu;n-m}}(t,\xi-\eta)\widehat{V_{\nu;m}}(t,\eta)d\eta.
\end{equation}
Therefore, for any $t\leq t'$ with $t, t'\in [0,T_{0}]$,
\begin{equation}\label{5-10A}
\begin{aligned}
&\widehat{V_{+;n}}(t',\xi)-\widehat{V_{+;n}}(t,\xi)\\
=&\sum_{\mu,\nu\in\{+,-\}}\int_{t}^{t'}\int_{\mathbb{R}^{2}}\sum_{m\in \mathbb{Z}}
e^{is \Phi_{\mu, \nu}^{n, m}(\xi, \eta)}\mathcal{M}_{n,m}(\xi,\eta)
\widehat{V_{\mu;n-m}}(s,\xi-\eta)\widehat{V_{\nu;m}}(s,\eta)d\eta ds\\
:=&\sum_{\mu,\nu\in\{+,-\}}\int_{t}^{t'}\mathcal{F}_{x,y}(M_{s}^{\mu,\nu})(\xi,n)ds.
\end{aligned}
\end{equation}

Since \eqref{5-2} is equivalent to
\begin{align}\label{5-A12}
\|V_{+}(t')-V_{+}(t)\|_{Z_{J}}~\leq~C|t'-t|\Big(1+\sup\limits_{s\in [t,t']}\|V_{+}(s)\|_{Z_{J}}\Big),
\end{align}
then by \eqref{5-10A} and \eqref{5-1}, \eqref{5-A12} as well as \eqref{5-2} will be obtained if there holds
\begin{equation}\label{5-AB}
\|M_s^{\mu, \nu}\|_{Z_J}\leq C\Big(1+\sup\limits_{s\in [t,t']}\|V_{+}(s)\|_{Z_{J}}\Big)~~~(s\in[0,T_{0}],\ \mu,\nu\in\{+,-\}).
\end{equation}
According to the definition of $\|\cdot\|_{Z_J}$, the proof of  \eqref{5-AB} will be
divided into the following two steps.

Additionally, for later use, we define that for $k\in \mathbb{Z}$,
\begin{equation}\label{2-3}
\begin{aligned}
&\chi_{k}^{1}:=\{(k_{1},k_{2})\in \mathbb{Z}^{2}:|\mathrm{max}(k_{1},k_{2})-k|\leq8,~~k_{1},k_{2}\geq -1\},\\
&\chi_{k}^{2}:=\{(k_{1},k_{2})\in \mathbb{Z}^{2}:\mathrm{max}(k_{1},k_{2})-k\geq8~\mathrm{and}~|k_{1}-k_{2}|\leq8,~~k_{1},k_{2}\geq -1\},\\
&\chi_{k}:=\chi_{k}^{1}\cup\chi_{k}^{2}.
\end{aligned}
\end{equation}
It is easy to check that
\begin{equation}\label{2-3333}
\mathrm{if}~\psi_{k}(x)\psi_{k_{1}}(x-y)\psi_{k_{2}}(y)\neq0,~\mathrm{then}~(k_{1},k_{2})\in\chi_{k}.
\end{equation}

\vskip 0.2cm

{\bf Step 1: Establish}
\begin{equation}\label{5-ABC}
2^{9(k+l)}\Big[\|R_{k}S_{l}M_{s}^{\mu,\nu}\|_{L^{2}(\mathbb{R}^{2}\times\mathbb{T})}+
\sum\limits_{1\leq j\leq 100}2^{j}\|\varphi_{j}(x)
\cdot R_{k}S_{l}M_{s}^{\mu,\nu}\|_{L^{2}(\mathbb{R}^{2}\times\mathbb{T})}\Big]
\leq C.
\end{equation}

By the definition of $R_k, S_l$ for $k, l\geq -1$, one has
\begin{equation}\label{5-1-1}
\|R_{k}S_{l}M_{s}^{\mu,\nu}\|_{L^{2}(\mathbb{R}^{2}\times\mathbb{T})}\lesssim 2^{k}2^{l/2}
\|\mathcal{F}_{x,y}(R_{k}S_{l}M_{s}^{\mu,\nu})(\xi, n)\|_{L^{\infty}_{\xi,n}}\lesssim2^k 2^{l/2}(I_1+I_2+I_3),
\end{equation}
where
\begin{equation*}
\begin{aligned}
I_1=&\sum_{(l_{1}, l_{2})\in\chi_{l}}2^{-10(l_{1}+l_{2})}
\|\mathcal{F}_{x}(R_{\leq k+8}V_{\mu;n}(s))(\xi)\langle n\rangle^{10}\psi_{l_{1}}(n)\|_{L^{2}_{\xi}l^{2}_{n}}
\\&~~~~~~~~~~~~~~~~~~~
~~~~~~~~~~~~~~~~~~\cdot\|(1+|\xi|+|n|)\mathcal{F}_{x}(R_{[k-8,k+8]}V_{\nu;n}(s))(\xi)\langle n\rangle^{10}
\psi_{l_{2}}(n)\|_{L^{2}_{\xi}l^{2}_{n}},\\
I_2=&\sum_{(l_{1},l_{2})\in\chi_{l}}2^{-10(l_{1}+l_{2})}
\|\mathcal{F}_{x}(R_{[k-8,k+8]}V_{\mu;n}(s))(\xi)\langle n\rangle^{10}
\psi_{l_{1}}(n)\|_{L^{2}_{\xi}l^{2}_{n}}\\&~~~~~~~~~~~~~~~~~~~~~
~~~~~~~~~~~~~~~~\cdot\|(1+|\xi|+|n|)\mathcal{F}_{x}(R_{\leq k+8}V_{\nu;n}(s))(\xi)\langle n\rangle^{10}
\psi_{l_{2}}(n)\|_{L^{2}_{\xi}l^{2}_{n}},\\
I_3=&\sum_{\mbox{\tiny$\begin{array}{c}
|k_{1}-k_{2}|\leq8,\mathrm{max}(k_{1},k_{2})\geq k+8,\\
(l_{1},l_{2})\in\chi_{l}\end{array}$}}2^{-10(l_{1}+l_{2})}\|\mathcal{F}_{x}(R_{k_{1}}V_{\mu;n}(s))(\xi)\langle n\rangle^{10}\psi_{l_{1}}(n)\|_{L^{2}_{\xi}l^{2}_{n}}
\\&~~~~~~~~~~~~~~~~~~~~~~~~~~~~~~~~
~~~~~~\cdot\|(1+|\xi|+|n|)\mathcal{F}_{x}(R_{k_{2}}V_{\nu;n}(s))(\xi)\langle n\rangle^{10}\psi_{l_{2}}(n)\|_{L^{2}_{\xi}l^{2}_{n}},
\end{aligned}
\end{equation*}
and the set $\chi_{l}$ has been defined in \eqref{2-3}.\\
Since $V_{\mu}\in H^{N}(\mathbb{R}^{2}\times\mathbb{T})$ with $\mu\in\{+,-\},$ ~direct calculation shows
\begin{equation}\label{5-13}
\begin{aligned}
2^{10k}(I_1+I_2+I_3)&\leq C\sum_{(l_{1},l_{2})\in\chi_{l}}2^{-10(l_{1}+l_{2})}
\|V_{\mu}(s)\|_{H^{N}(\mathbb{R}^{2}\times \mathbb{T})}\|V_{\nu}(s)\|_{H^{N}(\mathbb{R}^{2}\times \mathbb{T})}
\\&\leq C\sum_{(l_{1},l_{2})\in\chi_{l}}2^{-10(l_{1}+l_{2})}.
\end{aligned}
\end{equation}
Then \eqref{5-ABC} comes from \eqref{5-1-1} and \eqref{5-13}.

\vskip 0.1 true cm

{\bf Step 2: Establish}
\begin{equation}\label{5-ABCD}
\begin{aligned}
2^{9(k+l)}\sum\limits_{j\geq 100}
2^{\mathrm{min}(j,2J-j)}&\|\varphi_{j}(x)\cdot R_{k}S_{l}M_{s}^{\mu,\nu}
\|_{L^{2}(\mathbb{R}^{2}\times\mathbb{T})}
\leq C\Big(1+\sup\limits_{s\in [t,t']}\|V_{+}(s)\|_{Z_{J}}\Big).
\end{aligned}
\end{equation}

To this end, based on \eqref{5-10A} and \eqref{2-3333}, we decompose
\begin{equation}\label{5-AAA1}
\begin{aligned}
&R_{k}S_{l}M_{s}^{\mu,\nu}=\sum_{\mbox{\tiny$\begin{array}{c}
(k_{1}, k_{2})\in\chi_{k}, (l_{1}, l_{2})\in\chi_{l}\end{array}$}}
M_{s, k, k_{1}, k_{2}, l, l_{1}, l_{2}}^{\mu,\nu},\\
\end{aligned}
\end{equation}
where
\begin{equation*}
\begin{aligned}
&\mathcal{F}_{x,y}(M_{s,k,k_{1},k_{2},l,l_{1},l_{2}}^{\mu,\nu})(\xi,n)
:=\psi_{k}(\xi)\psi_{l}(n)\cdot\int_{\mathbb{R}^{2}}\sum_{m\in \mathbb{Z}}
e^{is\Phi_{\mu, \nu}^{n, m}(\xi, \eta)}\mathcal{M}_{n,m}(\xi,\eta)
\\&~~~~~~~~~~~~~~~~~~~~~~~~~~~~~~~~~~~~~~~~~~~
~\cdot\psi_{k_{1}}(\xi-\eta)\psi_{l_{1}}(n-m)\widehat{V_{\mu;n-m}}(s,\xi-\eta)
\psi_{k_{2}}(\eta)\psi_{l_{2}}(m)\widehat{V_{\nu;m}}(s,\eta)d\eta.
\end{aligned}
\end{equation*}

In addition, $M_{s,k,k_{1},k_{2},l,l_{1},l_{2}}^{\mu,\nu}$ can be written as
\begin{equation}\label{5-23AA}
\begin{aligned}
&M_{s,k,k_{1},k_{2},l,l_{1},l_{2}}^{\mu,\nu}(x,y)\\
=&\int_{\mathbb{R}^{2}\times\mathbb{T}}\int_{\mathbb{R}^{2}\times\mathbb{T}}K(s,x,y,x_{1},y_{1},x_{2},y_{2})
g_{k_{1}l_{1}}^{\mu}(s,x_{1},y_{1})
g_{k_{2}l_{2}}^{\nu}(s,x_{2},y_{2})
dx_{1}dy_{1}dx_{2}dy_{2},
\end{aligned}
\end{equation}
where
\begin{equation*}
g_{kl}^{\mu}(s, x, y):=\sum\limits_{n\in \mathbb{Z}}R_{k}V_{\mu; n}(s, x)\psi_{l}(n)\langle n\rangle^{4}e^{i n y},
\end{equation*}
and
\begin{equation}\label{5-A1A}
\begin{aligned}
&K(s,x,y,x_{1},y_{1},x_{2},y_{2})\\
:=&\sum_{n\in \mathbb{Z}}\sum_{m\in \mathbb{Z}}\Big(\psi_{l}(n)\psi_{[l_{1}-2,l_{1}+2]}(n-m)
\psi_{[l_{2}-2,l_{2}+2]}(m)e^{i[n\cdot (y-y_{1})+m\cdot(y_{1}-y_{2})]}\langle n-m\rangle^{-4}\langle m\rangle^{-4}\Big)\\
&~~~~~~~~~~~~~~~~~~~~~~~~~~~~~~~~~~~~~~~~~~~~~~~~~~
~~~\cdot\int_{\mathbb{R}^{2}\times\mathbb{R}^{2}}\Big(e^{i[(x-x_{1})\cdot\xi+(x_{1}-x_{2})\cdot\eta]}
e^{is \Phi_{\mu, \nu}^{n, m}(\xi, \eta)}\mathcal{M}_{n,m}(\xi,\eta)
\\&~~~~~~~~~~~~~~~~~~~~~~~~~~~~~~~~~~~~~~~~~~~
~~~~~~~~~~~~~\cdot\psi_{k}(\xi)\psi_{[k_{1}-2,k_{1}+2]}(\xi-\eta)\psi_{[k_{2}-2, k_{2}+2]}(\eta)\Big)d\xi d\eta.
\end{aligned}
\end{equation}

At first, when $j\geq 100$, by \eqref{5-7-00} and the method of non-stationary phase, one has
\begin{equation}\label{5-BAB}\begin{aligned}
&|K(s, x, y, x_1, y_1, x_2, y_2)|\\
\lesssim& \sup\limits_{n, m\in\mathbb{Z}}\Big|\psi_{[l_2-2, l_2+2]}(m)
\int_{\mathbb{R}^{2}\times\mathbb{R}^{2}}\Big(e^{i[(x-x_{1})\cdot\xi+(x_{1}-x_{2})\cdot\eta]}
e^{is \Phi_{\mu, \nu}^{n, m}(\xi, \eta)}\mathcal{M}_{n,m}(\xi,\eta)\\
&\cdot\psi_{k}(\xi)\psi_{[k_{1}-2,k_{1}+2]}(\xi-\eta)\psi_{[k_{2}-2, k_{2}+2]}(\eta)\Big)d\xi d\eta\Big|\\
\leq &C(2^{k_{2}}+2^{l_{2}})
2^{2(k_{1}+k_{2})}(|x-x_{1}|+|x_{1}-x_{2}|)^{-10}~~(\text{if}~|x-x_{1}|+|x_{1}-x_{2}|\geq2^{j-10}).
\end{aligned}
\end{equation}
In addition, similarly to \eqref{5-23AA}, we set
\begin{equation}\label{5-26AAA1}
\begin{aligned}
&N_{s,k,k_{1},k_{2},l,l_{1},l_{2}}^{\mu,\nu,j}(x,y)\\
:=&\int_{\mathbb{R}^{2}\times\mathbb{T}}\int_{\mathbb{R}^{2}\times\mathbb{T}}K(s,x,y,x_{1},y_{1},x_{2},y_{2})
g^{\mu}_{jk_{1}l_{1}}(s,x_{1},y_{1})
g^{\nu}_{jk_{2}l_{2}}(s,x_{2},y_{2})
dx_{1}dy_{1}dx_{2}dy_{2},
\end{aligned}
\end{equation}
where
\begin{equation*}
g^{\mu}_{jkl}(s,x,y):=\varphi_{[j-4,j+4]}(x)g_{kl}^{\mu}(s, x, y).
\end{equation*}

It follows from a direct calculation that
\begin{equation}\label{5-27AA}
\begin{aligned}
&\Big\|\varphi_{j}(x)\cdot (M_{s,k,k_{1},k_{2},l,l_{1},l_{2}}^{\mu,\nu}-N_{s,k,k_{1},
k_{2},l,l_{1},l_{2}}^{\mu,\nu,j})\Big\|_{L^{2}(\mathbb{R}^{2}\times\mathbb{T})}\\
\lesssim& 2^{j}\Big\|\varphi_{j}(x)\cdot (M_{s, k,k_{1},k_{2},l,l_{1},l_{2}}^{\mu,\nu}-N_{s,k,k_{1},k_{2},l,l_{1},l_{2}}^{\mu,\nu,j})
\Big\|_{L^{\infty}(\mathbb{R}^{2}\times\mathbb{T})}\lesssim 2^{j}(I_4+I_5),
\end{aligned}
\end{equation}
where
\begin{equation*}
\begin{aligned}
I_4=&\Big|\int_{\mathbb{R}^{2}\times\mathbb{T}}\int_{\mathbb{R}^{2}\times\mathbb{T}}\textbf{1}_{\{(x_{1},x_{2})||x_{1}|\notin
[2^{j-3},2^{j+3}]~\mathrm{or}~|x_{2}|\notin [2^{j-3},2^{j+3}]\}}\varphi_{j}(x)K(s,x,y,x_{1},y_{1},x_{2},y_{2})
\\&~~\cdot(1-\varphi_{[j-4,j+4]}(x_{1})\varphi_{[j-4,j+4]}(x_{2}))g_{k_{1}l_{1}}(s,x_{1},y_{1})
g_{k_{2}l_{2}}(s,x_{2},y_{2})
dx_{1}dy_{1}dx_{2}dy_{2}\Big|,
\end{aligned}
\end{equation*}
\begin{equation*}
\begin{aligned}
I_5=&\Big|\int_{\mathbb{R}^{2}\times\mathbb{T}}\int_{\mathbb{R}^{2}\times\mathbb{T}}\textbf{1}_{\{(x_{1},x_{2})||x_{1}|\in
[2^{j-3},2^{j+3}]~\mathrm{and}~|x_{2}|\in [2^{j-3},2^{j+3}]\}}\varphi_{j}(x)K(s,x,y,x_{1},y_{1},x_{2},y_{2})
\\&~~~\cdot(1-\varphi_{[j-4,j+4]}(x_{1})\varphi_{[j-4,j+4]}(x_{2}))g_{k_{1}l_{1}}(s,x_{1},y_{1})
g_{k_{2}l_{2}}(s,x_{2},y_{2})
dx_{1}dy_{1}dx_{2}dy_{2}\Big|.
\end{aligned}
\end{equation*}

In view of the support property of $\varphi_{[j-4,j+4]}$, then $I_5\equiv0$ holds.

For $I_4$, we have
\begin{equation}\label{5-28AAA}
\begin{aligned}
I_4\leq I_{41}+I_{42}+I_{43}+I_{44},
\end{aligned}
\end{equation}
where
\begin{equation}\label{5-30A1}
\begin{aligned}
I_{41}=&\int_{\mathbb{R}^{2}\times\mathbb{T}}\int_{\mathbb{R}^{2}\times\mathbb{T}}\textbf{1}_{\{(x_{1},x_{2}):\ |x_{1}|\leq
2^{j-3}~\mathrm{and}~|x_{2}|\leq 2^{j+3}\}}\Big|\varphi_{j}(x)K(s,x,y,x_{1},y_{1},x_{2},y_{2})
\\&\cdot(1-\varphi_{[j-4,j+4]}(x_{1})\varphi_{[j-4,j+4]}(x_{2}))g_{k_{1}l_{1}}^{\mu}(s,x_{1},y_{1})
g_{k_{2}l_{2}}^{\nu}(s,x_{2},y_{2})\Big|
dx_{1}dy_{1}dx_{2}dy_{2}
\\&+\int_{\mathbb{R}^{2}\times\mathbb{T}}\int_{\mathbb{R}^{2}\times\mathbb{T}}\textbf{1}_{\{(x_{1},x_{2}):\ |x_{1}|\leq
2^{j-3}~\mathrm{and}~|x_{2}|\geq 2^{j+3}\}}\Big|\varphi_{j}(x)K(s,x,y,x_{1},y_{1},x_{2},y_{2})
\\&\cdot(1-\varphi_{[j-4,j+4]}(x_{1})\varphi_{[j-4,j+4]}(x_{2}))g_{k_{1}l_{1}}^{\mu}(s,x_{1},y_{1})
g_{k_{2}l_{2}}^{\nu}(s,x_{2},y_{2})\Big|
dx_{1}dy_{1}dx_{2}dy_{2},\nonumber
\end{aligned}
\end{equation}
\begin{equation}\label{5-30A2}
\begin{aligned}
I_{42}=&\int_{\mathbb{R}^{2}\times\mathbb{T}}\int_{\mathbb{R}^{2}\times\mathbb{T}}\textbf{1}_{\{(x_{1},x_{2}):\ |x_{1}|\geq
2^{j+3}~\mathrm{and}~|x_{2}|\leq 2^{j+3}\}}\Big|\varphi_{j}(x)K(s,x,y,x_{1},y_{1},x_{2},y_{2})
\\&\cdot(1-\varphi_{[j-4,j+4]}(x_{1})\varphi_{[j-4,j+4]}(x_{2}))g_{k_{1}l_{1}}^{\mu}(s,x_{1},y_{1})g_{k_{2}l_{2}}^{\nu}(s,x_{2},y_{2})\Big|
dx_{1}dy_{1}dx_{2}dy_{2}
\\&+\int_{\mathbb{R}^{2}\times\mathbb{T}}\int_{\mathbb{R}^{2}\times\mathbb{T}}\textbf{1}_{\{(x_{1},x_{2}):\ |x_{1}|\geq
2^{j+3}~\mathrm{and}~|x_{2}|\geq 2^{j+3}\}}\Big|\varphi_{j}(x)K(s,x,y,x_{1},y_{1},x_{2},y_{2})
\\&\cdot(1-\varphi_{[j-4,j+4]}(x_{1})\varphi_{[j-4,j+4]}(x_{2}))g_{k_{1}l_{1}}^{\mu}(s,x_{1},y_{1})g_{k_{2}l_{2}}^{\nu}(s, x_{2},y_{2})\Big|
dx_{1}dy_{1}dx_{2}dy_{2},\nonumber
\end{aligned}
\end{equation}
\begin{equation}\label{5-30A3}
\begin{aligned}
I_{43}=&\int_{\mathbb{R}^{2}\times\mathbb{T}}\int_{\mathbb{R}^{2}\times\mathbb{T}}\textbf{1}_{\{(x_{1},x_{2}):\ |x_{2}|\leq
2^{j-3}~\mathrm{and}~|x_{1}|\leq 2^{j+3}\}}\Big|\varphi_{j}(x)K(s,x,y,x_{1},y_{1},x_{2},y_{2})
\\&\cdot(1-\varphi_{[j-4,j+4]}(x_{1})\varphi_{[j-4,j+4]}(x_{2}))g_{k_{1}l_{1}}^{\mu}(s,x_{1},y_{1})g_{k_{2}l_{2}}^{\nu}(s,x_{2},y_{2})\Big|
dx_{1}dy_{1}dx_{2}dy_{2}
\\&+\int_{\mathbb{R}^{2}\times\mathbb{T}}\int_{\mathbb{R}^{2}\times\mathbb{T}}\textbf{1}_{\{(x_{1},x_{2}):\ |x_{2}|\leq
2^{j-3}~\mathrm{and}~|x_{1}|\geq 2^{j+3}\}}\Big|\varphi_{j}(x)K(s,x,y,x_{1},y_{1},x_{2},y_{2})
\\&\cdot(1-\varphi_{[j-4,j+4]}(x_{1})\varphi_{[j-4,j+4]}(x_{2}))g_{k_{1}l_{1}}^{\mu}(s,x_{1},y_{1})g_{k_{2}l_{2}}^{\nu}(s,x_{2},y_{2})\Big|
dx_{1}dy_{1}dx_{2}dy_{2},\nonumber
\end{aligned}
\end{equation}
\begin{equation}\label{5-30A4}
\begin{aligned}
I_{44}=&\int_{\mathbb{R}^{2}\times\mathbb{T}}\int_{\mathbb{R}^{2}\times\mathbb{T}}\textbf{1}_{\{(x_{1},x_{2}):\ |x_{2}|\geq
2^{j+3}~\mathrm{and}~|x_{1}|\leq 2^{j+3}\}}\Big|\varphi_{j}(x)K(s,x,y,x_{1},y_{1},x_{2},y_{2})
\\&\cdot(1-\varphi_{[j-4,j+4]}(x_{1})\varphi_{[j-4,j+4]}(x_{2}))g_{k_{1}l_{1}}^{\mu}(s,x_{1},y_{1})g_{k_{2}l_{2}}^{\nu}(s,x_{2},y_{2})\Big|
dx_{1}dy_{1}dx_{2}dy_{2}
\\&+\int_{\mathbb{R}^{2}\times\mathbb{T}}\int_{\mathbb{R}^{2}\times\mathbb{T}}\textbf{1}_{\{(x_{1},x_{2}):\ |x_{2}|\geq
2^{j+3}~\mathrm{and}~|x_{1}|\geq 2^{j+3}\}}\Big|\varphi_{j}(x)K(s,x,y,x_{1},y_{1},x_{2},y_{2})
\\&\cdot(1-\varphi_{[j-4,j+4]}(x_{1})\varphi_{[j-4,j+4]}(x_{2}))g_{k_{1}l_{1}}^{\mu}(s,x_{1},y_{1})g_{k_{2}l_{2}}^{\nu}(s,x_{2},y_{2})\Big|
dx_{1}dy_{1}dx_{2}dy_{2}.\nonumber
\end{aligned}
\end{equation}

For $I_{41}, $ by \eqref{5-BAB}, we arrive at
\begin{equation}\label{5-34A}
\begin{aligned}
I_{41}&\lesssim (2^{k_{2}}+2^{l_{2}})2^{2(k_{1}+k_{2})}\Big(2^{-10j}2^{4j}+2^{-5j}2^{2j}\|(1+|x-x_{2}|)^{-3}\|_{L_{x_{2}}^{1}(\mathbb{R}^{2})}\Big)
\|g_{k_{1}l_{1}}^{\mu}\|_{L^{\infty}(\mathbb{R}^{2}\times\mathbb{T})}
\|g_{k_{2}l_{2}}^{\nu}\|_{L^{\infty}(\mathbb{R}^{2}\times\mathbb{T})}
\\&\lesssim (2^{k_{2}}+2^{l_{2}})2^{2(k_{1}+k_{2})}2^{-3j}\|\widehat{R_{k_{1}}V_{\mu;n}}(s,\xi)\langle n\rangle^{4}\psi_{l_{1}}(n)\|_{L^{1}_{\xi}l^{1}_{n}}
\|\widehat{R_{k_{2}}V_{\nu;n}}(s,\xi)\langle n\rangle^{4}\psi_{l_{2}}(n)\|_{L^{1}_{\xi}l^{1}_{n}}
\\&\lesssim (2^{k_{2}}+2^{l_{2}})2^{2(k_{1}+k_{2})}2^{-3j}2^{4l_{1}}\|\widehat{R_{k_{1}}V_{\mu;n}}(s,\xi)\psi_{l_{1}}(n)\|_{L^{1}_{\xi}l^{1}_{n}}
2^{4l_{2}}\|\widehat{R_{k_{2}}V_{\nu;n}}(s,\xi)\psi_{l_{2}}(n)\|_{L^{1}_{\xi}l^{1}_{n}}
\\&\lesssim (2^{k_{2}}+2^{l_{2}})2^{2(k_{1}+k_{2})}2^{-3j}2^{4(l_{1}+l_{2})}2^{k_{1}}2^{l_{1}/2}\|\widehat{R_{k_{1}}V_{\mu;n}}(s,\xi)
\psi_{l_{1}}(n)\|_{L^{2}_{\xi}l^{2}_{n}}
2^{k_{2}}2^{l_{2}/2}\|\widehat{R_{k_{2}}V_{\nu;n}}(s,\xi)\psi_{l_{2}}(n)\|_{L^{2}_{\xi}l^{2}_{n}}
\\&\lesssim(2^{k_{2}}+2^{l_{2}})2^{2(k_{1}+k_{2})}2^{-3j}2^{4(l_{1}+l_{2})}2^{k_{1}}2^{l_{1}/2}2^{k_{2}}
2^{l_{2}/2}2^{-N(k_{1}+l_{1})/2}2^{-N(k_{2}+l_{2})/2}
\\&\lesssim 2^{-3j}2^{-(N-8)(k_{1}+k_{2})/2}2^{-(N-12)(l_{1}+l_{2})/2}.
\end{aligned}
\end{equation}
Analogously,
\begin{equation}\label{5-35A}
\begin{aligned}
I_{42}&\lesssim (2^{k_{2}}+2^{l_{2}})2^{2(k_{1}+k_{2})}\Big(2^{-5j}2^{2j}\|(1+|x-x_{1}|)^{-3}\|_{L_{x_{1}}^{1}(\mathbb{R}^{2})}
\\&+2^{-5j}\|(1+|x-x_{1}|)^{-5/2}\|_{L_{x_{1}}^{1}(\mathbb{R}^{2})}
\|(1+|x-x_{2}|)^{-5/2}\|_{L_{x_{2}}^{1}(\mathbb{R}^{2})}\Big)\|g_{k_{1}l_{1}}^{\mu}\|_{L^{\infty}(\mathbb{R}^{2}\times\mathbb{T})}
\|g_{k_{2}l_{2}}^{\nu}\|_{L^{\infty}(\mathbb{R}^{2}\times\mathbb{T})}
\\&\lesssim (2^{k_{2}}+2^{l_{2}})2^{2(k_{1}+k_{2})}2^{-3j}\|g_{k_{1}l_{1}}^{\mu}\|_{L^{\infty}(\mathbb{R}^{2}\times\mathbb{T})}
\|g_{k_{2}l_{2}}^{\nu}\|_{L^{\infty}(\mathbb{R}^{2}\times\mathbb{T})}
\\&\lesssim 2^{-3j}2^{-(N-8)(k_{1}+k_{2})/2}2^{-(N-12)(l_{1}+l_{2})/2},
\end{aligned}
\end{equation}
\begin{equation}\label{5-36A}
\begin{aligned}
I_{43}&\lesssim (2^{k_{2}}+2^{l_{2}})2^{2(k_{1}+k_{2})}\Big(2^{-10j}2^{4j}+2^{-5j}2^{2j}
\|(1+|x-x_{1}|)^{-3}\|_{L_{x_{1}}^{1}(\mathbb{R}^{2})}\Big)\|g_{k_{1}l_{1}}^{\mu}
\|_{L^{\infty}(\mathbb{R}^{2}\times\mathbb{T})}
\|g_{k_{2}l_{2}}^{\nu}\|_{L^{\infty}(\mathbb{R}^{2}\times\mathbb{T})}
\\&\lesssim 2^{-3j}2^{-(N-8)(k_{1}+k_{2})/2}2^{-(N-12)(l_{1}+l_{2})/2}
\end{aligned}
\end{equation}
and
\begin{equation}\label{5-37A}
\begin{aligned}
I_{44}&\lesssim (2^{k_{2}}+2^{l_{2}})2^{2(k_{1}+k_{2})}\Big(2^{-5j}2^{2j}\|(1+|x-x_{2}|)^{-3}\|_{L_{x_{2}}^{1}(\mathbb{R}^{2})}+2^{-5j}
\|(1+|x-x_{1}|)^{-5/2}\|_{L_{x_{1}}^{1}(\mathbb{R}^{2})}
\\&~~~~~~~~~~~~~~~~~~~~~~~~~
~~~~~~~~~~~~~~~~~~~~\cdot\|(1+|x-x_{2}|)^{-5/2}\|_{L_{x_{2}}^{1}(\mathbb{R}^{2})}\Big)
\|g_{k_{1}l_{1}}^{\mu}\|_{L^{\infty}(\mathbb{R}^{2}\times\mathbb{T})}
\|g_{k_{2}l_{2}}^{\nu}\|_{L^{\infty}(\mathbb{R}^{2}\times\mathbb{T})}
\\&\lesssim 2^{-3j}2^{-(N-8)(k_{1}+k_{2})/2}2^{-(N-12)(l_{1}+l_{2})/2}.
\end{aligned}
\end{equation}
Collecting \eqref{5-27AA}-\eqref{5-37A} yields
\begin{equation}\label{5-33A}
\begin{aligned}
2^{9(k+l)}\sum\limits_{j\geq 100}
\sum_{\mbox{\tiny$\begin{array}{c}
(k_{1},k_{2})\in\chi_{k},(l_{1},l_{2})\in\chi_{l}\end{array}$}}
2^{j}\|\varphi_{j}(x)\cdot &(M_{s,k,k_{1},k_{2},l,l_{1},l_{2}}^{\mu,\nu}
-N_{s,k,k_{1},k_{2},l,l_{1},l_{2}}^{\mu,\nu,j})\|_{L^{2}(\mathbb{R}^{2}\times\mathbb{T})}
\leq C.
\end{aligned}
\end{equation}

\vskip 0.1 true cm

Next, we estimate the term
\begin{equation}
\begin{aligned}
2^{9(k+l)}\sum\limits_{j\geq 100}
\sum_{\mbox{\tiny$\begin{array}{c}
(k_{1},k_{2})\in\chi_{k},(l_{1},l_{2})\in\chi_{l}\end{array}$}}
2^{\mathrm{min}(j,2J-j)}\|\varphi_{j}(x)\cdot N_{s,k,k_{1},k_{2},l,l_{1},l_{2}}^{\mu,\nu,j}
&\|_{L^{2}(\mathbb{R}^{2}\times\mathbb{T})}.\nonumber
\end{aligned}
\end{equation}
Note that the Fourier transformation of the term $N_{s,k,k_{1},k_{2},l,l_{1},l_{2}}^{\mu,\nu,j}$ can be written as
\begin{equation}\label{5-34AAA1}
\begin{aligned}
&\mathcal{F}_{x, y}(N_{s, k, k_{1}, k_{2}, l, l_{1}, l_{2}}^{\mu,\nu, j})(\xi, n)\\
=&\sum_{m\in \mathbb{Z}}\int_{\mathbb{R}^{2}}\langle n-m\rangle^{-4}\langle m\rangle^{-4}
e^{is \Phi_{\mu, \nu}^{n, m}(\xi, \eta)}
\mathcal{M}_{n,m}(\xi,\eta)\psi_{k}(\xi)\psi_{[k_{1}-2,k_{1}+2]}(\xi-\eta)\psi_{[k_{2}-2,k_{2}+2]}(\eta)\\
&\cdot\psi_{l}(n)\psi_{[l_{1}-2,l_{1}+2]}(n-m)
\psi_{[l_{2}-2,l_{2}+2]}(m)\mathcal{F}_{x,y}(g^{\mu}_{jk_{1}l_{1}})(s,\xi-\eta,n-m)
\mathcal{F}_{x,y}(g^{\nu}_{jk_{2}l_{2}})(s,\eta,m)d\eta.
\end{aligned}
\end{equation}
Then
\begin{equation}\label{5-35A}
\begin{aligned}
&\|N_{s,k,k_{1},k_{2},l,l_{1},l_{2}}^{\mu,\nu,j}\|_{L^{2}(\mathbb{R}^{2}\times\mathbb{T})}\\
=&\|\mathcal{F}_{x,y}(N_{s,k,k_{1},k_{2},l,l_{1},l_{2}}^{\mu,\nu,j})(\xi,n)\|_{L^{2}_{\xi}l^{2}_{n}}\\
\lesssim& (2^{k_{2}}+2^{l_{2}})2^{-4(l_{1}+l_{2})}
\|\mathcal{F}_{x,y}(g^{\mu}_{jk_{1}l_{1}})(s,\xi,n)\psi_{[k_{1}-2,k_{1}+2]}(\xi)\psi_{[l_{1}-2,l_{1}+2]}(n)\|_{L^{1}_{\xi}l^{1}_{n}}\\
&~~~~~~~~~~~~~~~~~~~~~~~~~~~~~~~~~~~~~~~~~~\cdot\|\mathcal{F}_{x,y}(g^{\nu}_{jk_{2}l_{2}})(s,\xi,n)
\psi_{[k_{2}-2,k_{2}+2]}(\xi)\psi_{[l_{2}-2,l_{2}+2]}(n)\|_{L^{2}_{\xi}l^{2}_{n}}\\
\lesssim& (2^{k_{2}}+2^{l_{2}})2^{-4(l_{1}+l_{2})}
2^{k_{1}}2^{l_{1}/2}2^{4l_{1}}\|\varphi_{[j-4,j+4]}(x)
\cdot R_{k_{1}}S_{l_{1}}V_{\mu}\|_{L^{2}(\mathbb{R}^{2}\times\mathbb{T})}\\
&~~~~~~~~~~~~~~~~~~~~~
~~~~~~~~~~~~~~~~~~~~~~~~~\cdot2^{4l_{2}}\|\varphi_{[j-4,j+4]}(x)\cdot R_{k_{2}}S_{l_{2}}V_{\nu}\|_{L^{2}(\mathbb{R}^{2}\times\mathbb{T})}\\
=&(2^{k_{2}}+2^{l_{2}})2^{k_{1}}2^{l_{1}/2}\|\varphi_{[j-4,j+4]}(x)
\cdot R_{k_{1}}S_{l_{1}}V_{\mu}\|_{L^{2}(\mathbb{R}^{2}\times\mathbb{T})}
\|\varphi_{[j-4,j+4]}(x)\cdot R_{k_{2}}S_{l_{2}}V_{\nu}\|_{L^{2}(\mathbb{R}^{2}\times\mathbb{T})}.
\end{aligned}
\end{equation}

For any integer $J\geq 0$, we set
\begin{equation}\label{5-35AB}
I_{6}:=\sum\limits_{j\geq 100}
\Big(2^{\mathrm{min}(j,2J-j)}\|\varphi_{[j-4,j+4]}(x)\cdot R_{k_{1}}S_{l_{1}}V_{\mu}
\|_{L^{2}(\mathbb{R}^{2}\times\mathbb{T})}\|\varphi_{[j-4,j+4]}(x)
\cdot R_{k_{2}}S_{l_{2}}V_{\nu}\|_{L^{2}(\mathbb{R}^{2}\times\mathbb{T})}\Big).
\end{equation}
By $V_{\mu}\in H^{N}(\mathbb{R}^{2}\times\mathbb{T})$ and \eqref{5-1}, one has
\begin{equation}\label{5-37AA}
\begin{aligned}
I_{6}&\lesssim \|R_{k_{1}}S_{l_{1}}V_{\mu}\|_{L^{2}(\mathbb{R}^{2}\times\mathbb{T})}
\sum\limits_{j\geq 100}\Big(2^{\mathrm{min}(j,2J-j)}\|\varphi_{[j-4,j+4]}(x)\cdot R_{k_{2}}S_{l_{2}}V_{\nu}
\|_{L^{2}(\mathbb{R}^{2}\times\mathbb{T})}\Big)
\\&\lesssim 2^{-N(k_{1}+l_{1})/2}2^{-9(k_{2}+l_{2})}\|V_{+}(s)\|_{Z_{J}}.
\end{aligned}
\end{equation}
Similarly,
\begin{equation}\label{5-38AA}
I_{6}
\lesssim 2^{-N(k_{2}+l_{2})/2}2^{-9(k_{1}+l_{1})}\|V_{+}(s)\|_{Z_{J}}.
\end{equation}
Combining \eqref{5-37AA} and \eqref{5-38AA} yields
\begin{equation}\label{5-39AA}\begin{aligned}
I_{6}&\lesssim 2^{-N(k_{1}+l_{1})/4}2^{-9(k_{2}+l_{2})/2}\|V_{+}(s)\|^{1/2}_{Z_{J}}
2^{-N(k_{2}+l_{2})/4}2^{-9(k_{1}+l_{1})/2}\|V_{+}(s)\|^{1/2}_{Z_{J}}
\\&= 2^{-N(k_{1}+k_{2})/4}2^{-N(l_{1}+l_{2})/4}2^{-9(k_{1}+k_{2})/2}2^{-9(l_{1}+l_{2})/2}\|V_{+}(s)\|_{Z_{J}}.
\end{aligned}
\end{equation}
It follows from \eqref{5-35A}-\eqref{5-35AB} and \eqref{5-39AA} that
\begin{equation}\label{5-40A}
\begin{aligned}
&\sum\limits_{j\geq 100}
\sum_{\mbox{\tiny$\begin{array}{c}
(k_{1},k_{2})\in\chi_{k},(l_{1},l_{2})\in\chi_{l}\end{array}$}}
2^{\mathrm{min}(j,2J-j)}\|\varphi_{j}(x)
\cdot N_{s,k,k_{1},k_{2},l,l_{1},l_{2}}^{\mu,\nu,j}\|_{L^{2}(\mathbb{R}^{2}\times\mathbb{T})}\\
\lesssim& \sum\limits_{j\geq 100}
\sum_{\mbox{\tiny$\begin{array}{c}
(k_{1},k_{2})\in\chi_{k},(l_{1},l_{2})\in\chi_{l}\end{array}$}}
2^{\mathrm{min}(j,2J-j)}\|N_{s,k,k_{1},k_{2},l,l_{1},l_{2}}^{\mu,\nu,j}\|_{L^{2}(\mathbb{R}^{2}\times\mathbb{T})}\\
\lesssim& \sum_{\mbox{\tiny$\begin{array}{c}
(k_{1},k_{2})\in\chi_{k},(l_{1},l_{2})\in\chi_{l}\end{array}$}}
(2^{k_{2}}+2^{l_{2}})2^{k_{1}}2^{l_{1}/2}I_{6}\\
\lesssim& \sum_{\mbox{\tiny$\begin{array}{c}
(k_{1},k_{2})\in\chi_{k},(l_{1},l_{2})\in\chi_{l}\end{array}$}}
2^{-N(k_{1}+k_{2})/4}2^{-N(l_{1}+l_{2})/4}\|V_{+}(s)\|_{Z_{J}}\\
\leq& C 2^{-9(k+l)}\|V_{+}(s)\|_{Z_{J}}.
\end{aligned}
\end{equation}
Thus, the proof of  \eqref{5-ABCD} is obtained from \eqref{5-AAA1}, \eqref{5-33A} and \eqref{5-40A}. In view
of \eqref{5-ABC} and \eqref{5-ABCD}, we complete the proof of \eqref{5-AB} and further Proposition 4.1 is established.

\section{\textbf{Proof of Proposition 4.2}}

\subsection {Reformulation of $U$}

As in deriving \eqref{5-8} and \eqref{5-10A}, we have
\begin{equation}\label{6-1}
\begin{aligned}
\frac{d}{dt}[\widehat{V_{+;n}}(t,\xi)]=\sum_{\mu,\nu\in\{+,-\}}\int_{\mathbb{R}^{2}}\sum_{m\in \mathbb{Z}}
e^{is \Phi_{\mu,\nu}^{n,m}(\xi,\eta)}\mathcal{M}_{n,m}(\xi,\eta)\widehat{V_{\mu;n-m}}(t,\xi-\eta)\widehat{V_{\nu;m}}(t,\eta)d\eta
\end{aligned}
\end{equation}
and
\begin{equation}\label{6-2}
\begin{aligned}
&\widehat{V_{+;n}}(t,\xi)-\widehat{V_{+;n}}(0,\xi)\\
=& \sum_{\mu,\nu\in\{+,-\}}\int_{0}^{t}\int_{\mathbb{R}^{2}}\sum_{m\in \mathbb{Z}}
e^{is \Phi_{\mu,\nu}^{n,m}(\xi,\eta)}\mathcal{M}_{n,m}(\xi,\eta)
\widehat{V_{\mu;n-m}}(s,\xi-\eta)\widehat{V_{\nu;m}}(s,\eta)d\eta ds\\
=& \sum_{\mu,\nu\in\{+,-\}}-\int_{0}^{t}\int_{\mathbb{R}^{2}}\sum_{m\in \mathbb{Z}}
e^{is \Phi_{\mu,\nu}^{n,m}(\xi,\eta)}\frac{\mathcal{M}_{n,m}(\xi,\eta)}{i \Phi_{\mu,\nu}^{n, m}(\xi,\eta)}
\frac{d}{ds}\Big[\widehat{V_{\mu;n-m}}(s,\xi-\eta)\widehat{V_{\nu;m}}(s,\eta)\Big]d\eta ds\\
&~~~~~~~~~~~~~~~+\sum_{\mu,\nu\in\{+,-\}}\int_{\mathbb{R}^{2}}\sum_{m\in \mathbb{Z}}
e^{it \Phi_{\mu,\nu}^{n,m}(\xi,\eta)}\frac{\mathcal{M}_{n,m}(\xi,\eta)}{i \Phi_{\mu,\nu}^{n,m}(\xi,\eta)}
\Big[\widehat{V_{\mu;n-m}}(t,\xi-\eta)\widehat{V_{\nu;m}}(t,\eta)\Big]d\eta\\
&~~~~~~~~~~~~~~~-\sum_{\mu,\nu\in\{+,-\}}\int_{\mathbb{R}^{2}}\sum_{m\in \mathbb{Z}}
\frac{\mathcal{M}_{n,m}(\xi,\eta)}{i \Phi_{\mu,\nu}^{n,m}(\xi,\eta)}
\Big[\widehat{V_{\mu;n-m}}(0,\xi-\eta)\widehat{V_{\nu;m}}(0,\eta)\Big]d\eta.
\end{aligned}
\end{equation}

By introducing
\begin{equation}\label{6-3A}
\begin{aligned}
&\widehat{W_{n}}(t,\xi)\\
:=&\widehat{V_{+;n}}(t,\xi)-\sum_{\mu,\nu\in\{+,-\}}\int_{\mathbb{R}^{2}}\sum_{m\in \mathbb{Z}}
e^{it\Phi_{\mu,\nu}^{n,m}(\xi,\eta)}
\frac{\mathcal{M}_{n,m}(\xi,\eta)}{i\Phi_{\mu,\nu}^{n,m}(\xi,\eta)}
\Big[\widehat{V_{\mu;n-m}}(t,\xi-\eta)\widehat{V_{\nu;m}}(t,\eta)\Big]d\eta\\
:=&\widehat{V_{+, n}}(t, \xi)-\sum\limits_{\mu, \nu\in\{+, -\}}\mathcal{F}_{x, y}(Q_{\mu\nu})(t, \xi, n),
\end{aligned}
\end{equation}
it follows from \eqref{6-2} and \eqref{6-3A} that
\begin{equation}\label{6-4AAA}
\begin{aligned}
\widehat{W_{n}}(t,\xi)-\widehat{W_{n}}(0,\xi)=\sum_{\mu,\nu\in\{+,-\}}i\int_{0}^{t}H_{\mu\nu}(s,\xi,n)ds,
\end{aligned}
\end{equation}
where
\begin{equation}\label{6-4AAB}
\begin{aligned}
H_{\mu\nu}(s,\xi,n):=\int_{\mathbb{R}^{2}}\sum_{m\in \mathbb{Z}}
e^{is\Phi_{\mu,\nu}^{n,m}(\xi,\eta)}\frac{\mathcal{M}_{n,m}(\xi,\eta)}{\Phi_{\mu,\nu}^{n,m}(\xi,\eta)}
\frac{d}{ds}\Big[\widehat{V_{\mu;n-m}}(s,\xi-\eta)\widehat{V_{\nu;m}}(s,\eta)\Big]d\eta.
\end{aligned}
\end{equation}
In addition, by \eqref{6-1}, \eqref{2-9} and \eqref{2-11}, one has
\begin{equation}\label{6-6AAA}
\begin{aligned}
\frac{d}{ds}[\widehat{V_{\mu;n}}(s,\xi)]=\sum_{\lambda,\omega\in\{+,-\}}e^{is\Lambda_{\mu;n}(\xi)}\int_{\mathbb{R}^{2}}\sum_{l\in \mathbb{Z}}
\mathcal{M}_{n,l}(\xi,\zeta)\widehat{U_{\lambda;n-l}}(s,\xi-\zeta)\widehat{U_{\omega;l}}(s,\zeta)d\zeta.
\end{aligned}
\end{equation}
Plugging \eqref{6-6AAA} into \eqref{6-4AAB} yields
\begin{equation}\label{6-7AB}
\begin{aligned}
&H_{\mu\nu}(s,\xi,n)\\
=& \sum_{\lambda,\omega\in\{+,-\}}e^{is\Lambda_{n}(\xi)}\sum_{m\in \mathbb{Z}}\sum_{l\in \mathbb{Z}}\int_{\mathbb{R}^{2}}
\int_{\mathbb{R}^{2}}\frac{\mathcal{M}_{n,n-m}(\xi,\xi-\eta)}{\Phi_{\mu,\nu}^{n,n-m}(\xi,\xi-\eta)}\\
&~~~~~~~~~~~~~~~~~~~~~~~~~~\cdot\mathcal{M}_{m,l}(\eta,\zeta)
\widehat{U_{\nu;n-m}}(s,\xi-\eta)\widehat{U_{\lambda;m-l}}(s,\eta-\zeta)\widehat{U_{\omega;l}}(s,\zeta)d\zeta d\eta\\
&+\sum_{\lambda,\omega\in\{+,-\}}e^{is\Lambda_{n}(\xi)}\sum_{m\in \mathbb{Z}}\sum_{l\in \mathbb{Z}}\int_{\mathbb{R}^{2}}
\int_{\mathbb{R}^{2}}\frac{\mathcal{M}_{n,m}(\xi,\eta)}{\Phi_{\mu,\nu}^{n,m}(\xi,\eta)}\\
&~~~~~~~~~~~~~~~~~~~~~~~~~~\cdot\mathcal{M}_{m,l}(\eta,\zeta)
\widehat{U_{\mu;n-m}}(s,\xi-\eta)\widehat{U_{\lambda;m-l}}(s,\eta-\zeta)\widehat{U_{\omega;l}}(s,\zeta)d\zeta d\eta.
\end{aligned}
\end{equation}
It follows from \eqref{6-7AB} and \eqref{6-4AAA} that
\begin{equation}\label{6-8ABC}
\begin{aligned}
&\widehat{W_{n}}(t,\xi)-\widehat{W_{n}}(0,\xi)\\
=& \sum_{\mu,\sigma,\lambda,\omega\in\{+,-\}}e^{is\Lambda_{n}(\xi)}\sum_{m\in \mathbb{Z}}
\sum_{l\in \mathbb{Z}}i\int_{0}^{t}\int_{\mathbb{R}^{2}}
\int_{\mathbb{R}^{2}}\Big(\frac{\mathcal{M}_{n,m}(\xi,\eta)}{\Phi_{\sigma,\mu}^{n,m}(\xi,\eta)}+
\frac{\mathcal{M}_{n,n-m}(\xi,\xi-\eta)}{\Phi_{\mu,\sigma}^{n,n-m}(\xi,\xi-\eta)}\Big)\\
&~~~~~~~~~~~~~~~~~~~~~~\cdot\mathcal{M}_{m,l}(\eta,\zeta)
\widehat{U_{\sigma;n-m}}(s,\xi-\eta)\widehat{U_{\lambda;m-l}}(s,\eta-\zeta)\widehat{U_{\omega;l}}(s,\zeta)d\zeta d\eta.
\end{aligned}
\end{equation}

\subsection{Preliminary for the proof of Proposition \ref{prop4-2}}

To prove Proposition \ref{prop4-2}, by \eqref{2-8} and \eqref{6-3A}, we decompose $e^{it\Lambda}U(t)-U(0)$ as
\begin{equation*}\begin{aligned}
&e^{it\Lambda}U(t)-U(0)\\[2mm]
=&(e^{it\Lambda}U(t)-W(t))-(e^{it\Lambda}U(t)-W(t))|_{t=0}+(W(t)-W(0))\\[2mm]
=&\sum\limits_{\mu, \nu\in\{+, -\}}Q_{\mu \nu}(t, x, y)-\sum\limits_{\mu, \nu\in\{+, -\}}Q_{\mu \nu}(0, x, y)+W(t)-W(0).
\end{aligned}
\end{equation*}
Hence, Proposition 4.2 can follow from Lemma 6.1 and Lemma 6.2 below:

\begin{lem}\label{lem6-1}
{\it For $T_{0}> 0$ and $\mu, \nu\in \{+, -\}$, we have
\begin{equation}\label{6-19AAAB}
\sup\limits_{t\in [0, T_0]}\|Q_{\mu \nu}(t)\|_{\mathcal{H}^{N_0}}:=\sup\limits_{t\in [0,T_{0}]}\|Q_{\mu\nu}(t)\|_{H^{N_{0}}(\mathbb{R}^{2}\times \mathbb{T})}
+\sup\limits_{t\in [0,T_{0}]}\|Q_{\mu\nu}(t)\|_{Z}\lesssim\varepsilon_{0}^{2},
\end{equation}}
\end{lem}

\begin{lem}\label{lem6-2}
{\it With $T_{0}=e^{c_{0}/\epsilon_{0}^{2}}$, we have
\begin{equation*}
\sup\limits_{t\in [0,T_{0}]}\|W(t)-W(0)\|_{\mathcal{H}^{N_{0}}}\lesssim \varepsilon_{0}^{3}+c_{0}\varepsilon_{0}.
\end{equation*}}
\end{lem}

In addition, combining the assumption \eqref{3-3} in Proposition \ref{prop4-2} with Lemma \ref{HC-5} yields
\begin{subequations}\label{6-2-0}\begin{align}
&~~~~~~~~~~~~~~~~~~~~~~~\|R_{k}S_{l}U_{\pm}(t)\|_{L^{\infty}(\mathbb{R}^{2}\times\mathbb{T})}
\lesssim\varepsilon_{0}(1+t)^{-1}2^{-7(k+l)},\label{6-10A}\\[2mm]
&~~~~\|\mathcal{F}_{x,y}(R_{k}S_{l}U_{\pm}(t,x,y))(\xi,n)\|_{L^{\infty}_{\xi,n}}
=\|\mathcal{F}_{x,y}(R_{k}S_{l}V_{\pm}(t,x,y))(\xi,n)\|_{L^{\infty}_{\xi,n}}
\lesssim \varepsilon_{0}2^{-9(k+l)},\label{6-11A}\\[2mm]
&\|R_{k}S_{l}U_{\pm}(t)\|_{L^{2}(\mathbb{R}^{2}\times\mathbb{T})}
+\|R_{k}S_{l}V_{\pm}(t)\|_{L^{2}(\mathbb{R}^{2}\times\mathbb{T})}
\lesssim \varepsilon_{0}\mathrm{min}
(2^{-N_{0}(k+l)/2},\  (1+t)^{\kappa}2^{-N(k+l)/2}).\label{6-12A}
\end{align}
\end{subequations}

\subsection{Proof of Lemma \ref{lem6-1}}

\begin{proof} Due to \eqref{6-3A}, \eqref{2-11} and \eqref{2-9},  we write
\begin{equation}\label{6-20A}
\mathcal{F}_{x,  y}(Q_{\mu\nu})(t, \xi, n)= e^{it\Lambda_{n}(\xi)}\int_{\mathbb{R}^{2}}\sum_{m\in \mathbb{Z}}
\frac{\mathcal{M}_{n, m}(\xi,\eta)}{\Phi_{\mu, \nu}^{n,\ m}(\xi, \eta)}
\widehat{U_{\mu; n-m}}(t, \xi-\eta)\widehat{U_{\nu; m}}(t,\eta)d\eta.
\end{equation}

The proof of \eqref{6-19AAAB} will be divided into the following two steps.\vskip 0.1cm

\textbf{Step 1: Establish}
\begin{equation}\label{6-2-1}
\sup\limits_{t\in [0,T_{0}]}\|Q_{\mu\nu}(t)\|_{H^{N_{0}}(\mathbb{R}^{2}\times \mathbb{T})}\lesssim\varepsilon_{0}^{2}.
\end{equation}

\vskip 0.1 true cm

For $k, l\geq -1$ and $t\in [0,T_{0}]$, in view of \eqref{6-20A} and \eqref{2-3333}, one has
\begin{equation}\label{6-17A}
\begin{aligned}
&\mathcal{F}_{x,y}(R_{k}S_{l}Q_{\mu\nu})(t,\xi,n)\\
=& \psi_{k}(\xi)\psi_{l}(n)\cdot e^{it\Lambda_{n}(\xi)}\int_{\mathbb{R}^{2}}\sum_{m\in \mathbb{Z}}
\frac{\mathcal{M}_{n,m}(\xi,\eta)}{\Phi_{\mu,\nu}^{n,m}(\xi,\eta)}
\widehat{U_{\mu;n-m}}(t,\xi-\eta)\widehat{U_{\nu;m}}(t,\eta)d\eta\\
=& \sum_{\mbox{\tiny$\begin{array}{c}(k_{1},k_{2})\in\chi_{k},(l_{1},l_{2})\in\chi_{l}\end{array}$}}
\psi_{k}(\xi)\psi_{l}(n)\cdot e^{it\Lambda_{n}(\xi)}\int_{\mathbb{R}^{2}}\sum_{m\in \mathbb{Z}}
\frac{\mathcal{M}_{n,m}(\xi,\eta)}{\Phi_{\mu,\nu}^{n,m}(\xi,\eta)}\\
&~~~~~~~~~~~
~~~~~~~~~~~~~\cdot\psi_{k_{1}}(\xi-\eta)\psi_{l_{1}}(n-m)\widehat{U_{\mu;n-m}}(t,\xi-\eta)
\psi_{k_{2}}(\eta)\psi_{l_{2}}(m)\widehat{U_{\nu;m}}(t,\eta)d\eta\\
=& \sum_{\mbox{\tiny$\begin{array}{c}(k_{1},k_{2})\in\chi_{k},(l_{1},l_{2})\in\chi_{l}\end{array}$}}
e^{it\Lambda_{n}(\xi)}\int_{\mathbb{R}^{2}}\sum_{m\in \mathbb{Z}}
P^{\mu\nu}_{k,k_{1},k_{2},l,l_{1},l_{2}}(\xi,n,\eta,m)\\
&~~~~~~~~~~~
~~~~~~~~~~~~~~~~~~~~~~~~~\cdot\mathcal{F}_{x,y}(R_{k_{1}}S_{l_{1}}U_{\mu})(t,\xi-\eta,n-m)
\mathcal{F}_{x,y}(R_{k_{2}}S_{l_{2}}U_{\nu})(t,\eta,m)d\eta,
\end{aligned}
\end{equation}
where
\begin{equation}\label{6-18AAA}
\begin{aligned}
&P^{\mu\nu}_{k,k_{1},k_{2},l,l_{1},l_{2}}(\xi,n,\eta,m)\\
:=& \frac{\mathcal{M}_{n,m}(\xi,\eta)}{\Phi_{\mu,\nu}^{n,m}(\xi,\eta)}
\psi_{k}(\xi)\psi_{[k_{1}-2,k_{1}+2]}(\xi-\eta)\psi_{[k_{2}-2,k_{2}+2]}(\eta)\psi_{l}(n)
\psi_{[l_{1}-2,l_{1}+2]}(n-m)\psi_{[l_{2}-2,l_{2}+2]}(m).
\end{aligned}
\end{equation}

Under the assumption \eqref{3-3}, applying Lemma \ref{HC-8} in Appendix A to \eqref{6-17A} with $r=\infty$
and utilizing \eqref{6-12A} yield
\begin{equation}\label{6-21ABCDF}
\begin{aligned}
&\|R_{k}S_{l}Q_{\mu\nu}(t)\|_{L^{2}(\mathbb{R}^{2}\times\mathbb{T})}
=\|\mathcal{F}_{x,y}(R_{k}S_{l}Q_{\mu\nu})(t,\xi,n)\|_{L^{2}_{\xi}l^{2}_{n}}\\
\lesssim& \sum_{\mbox{\tiny$\begin{array}{c}(k_{1},k_{2})\in\chi_{k},(l_{1},l_{2})\in\chi_{l}\end{array}$}}
2^{l+l_{2}}(2^{k}+2^{l})(2^{k_{1}}+2^{k_{2}}+2^{l_{2}})2^{(k+l/2)}\\
&~~~~~~~~~~~~~~~~~~~~~~~~~~~
~~~~~~~~\cdot\|R_{k_{1}}S_{l_{1}}U_{\mu}\|_{L^{2}(\mathbb{R}^{2}\times \mathbb{T})}
\|R_{k_{2}}S_{l_{2}}U_{\nu}\|_{L^{2}(\mathbb{R}^{2}\times \mathbb{T})}\\
\lesssim& \sum_{\mbox{\tiny$\begin{array}{c}(k_{1},k_{2})\in\chi_{k},(l_{1},l_{2})\in\chi_{l}\end{array}$}}
2^{2k+k_{1}+k_{2}}2^{3l+2l_{2}}
\varepsilon_{0}(1+t)^{\kappa}2^{-N(k_{1}+l_{1})/2}
\varepsilon_{0}(1+t)^{\kappa}2^{-N(k_{2}+l_{2})/2}\\
\lesssim& \varepsilon_{0}^{2}(1+t)^{2\kappa}2^{-(N/2-5)(k+l)}.
\end{aligned}
\end{equation}

Similarly, by the assumption \eqref{3-3}, Lemma \ref{HC-8} with $r=2$, \eqref{6-10A} and \eqref{6-12A}, we obtain
\begin{equation}\label{6-23AB}
\begin{aligned}
&\|R_{k}S_{l}Q_{\mu\nu}(t)\|_{L^{2}(\mathbb{R}^{2}\times\mathbb{T})}\\
\lesssim& \sum_{\mbox{\tiny$\begin{array}{c}(k_{1},k_{2})\in\chi_{k},(l_{1},l_{2})\in\chi_{l}\end{array}$}}
2^{l+l_{2}}(2^{k}+2^{l})(2^{k_{1}}+2^{k_{2}}+2^{l_{2}})\\
&~~~\cdot\mathrm{min}\Big(\|R_{k_{1}}S_{l_{1}}U_{\mu})\|_{L^{2}(\mathbb{R}^{2}\times \mathbb{T})}
\|R_{k_{2}}S_{l_{2}}U_{\nu}
\|_{L^{\infty}(\mathbb{R}^{2}\times \mathbb{T})}, \|R_{k_{1}}S_{l_{1}}U_{\mu})\|_{L^{\infty}(\mathbb{R}^{2}\times \mathbb{T})}\|R_{k_{2}}S_{l_{2}}U_{\nu}\|_{L^{2}(\mathbb{R}^{2}\times \mathbb{T})}\Big)\\
\lesssim& \sum_{\mbox{\tiny$\begin{array}{c}(k_{1},k_{2})\in\chi_{k},(l_{1},l_{2})\in\chi_{l}\end{array}$}}
2^{k+k_{1}+k_{2}}2^{2l+2l_{2}}
\mathrm{min}\Big(\varepsilon_{0}2^{-N_{0}(k_{1}+l_{1})/2}
\varepsilon_{0}(1+t)^{-1}2^{-7(k_{2}+l_{2})},\\
&~~~~~~~~~~~~~~~~~~~~~~~~~~~~~~~~~~~~
~~~~~~~~~~~~~~~~~~~~~~~~~\varepsilon_{0}2^{-N_{0}(k_{2}+l_{2})/2}
\varepsilon_{0}(1+t)^{-1}2^{-7(k_{1}+l_{1})}\Big)\\
\lesssim& \sum_{\mbox{\tiny$\begin{array}{c}(k_{1},k_{2})\in\chi_{k},(l_{1},l_{2})\in\chi_{l}\end{array}$}}
2^{k+k_{1}+k_{2}}2^{2l+2l_{2}}\cdot\Big[\varepsilon_{0}2^{-N_{0}(k_{1}+l_{1})/2}
\varepsilon_{0}(1+t)^{-1}2^{-7(k_{2}+l_{2})}\Big]^{1/2}\\
&~~~~~~~~~~~~~~~~~~~~~~~~~~~~
~~~~~~~~~~~~~~~~~~~~~~~\cdot\Big[\varepsilon_{0}2^{-N_{0}(k_{2}+l_{2})/2}
\varepsilon_{0}(1+t)^{-1}2^{-7(k_{1}+l_{1})}\Big]^{1/2}\\
\lesssim& \varepsilon_{0}^{2}(1+t)^{-1}2^{-(N_{0}/4-1)(k+l)}.
\end{aligned}
\end{equation}
Hence, it derives from \eqref{6-21ABCDF} and \eqref{6-23AB} that
\begin{equation}\label{6-25AB}
\begin{aligned}
&\|R_{k}S_{l}Q_{\mu\nu}(t)\|_{L^{2}(\mathbb{R}^{2}\times\mathbb{T})}\\
\lesssim& \Big[\varepsilon_{0}^{2}(1+t)^{2\kappa}2^{-(N/2-5)(k+l)}\Big]^{(N-12)/(N-10)}
\Big[\varepsilon_{0}^{2}(1+t)^{-1}2^{-(N_{0}/4-1)(k+l)}\Big]^{2/(N-10)}\\
\lesssim& \varepsilon_{0}^{2}2^{-(N/2-6)(k+l)}.
\end{aligned}
\end{equation}
Moreover,
\begin{equation}\label{6-26AB}
\begin{aligned}
&\|Q_{\mu\nu}(t)\|_{H^{N_{0}}(\mathbb{R}^{2}\times \mathbb{T})}\\
\lesssim& \Big(\sum_{k,l\geq-1}\Big\|\Big(\sqrt{1+|\xi|^{2}+n^{2}}\Big)^{N_{0}}
\mathcal{F}_{x,y}(R_{k}S_{l}Q_{\mu\nu})(t,\xi,n)\Big\|_{L^{2}_{\xi}l^{2}_{n}}^{2}\Big)^{1/2}\\
\lesssim& \Big(\sum_{k,l\geq-1}(2^{N_{0}k}+2^{N_{0}l})^{2}
\Big\|\mathcal{F}_{x,y}(R_{k}S_{l}Q_{\mu\nu})(t,\xi,n)\Big\|_{L^{2}_{\xi}l^{2}_{n}}^{2}\Big)^{1/2}\\
=& \Big(\sum_{k,l\geq-1}(2^{N_{0}k}+2^{N_{0}l})^{2}
\|R_{k}S_{l}Q_{\mu\nu}(t)\|_{L^{2}(\mathbb{R}^{2}\times\mathbb{T})}^{2}\Big)^{1/2}.
\end{aligned}
\end{equation}
Combining \eqref{6-25AB} and \eqref{6-26AB} derives \eqref{6-2-1}.
\vskip 0.2 true cm

\textbf{Step 2: Establish}
\begin{equation}\label{6-2-2}
\|Q_{\mu\nu}(t)\|_{Z}\lesssim\varepsilon_{0}^{2}.
\end{equation}

\vskip 0.1 true cm

By the definition of $Z$-norm in \eqref{2-1A}, then
\begin{equation}\label{6-26D}
\|Q_{\mu\nu}(t)\|_{Z}:=
\sup\limits_{k,l\geq -1}2^{9(k+l)}\Big[\|R_{k}S_{l}Q_{\mu\nu}(t)\|_{L^{2}(\mathbb{R}^{2}\times\mathbb{T})}
+\sum\limits_{j\in \mathbb{Z}^{+}}2^{j}\|\varphi_{j}(x)\cdot R_{k}S_{l}Q_{\mu\nu}(t)\|_{L^{2}(\mathbb{R}^{2}\times\mathbb{T})}\Big].
\end{equation}
Using \eqref{6-25AB}, we have
\begin{equation}\label{6-28D}
\sup\limits_{t\in [0,T_{0}]}\sup\limits_{k,l\geq -1}2^{9(k+l)}\|R_{k}S_{l}Q_{\mu\nu}(t)
\|_{L^{2}(\mathbb{R}^{2}\times\mathbb{T})}\lesssim \varepsilon_{0}^{2}.
\end{equation}
Therefore, with \eqref{6-26D} and \eqref{6-28D}, the proof of \eqref{6-2-2} is reduced  to show that
for $k,l\geq -1$ and $t\in [0,T_{0}]$,
\begin{equation}\label{6-27D}
2^{9(k+l)}\sum\limits_{j\in \mathbb{Z}^{+}}2^{j}
\|\varphi_{j}(x)\cdot R_{k}S_{l}Q_{\mu\nu}(t)\|_{L^{2}(\mathbb{R}^{2}\times\mathbb{T})}\lesssim \varepsilon_{0}^{2}.
\end{equation}

Due to \eqref{6-23AB}, then
\begin{equation}\label{6-30AD}
2^{9(k+l)}\sum\limits_{j\leq J_{k,l,t}}2^{j}
\|\varphi_{j}(x)\cdot R_{k}S_{l}Q_{\mu\nu}(t)\|_{L^{2}(\mathbb{R}^{2}\times\mathbb{T})}\lesssim \varepsilon_{0}^{2},
\end{equation}
where
\begin{equation}\label{6-31D}
2^{J_{k,l,t}}:=100(1+t)(2^{k}+2^{l}).
\end{equation}

Based on \eqref{6-17A}-\eqref{6-18AAA}, we decompose $R_k S_l Q_{\mu\nu}$ as
\begin{equation}\label{6-32AA}
R_{k}S_{l}Q_{\mu\nu}(t,x,y)=\sum_{\mbox{\tiny$\begin{array}{c}(k_{1},k_{2})\in\chi_{k},(l_{1},l_{2})\in\chi_{l}\end{array}$}}
W_{k,k_{1},k_{2},l,l_{1},l_{2}}^{\mu,\nu}(t,x,y),
\end{equation}
where
\begin{equation*}\begin{aligned}
&\mathcal{F}_{x,y}(W_{k,k_{1},k_{2},l,l_{1},l_{2}}^{\mu,\nu})(t,\xi,n)\\
:=&\int_{\mathbb{R}^{2}}\sum_{m\in \mathbb{Z}}
e^{is \Phi_{\mu,\nu}^{n,m}(\xi,\eta)}
P^{\mu\nu}_{k,k_{1},k_{2},l,l_{1},l_{2}}(\xi,n,\eta,m)\times\\
&~~~~~~~~\psi_{k_{1}}(\xi-\eta)\psi_{l_{1}}(n-m)\widehat{V_{\mu;n-m}}(s,\xi-\eta)
\psi_{k_{2}}(\eta)\psi_{l_{2}}(m)\widehat{V_{\nu;m}}(s,\eta)d\eta.
\end{aligned}
\end{equation*}

In view of \eqref{6-30AD} and \eqref{6-32AA},  we can also reduce the proof of \eqref{6-27D} into the following estimate
\begin{equation}\label{6-35AA}
2^{9(k+l)}\sum\limits_{j\geq J_{k,l,t}}\sum_{\mbox{\tiny$\begin{array}{c}(k_{1},k_{2})\in\chi_{k},(l_{1},l_{2})\in\chi_{l}\end{array}$}}
2^{j}\|\varphi_{j}(x)\cdot W_{k,k_{1},k_{2},l,l_{1},l_{2}}^{\mu,\nu}\|_{L^{2}(\mathbb{R}^{2}\times\mathbb{T})}
\lesssim \varepsilon_{0}^{2}.
\end{equation}

To obtain \eqref{6-35AA}, we write $W_{k,k_{1},k_{2},l,l_{1},l_{2}}^{\mu,\nu}(t,x,y)$ as
\begin{equation}\label{6-36AA}
\begin{aligned}
W_{k,k_{1},k_{2},l,l_{1},l_{2}}^{\mu,\nu}(t,x,y)=
\int_{\mathbb{R}^{2}\times\mathbb{T}}\int_{\mathbb{R}^{2}\times\mathbb{T}}&L^{\mu,\nu}_{k,k_{1},k_{2},l,l_{1},l_{2}}(t,x,y,x_{1},y_{1},x_{2},y_{2})
R_{k_{1}}S_{l_{1}}V_{\mu}(t,x_{1},y_{1})
\\&~~~~~~~~~~~~~~~~~~~~~~~~\cdot R_{k_{2}}S_{l_{2}}V_{\nu}(t,x_{2},y_{2})dx_{1}dy_{1}dx_{2}dy_{2},
\end{aligned}
\end{equation}
where
\begin{equation}\label{6-37AA}
\begin{aligned}
&L^{\mu,\nu}_{k,k_{1},k_{2},l,l_{1},l_{2}}(t,x,y,x_{1},y_{1},x_{2},y_{2})
:=\sum_{n\in \mathbb{Z}}\sum_{m\in \mathbb{Z}}\Big(e^{i[n\cdot (y-y_{1})+m\cdot(y_{1}-y_{2})]}\Big)
\\&~~~~~~~~~~~~~~\cdot\int_{\mathbb{R}^{2}\times\mathbb{R}^{2}}e^{i[(x-x_{1})\cdot\xi+(x_{1}-x_{2})\cdot\eta]}
e^{it \Phi_{\mu,\nu}^{n,m}(\xi,\eta)}P^{\mu\nu}_{k,k_{1},k_{2},l,l_{1},l_{2}}(\xi,n,\eta,m)d\xi d\eta.
\end{aligned}
\end{equation}

Along the method of non-stationary phase, define the operator $L$ as
\begin{equation}\label{6-38AA}
\begin{aligned}
L:=\frac{\sum\limits_{d=1}^{2}\Big(t\partial_{\xi_{d}} \Phi_{\mu,\nu}^{n,m}(\xi,\eta)+x^{d}-x_{1}^{d}\Big)\partial_{\xi_{d}}
+\sum\limits_{d=1}^{2}\Big(t\partial_{\eta_{d}} \Phi_{\mu,\nu}^{n,m}(\xi,\eta)+x_{1}^{d}-x_{2}^{d}\Big)\partial_{\eta_{d}}}
{i\sum\limits_{d=1}^{2}\Big(t\partial_{\xi_{d}} \Phi_{\mu,\nu}^{n,m}(\xi,\eta)+x^{d}-x_{1}^{d}\Big)^{2}
+i\sum\limits_{d=1}^{2}\Big(t\partial_{\eta_{d}} \Phi_{\mu,\nu}^{n,m}(\xi,\eta)+x_{1}^{d}-x_{2}^{d}\Big)^{2}},
\end{aligned}
\end{equation}
then a direct computation shows
\begin{equation}\label{6-39AA}
\begin{aligned}
L(e^{i[(x-x_{1})\cdot\xi+(x_{1}-x_{2})\cdot\eta]}e^{it \Phi_{\mu,\nu}^{n,m}(\xi,\eta)})
=e^{i[(x-x_{1})\cdot\xi+(x_{1}-x_{2})\cdot\eta]}e^{it \Phi_{\mu,\nu}^{n,m}(\xi,\eta)}.
\end{aligned}
\end{equation}

With the notations in \eqref{6-31D} and \eqref{2-11}, we have
\begin{equation*}\label{6-40AA}
2^{j}\geq 2^{J_{k,l,t}}=100(1+t)(2^{k}+2^{l})~~(j\ge J_{k,l,t})
\end{equation*}
and
\begin{equation*}
t|\partial_{\xi_{d}} \Phi_{\mu,\nu}^{n,m}(\xi,\eta)|+t|\partial_{\eta_{d}} \Phi_{\mu,\nu}^{n,m}(\xi,\eta)|\leq 4(1+t)~~ (d=1, 2).
\end{equation*}
Combining these two inequalities yields that for $|x-x_{1}|+|x_{1}-x_{2}|\geq 2^{j-10}$ and $j\geq J_{k, l, t}$,
\begin{equation}\label{6-41AA}
\begin{aligned}
&\sum\limits_{d=1}^{2}\Big(t\partial_{\xi_{d}} \Phi_{\mu,\nu}^{n,m}(\xi,\eta)+x^{d}-x_{1}^{d}\Big)^{2}
+\sum\limits_{d=1}^{2}\Big(t\partial_{\eta_{d}} \Phi_{\mu,\nu}^{n,m}(\xi,\eta)+x_{1}^{d}-x_{2}^{d}\Big)^{2}\\
\gtrsim& \Big(|x-x_{1}|+|x_{1}-x_{2}|\Big)^{2}.
\end{aligned}
\end{equation}

Therefore, when $|x-x_{1}|+|x_{1}-x_{2}|\geq 2^{j-10}$ and $j\geq J_{k, l, t}$, it follows
from \eqref{6-37AA}-\eqref{6-41AA} and the integration by parts that
\begin{equation}\label{6-42A}
\begin{aligned}
&\Big|L^{\mu,\nu}_{k,k_{1},k_{2},l,l_{1},l_{2}}(t,x,y,x_{1},y_{1},x_{2},y_{2})\Big|\\
\lesssim& 2^{l}2^{l_{2}}\sup\limits_{n,m\in \mathbb{Z}} \Big|\int_{\mathbb{R}^{2}\times\mathbb{R}^{2}}e^{i[(x-x_{1})\cdot\xi+(x_{1}-x_{2})\cdot\eta]}
e^{it \Phi_{\mu,\nu}^{n,m}(\xi,\eta)}P^{\mu\nu}_{k,k_{1},k_{2},l,l_{1},l_{2}}(\xi,n,\eta,m)d\xi d\eta\Big|\\
=& 2^{l}2^{l_{2}}\sup\limits_{n,m\in \mathbb{Z}} \Big|\int_{\mathbb{R}^{2}\times\mathbb{R}^{2}}L^{(11)}(e^{i[(x-x_{1})\cdot\xi+(x_{1}-x_{2})
\cdot\eta]}e^{it \Phi_{\mu,\nu}^{n,m}(\xi,\eta)})\\
&~~~~~~~~~~~~~~~~~~~~~~~~~~~~~~~~~~~~~~~~~~
~~~~~~~~~~~~~~~~~\cdot P^{\mu\nu}_{k,k_{1},k_{2},l,l_{1},l_{2}}(\xi,n,\eta,m)d\xi d\eta\Big|\\
\lesssim& 2^{l}2^{l_{2}}(2^{k_{2}}+2^{l_{2}})(2^{k}+2^{l})2^{2(k_{1}+k_{2})}(|x-x_{1}|+|x_{1}-x_{2}|)^{-11}\\
\lesssim& 2^{l}2^{l_{2}}(2^{k_{2}}+2^{l_{2}})2^{2(k_{1}+k_{2})}(|x-x_{1}|+|x_{1}-x_{2}|)^{-10},
\end{aligned}
\end{equation}
here we have used \eqref{5-7-0}-\eqref{5-7-1}, \eqref{7-16BAB} and Lemma \ref{HC-11} as well as \eqref{6-41AA}.
We specially point out that the bound obtained in \eqref{6-42A} is independent of $T_{0}$.

By the similar argument to deal with the related estimates of $M_{s, k, k_1, k_2, l, l_1, l_2}^{\mu, \nu}$
in \eqref{5-27AA} and \eqref{5-40A}, we define
\begin{equation}\label{6-43AA}
\begin{aligned}
&H_{k,k_{1},k_{2},l,l_{1},l_{2}}^{\mu,\nu,j}(t,x,y):=
\int_{\mathbb{R}^{2}\times\mathbb{T}}\int_{\mathbb{R}^{2}\times\mathbb{T}}
L^{\mu,\nu}_{k,k_{1},k_{2},l,l_{1},l_{2}}(t,x,y,x_{1},y_{1},x_{2},y_{2})
f^{\mu}_{jk_{1}l_{1}}(s,x_{1},y_{1})\\
&~~~~~~~~~~~~~~~~~~~~~~~~~~~~~~~~~~~~~~~~~~~~~~~~~~
~~~~~~~~~~~~~~~~~~~~~~~~~~~~\cdot f^{\nu}_{jk_{2}l_{2}}(s,x_{2},y_{2})dx_{1}dy_{1}dx_{2}dy_{2},
\end{aligned}
\end{equation}
where
\begin{equation*}
f^{\mu}_{jkl}(t,x,y):=\varphi_{[j-4,j+4]}(x)\cdot R_{k}S_{l}V_{\mu}(t,x,y).
\end{equation*}
Utilizing \eqref{6-42A} and the same argument to obtain \eqref{5-33A} yields
\begin{equation}\label{6-44AA}
2^{9(k+l)}\sum\limits_{j\geq J_{k,l,t}}\sum_{\mbox{\tiny$\begin{array}{c}(k_{1},k_{2})\in\chi_{k},(l_{1},l_{2})\in\chi_{l}\end{array}$}}
2^{j}\|\varphi_{j}(x)\cdot (W_{k,k_{1},k_{2},l,l_{1},l_{2}}^{\mu,\nu}-H_{k,k_{1},k_{2},l,l_{1},l_{2}}^{\mu,\nu,j})
\|_{L^{2}(\mathbb{R}^{2}\times\mathbb{T})}
\lesssim \varepsilon_{0}^{2}.
\end{equation}
On the other hand, by the analogous argument to obtain \eqref{5-40A}, one has
\begin{equation}\label{6-45AA}
2^{9(k+l)}\sum\limits_{j\geq J_{k,l,t}}
\sum_{\mbox{\tiny$\begin{array}{c}(k_{1},k_{2})\in\chi_{k},(l_{1},l_{2})\in\chi_{l}\end{array}$}}
2^{j}\|\varphi_{j}(x)\cdot H_{k,k_{1},k_{2},l,l_{1},l_{2}}^{\mu,\nu,j})\|_{L^{2}(\mathbb{R}^{2}\times\mathbb{T})}
\lesssim \varepsilon_{0}^{2}.
\end{equation}
Combining \eqref{6-44AA} with \eqref{6-45AA} derives \eqref{6-35AA}. Finally, \eqref{6-19AAAB} comes from \eqref{6-2-1} and \eqref{6-2-2}
and then the proof of Lemma 6.1 is completed.
\end{proof}

\subsection{Proof of Lemma 6.2.}

To deal with the related estimates hereafter, we will use the following notations:
\begin{equation*}
k_{\mathrm{min}}:=\mathrm{min}(k_{1},k_{2},k_{3}),~~~k_{\mathrm{med}}:
=\mathrm{med}(k_{1},k_{2},k_{3}),~~~k_{\mathrm{max}}:=\mathrm{max}(k_{1},k_{2},k_{3}),
\end{equation*}
and
\begin{equation}\label{2-11good-1}
k_{l_{\mathrm{min}}}, k_{l_{\mathrm{med}}}, k_{l_{\mathrm{max}}} :=k_{j}
~\mathrm{for}~l_{\mathrm{min}}=l_{j} ~\mathrm{or}~
l_{\mathrm{med}}=l_{j} ~\mathrm{or}~l_{\mathrm{max}}=l_{j}, \mathrm{respectively};
\end{equation}
\begin{equation}\label{2-11good-2}
l_{k_{\mathrm{min}}}, l_{k_{\mathrm{med}}}, l_{k_{\mathrm{max}}} :=l_{j}
~\mathrm{for}~k_{\mathrm{min}}=k_{j} ~\mathrm{or}~
k_{\mathrm{med}}=k_{j}~\mathrm{or}~k_{\mathrm{max}}=k_{j}, \mathrm{respectively}.
\end{equation}

\begin{proof} The proof of Lemma 6.2 will be divided into the following four steps.

\textbf{Step 1. Establish}
\begin{equation}\label{6-77A}
\sup\limits_{t\in [0,T_{0}]}\|W(t)-W(0)\|_{H^{N_{0}}(\mathbb{R}^{2}\times\mathbb{T})}\lesssim \varepsilon_{0}^{3}
\end{equation}
and
\begin{equation}\label{6-78A}
\sup\limits_{t\in [0,T_{0}]}\ \sup\limits_{k, l\geq-1}2^{9(k+l)}\|R_{k}S_{l}(W(t)-W(0))
\|_{L^{2}(\mathbb{R}^{2}\times\mathbb{T})}\lesssim \varepsilon_{0}^{3}.
\end{equation}

By \eqref{6-8ABC}, we have
\begin{equation}\label{6-47ABC}
\mathcal{F}_{x,y}(R_{k}S_{l}(W(t)-W(0)))(t,\xi,n)
=\sum_{\mu,\sigma,\lambda,\omega\in\{+,-\}}i\int_{0}^{t}
\mathcal{F}_{x,y}(R_{k}S_{l}F_{\mu\sigma\lambda\omega})(s, \xi, n)ds,
\end{equation}
where
\begin{equation}\label{6-48A}
\begin{aligned}
&\mathcal{F}_{x,y}(R_{k}S_{l}F_{\mu\sigma\lambda\omega})(s,\xi,n)\\
=& \psi_{k}(\xi)\psi_{l}(n)\cdot e^{is\Lambda_{n}(\xi)}\sum_{m\in \mathbb{Z}}
\sum_{l'\in \mathbb{Z}}\int_{\mathbb{R}^{2}}
\int_{\mathbb{R}^{2}}\Big(\frac{\mathcal{M}_{n,m}(\xi,\eta)}{\Phi_{\sigma,\mu}^{n,m}(\xi,\eta)}+
\frac{\mathcal{M}_{n,n-m}(\xi,\xi-\eta)}{\Phi_{\mu,\sigma}^{n,n-m}(\xi,\xi-\eta)}\Big)\\
&~~~~~~~~~~~~~~~~~~~~~~~~~~~~~~~~~~~\cdot\mathcal{M}_{m,l'}(\eta,\zeta)
\widehat{U_{\sigma;n-m}}(s,\xi-\eta)\widehat{U_{\lambda;m-l'}}(s,\eta-\zeta)
\widehat{U_{\omega;l'}}(s,\zeta)d\zeta d\eta\\
=& \psi_{k}(\xi)\psi_{l}(n)\cdot e^{is\Lambda_{n}(\xi)}
\sum_{\mbox{\tiny$\begin{array}{c}k_{1},k_{2},k_{3},l_{1},l_{2},l_{3}\geq-1\end{array}$}}
\sum_{m\in \mathbb{Z}}
\sum_{l'\in \mathbb{Z}}\int_{\mathbb{R}^{2}}
\int_{\mathbb{R}^{2}}\Big(\frac{\mathcal{M}_{n,m}(\xi,\eta)}{\Phi_{\sigma,\mu}^{n,m}(\xi,\eta)}+
\frac{\mathcal{M}_{n,n-m}(\xi,\xi-\eta)}{\Phi_{\mu,\sigma}^{n,n-m}(\xi,\xi-\eta)}\Big)\\
&~\cdot\mathcal{M}_{m,l'}(\eta,\zeta)\psi_{k_{1}}(\xi-\eta)\psi_{l_{1}}(n-m)\widehat{U_{\sigma;n-m}}(s,\xi-\eta)
\psi_{k_{2}}(\eta-\zeta)\psi_{l_{2}}(m-l')\widehat{U_{\lambda;m-l'}}(s,\eta-\zeta)\\
&~~~~~~~~~~~~~~~~~~~~~~~~~~~~~~~~~~~~~~~~~~~~~~~~~~~~~~~~~~~~~~~~~~~~~~~~~~~~~~~~~~~~~~
~~~~~~~~\cdot\psi_{k_{3}}(\zeta))\psi_{l_{3}}(l'))\widehat{U_{\omega;l'}}(s,\zeta)d\zeta d\eta\\
=& \psi_{k}(\xi)\psi_{l}(n)\cdot e^{is\Lambda_{n}(\xi)}
\sum_{\mbox{\tiny$\begin{array}{c}k_{1},k_{2},k_{3},l_{1},l_{2},l_{3}\geq-1\end{array}$}}\sum_{m\in \mathbb{Z}}
\sum_{l'\in \mathbb{Z}}\int_{\mathbb{R}^{2}}
\int_{\mathbb{R}^{2}}\Big(\frac{\mathcal{M}_{n,m}(\xi,\eta)}{\Phi_{\sigma,\mu}^{n,m}(\xi,\eta)}+
\frac{\mathcal{M}_{n,n-m}(\xi,\xi-\eta)}{\Phi_{\mu,\sigma}^{n,n-m}(\xi,\xi-\eta)}\Big)\\
&~~~~~~~~~~~~~~~~
~~~\cdot\mathcal{M}_{m,l'}(\eta,\zeta)\mathcal{F}_{x,y}(R_{k_{1}}S_{1}U_{\sigma})(s,\xi-\eta,n-m)
\mathcal{F}_{x,y}(R_{k_{2}}S_{2}U_{\lambda})(s,\eta-\zeta,m-l')\\
&~~~~~~~~~~~~~~~~~~~~~~~~~~~~~~~~~~~~~~~~~~~~~~~~~~~~~~~~~~~~~~~~~~~~~~~~~~~~~~
~~~~~~~~~~~~~~\cdot\mathcal{F}_{x,y}(R_{k_{3}}S_{3}U_{\omega})(s,\zeta,l')d\zeta d\eta.
\end{aligned}
\end{equation}

As in \cite{IP2} (page 799), we set
\begin{equation}\label{6-49D}
\begin{aligned}
&\mathcal{Y}_{k}^{1}:=\{(k_{1},k_{2},k_{3})\in \mathbb{Z}^{3}:|\mathrm{max}(k_{1},k_{2},k_{3})-k|\leq4\},\\[2mm]
&\mathcal{Y}_{k}^{2}:=\{(k_{1},k_{2},k_{3})\in \mathbb{Z}^{3}:
\mathrm{max}(k_{1},k_{2},k_{3})-k\geq4,~\mathrm{max}(k_{1},k_{2},k_{3})-\mathrm{med}(k_{1},k_{2},k_{3})\leq 4\},\\[2mm]
&\mathcal{Y}_{k}:=\mathcal{Y}_{k}^{1}\cup\mathcal{Y}_{k}^{2}.
\end{aligned}
\end{equation}
By the same idea as in \eqref{2-3333}, on the support of the integrand
in \eqref{6-48A},~$(k_{1},k_{2},k_{3})\in \mathcal{Y}_{k}$ and $(l_{1},l_{2},l_{3})\in \mathcal{Y}_{l}$ hold.
Then \eqref{6-48A} can be written as
\begin{equation}\label{6-50A-1}
\mathcal{F}_{x,y}(R_{k}S_{l}F_{\mu\sigma\lambda\omega})(s,\xi,n)=\sum_{\mbox{\tiny$\begin{array}{c}
k_{1},k_{2},k_{3},l_{1},l_{2},l_{3}\geq-1\\
(k_{1},k_{2},k_{3})\in \mathcal{Y}_{k},~(l_{1},l_{2},l_{3})\in \mathcal{Y}_{l}\end{array}$}}
\mathcal{F}_{x,y}(F^{\mu\sigma\lambda\omega;k,l}_{k_{1},k_{2},k_{3},l_{1},l_{2},l_{3}})(s,\xi,n),
\end{equation}
where
\begin{equation}\label{6-50AA-1-1}
\begin{aligned}
&\mathcal{F}_{x,y}(F^{\mu\sigma\lambda\omega;k,l}_{k_{1},k_{2},k_{3},l_{1},l_{2},l_{3}})(s,\xi,n)\\
:=& \psi_{k}(\xi)\psi_{l}(n)\cdot e^{is\Lambda_{n}(\xi)}
\sum_{m\in \mathbb{Z}}\sum_{l'\in \mathbb{Z}}\int_{\mathbb{R}^{2}}
\int_{\mathbb{R}^{2}}\Big(\frac{\mathcal{M}_{n,m}(\xi,\eta)}{\Phi_{\sigma,\mu}^{n,m}(\xi,\eta)}+
\frac{\mathcal{M}_{n,n-m}(\xi,\xi-\eta)}{\Phi_{\mu,\sigma}^{n,n-m}(\xi,\xi-\eta)}\Big)
\mathcal{M}_{m,l'}(\eta,\zeta)
\\&\cdot\mathcal{F}_{x,y}(R_{k_{1}}S_{l_{1}}U_{\sigma})(s,\xi-\eta,n-m)
\mathcal{F}_{x,y}(R_{k_{2}}S_{l_{2}}U_{\lambda})(s,\eta-\zeta,m-l')\\
&\cdot\mathcal{F}_{x,y}(R_{k_{3}}S_{l_{3}}U_{\omega})(s,\zeta,l')d\zeta d\eta.
\end{aligned}
\end{equation}

By Lemma \ref{HC-9} in Appendix A with $r=2$, we have  from \eqref{6-50A-1} and \eqref{6-50AA-1-1} that
\begin{equation}\label{6-66A}
\begin{aligned}
&\|R_{k}S_{l}F_{\mu\sigma\lambda\omega}(s)\|_{L^{2}(\mathbb{R}^{2}\times\mathbb{T})}\\
\lesssim& \sum_{\mbox{\tiny$\begin{array}{c}
k_{1},k_{2},k_{3},l_{1},l_{2},l_{3}\geq-1\\
(k_{1},k_{2},k_{3})\in \mathcal{Y}_{k},~(l_{1},l_{2},l_{3})\in \mathcal{Y}_{l}\end{array}$}}
\Big\|\mathcal{F}_{x,y}(F^{\mu\sigma\lambda\omega;k,l}_{k_{1},k_{2},k_{3},l_{1},l_{2},l_{3}})(s,\xi,n)
\Big\|_{L^{2}_{\xi}l^{2}_{n}}\\
\lesssim& \sum_{\mbox{\tiny$\begin{array}{c}
k_{1},k_{2},k_{3},l_{1},l_{2},l_{3}\geq-1\\
(k_{1},k_{2},k_{3})\in \mathcal{Y}_{k},~(l_{1},l_{2},l_{3})\in \mathcal{Y}_{l}\end{array}$}}
2^{k+k_{1}+\mathrm{max}(k_{2},k_{3})+k_{3}}2^{3l+2\mathrm{max}(l_{2},l_{3})+l_{3}}
\mathrm{min}(J_1, J_2),
\end{aligned}
\end{equation}
where
\begin{subequations}\label{6-4-1}
\begin{align}
J_1:=&\|R_{k_{l_{\mathrm{min}}}}S_{l_{\mathrm{min}}}U_{\pm}(s)\|_{L^{\infty}(\mathbb{R}^{2}\times\mathbb{T})}
\|R_{k_{l_{\mathrm{med}}}}S_{l_{\mathrm{med}}}U_{\pm}(s)\|_{L^{\infty}(\mathbb{R}^{2}\times\mathbb{T})}
\|R_{k_{l_{\mathrm{max}}}}S_{l_{\mathrm{max}}}U_{\pm}(s)\|_{L^{2}(\mathbb{R}^{2}\times\mathbb{T})},\label{6-67A}\\
J_2:=&\|R_{k_{_{\mathrm{min}}}}S_{l_{k_{_{\mathrm{min}}}}}U_{\pm}(s)\|_{L^{\infty}(\mathbb{R}^{2}\times\mathbb{T})}
\|R_{k_{_{\mathrm{med}}}}S_{l_{k_{_{\mathrm{med}}}}}U_{\pm}(s)\|_{L^{\infty}(\mathbb{R}^{2}\times\mathbb{T})}
\|R_{k_{_{\mathrm{max}}}}S_{l_{k_{_{\mathrm{max}}}}}U_{\pm}(s)\|_{L^{2}(\mathbb{R}^{2}\times\mathbb{T})}.\label{6-68A}
\end{align}
\end{subequations}

It follows from \eqref{6-10A}, \eqref{6-12A} and the notations in \eqref{2-11good-1}-\eqref{2-11good-2} that
\begin{equation}\label{6-72A}
\begin{aligned}
J_1&\lesssim \varepsilon_{0}(1+s)^{-1}2^{-7(k_{l_{\mathrm{min}}}+l_{\mathrm{min}})}
\varepsilon_{0}(1+s)^{-1}2^{-7(k_{l_{\mathrm{med}}}+l_{\mathrm{med}})}\varepsilon_{0}(1+s)^{\kappa}
2^{-N(k_{l_{\mathrm{max}}}+l_{\mathrm{max}})/2}\\
&\lesssim \varepsilon_{0}^{3}(1+s)^{-(2-\kappa)}2^{-7(k_{1}+k_{2}+k_{3})}
2^{-7(l_{1}+l_{2}+l_{3})}2^{-(N/2-7)l_{\mathrm{max}}}
\end{aligned}
\end{equation}
and
\begin{equation}\label{6-73A}
\begin{aligned}
J_2&\lesssim \varepsilon_{0}(1+s)^{-1}2^{-7(k_{\mathrm{min}}+l_{k_{\mathrm{min}}})}
\varepsilon_{0}(1+s)^{-1}2^{-7(k_{\mathrm{med}}+l_{k_{\mathrm{med}}})}
\varepsilon_{0}(1+s)^{\kappa}2^{-N(k_{\mathrm{max}}+l_{k_{\mathrm{max}}})/2}
\\&\lesssim \varepsilon_{0}^{3}(1+s)^{-(2-\kappa)}2^{-7(l_{1}+l_{2}+l_{3})}
2^{-7(k_{1}+k_{2}+k_{3})}2^{-(N/2-7)k_{\mathrm{max}}}
\end{aligned}
\end{equation}

\noindent Therefore, using \eqref{6-72A} and \eqref{6-73A}, we deduce
\begin{equation}\label{6-75A}
\begin{aligned}
&\sum_{\mbox{\tiny$\begin{array}{c}
k_{1},k_{2},k_{3},l_{1},l_{2},l_{3}\geq-1\\
(k_{1},k_{2},k_{3})\in \mathcal{Y}_{k},~(l_{1},l_{2},l_{3})\in \mathcal{Y}_{l}\end{array}$}}
2^{k+k_{1}+\mathrm{max}(k_{2},k_{3})+k_{3}}2^{3l+2\mathrm{max}(l_{2},l_{3})+l_{3}}\mathrm{min}(J_1, J_2)\\
\lesssim& \sum_{\mbox{\tiny$\begin{array}{c}
k_{1},k_{2},k_{3},l_{1},l_{2},l_{3}\geq-1\\
(k_{1},k_{2},k_{3})\in \mathcal{Y}_{k},~(l_{1},l_{2},l_{3})\in \mathcal{Y}_{l}\end{array}$}}
2^{k+k_{1}+\mathrm{max}(k_{2},k_{3})+k_{3}}2^{3l+2\mathrm{max}(l_{2},l_{3})+l_{3}}J_1^{1/2} J_2^{1/2}\\
\lesssim& \varepsilon_{0}^{3}(1+s)^{-(2-\kappa)}2^{-(N/4-3)(k+l)}.
\end{aligned}
\end{equation}
By \eqref{6-66A} and \eqref{6-75A}, we arrive at
\begin{equation}\label{6-76A}
\begin{aligned}
\|R_{k}S_{l}F_{\mu\sigma\lambda\omega}(s)\|_{L^{2}(\mathbb{R}^{2}\times\mathbb{T})}
\lesssim \varepsilon_{0}^{3}(1+s)^{-(2-\kappa)}2^{-(N/4-3)(k+l)}.
\end{aligned}
\end{equation}
In view of \eqref{6-47ABC}, \eqref{6-26AB} and \eqref{6-76A}, \eqref{6-77A} and \eqref{6-78A} are proved.

\vskip 0.2 true cm

\textbf{Step 2}: Establish that for $k,l\geq -1$ and $t\in [0,T_{0}]$,
\begin{equation}\label{6-79S}
2^{9(k+l)}\sum\limits_{j\in \mathbb{Z}^{+}}2^{j}\|\varphi_{j}(x)\cdot R_{k}S_{l}(W(t)-W(0))
\|_{L^{2}(\mathbb{R}^{2}\times\mathbb{T})}\lesssim \varepsilon_{0}^{3}+c_{0}\varepsilon_{0}.
\end{equation}

Due to \eqref{6-47ABC}, the proof of \eqref{6-79S} can be reduced to that for $k,l\geq-1$ and $\mu,\sigma,\lambda,\omega\in\{+,-\}$,
\begin{equation}\label{6-80AAA}
2^{9(k+l)}\sum\limits_{j\in \mathbb{Z}^{+}}2^{j}\Big\|\varphi_{j}(x)\cdot \int_{0}^{t}R_{k}S_{l}F_{\mu\sigma\lambda\omega}(s)ds\Big\|_{L^{2}(\mathbb{R}^{2}\times\mathbb{T})}
\lesssim \varepsilon_{0}^{3}+c_{0}\varepsilon_{0}.
\end{equation}

Note that
\begin{equation}\label{6-aabb}
\begin{aligned}
&2^{9(k+l)}\sum\limits_{j\in \mathbb{Z}^{+}}2^{j}\Big\|\varphi_{j}(x)\cdot \int_{0}^{t}R_{k}S_{l}F_{\mu\sigma\lambda\omega}(s)ds\Big\|_{L^{2}(\mathbb{R}^{2}\times\mathbb{T})}\\
\lesssim& 2^{9(k+l)}\int_{0}^{t}\sum\limits_{j\in \mathbb{Z}^{+}}\Big(2^{j}\Big\|\varphi_{j}(x)\cdot R_{k}S_{l}F_{\mu\sigma\lambda\omega}(s)\Big\|_{L^{2}(\mathbb{R}^{2}\times\mathbb{T})}\Big)ds\\
\leq& 2^{9(k+l)}\int_{0}^{t}\Big(\sum\limits_{0<j\leq M(k,s)}2^{j}\Big\|\varphi_{j}(x)\cdot
R_{k}S_{l}F_{\mu\sigma\lambda\omega}(s)\Big\|_{L^{2}(\mathbb{R}^{2}\times\mathbb{T})}\Big)ds\\
&+2^{9(k+l)}\int_{0}^{t}\Big(\sum\limits_{j\geq M(k,s)}2^{j}\Big\|\varphi_{j}(x)\cdot
R_{k}S_{l}F_{\mu\sigma\lambda\omega}(s)\Big\|_{L^{2}(\mathbb{R}^{2}\times\mathbb{T})}\Big)ds:=J_3+J_4,
\end{aligned}
\end{equation}
where
\begin{equation}\label{6-82A}
M(k,s):=100+\mathrm{max}(9\mathrm{log}_{2}(1+s)/10,10k).
\end{equation}

By \eqref{6-76A}, $J_3$ in \eqref{6-aabb} can be estimated as
\begin{equation}\label{6-81A}
J_3\lesssim \varepsilon_{0}^{3}.
\end{equation}

On the other hand, due to \eqref{6-50A-1},  $J_4$ in \eqref{6-aabb} is treated as
\begin{equation}\label{6-84A}
\begin{aligned}
J_4=& 2^{9(k+l)}\int_{0}^{t}\Big(\sum\limits_{j\geq M(k,s)}2^{j}
\Big\|\sum_{\mbox{\tiny$\begin{array}{c}
k_{1},k_{2},k_{3},l_{1},l_{2},l_{3}\geq-1\\
(k_{1},k_{2},k_{3})\in \mathcal{Y}_{k},~(l_{1},l_{2},l_{3})\in \mathcal{Y}_{l}\end{array}$}}
\varphi_{j}(x)\cdot F^{\mu\sigma\lambda\omega;k,l}_{k_{1},k_{2},k_{3},l_{1},l_{2},l_{3}}(s)
\Big\|_{L^{2}(\mathbb{R}^{2}\times\mathbb{T})}\Big)ds\\
\lesssim& 2^{9(k+l)}\int_{0}^{t}\Big(\sum\limits_{j\geq M(k,s)}2^{j}
\sum_{\mbox{\tiny$\begin{array}{c}
k_{1},k_{2},k_{3},l_{1},l_{2},l_{3}\geq-1\\
(k_{1},k_{2},k_{3})\in \mathcal{Y}_{k},~(l_{1},l_{2},l_{3})\in \mathcal{Y}_{l}\\ \mathrm{max}(k_{1},k_{2},k_{3})\geq j/10\end{array}$}}
\Big\|\mathcal{F}_{x,y}(F^{\mu\sigma\lambda\omega;k,l}_{k_{1},k_{2},k_{3},l_{1},l_{2},l_{3}})(s,\xi,n)
\Big\|_{L^{2}_{\xi}l^{2}_{n}}\Big)ds\\
+& 2^{9(k+l)}\int_{0}^{t}\Big(\sum\limits_{j\geq M(k,s)}2^{j}
\Big\|\sum_{\mbox{\tiny$\begin{array}{c}
k_{1},k_{2},k_{3},l_{1},l_{2},l_{3}\geq-1\\
(k_{1},k_{2},k_{3})\in \mathcal{Y}_{k},~(l_{1},l_{2},l_{3})\in \mathcal{Y}_{l}\\ k_1, k_2, k_3\in [-1, j/10]\end{array}$}}\varphi_j(x)\cdot F^{\mu\sigma\lambda\omega;k,l}_{k_{1},k_{2},k_{3},l_{1},l_{2},l_{3}}(s)
\Big\|_{L^{2}(\mathbb{R}^2\times\mathbb{T}}\Big)ds\\
:=&J_5+J_6.
\end{aligned}
\end{equation}

At first, we deal with $J_5$. By Lemma \ref{HC-9} in Appendix A with $r=+\infty$, one has
\begin{equation}\label{6-93A}
\begin{aligned}
\Big\|\mathcal{F}_{x,y}(F^{\mu\sigma\lambda\omega;k,l}_{k_{1},k_{2},k_{3},l_{1},l_{2},l_{3}})(s,\xi,n)
\Big\|_{L^{2}_{\xi}l^{2}_{n}}
\lesssim 2^{k+k_{1}+\mathrm{max}(k_{2},k_{3})+k_{3}}2^{3l+2\mathrm{max}(l_{2},l_{3})+l_{3}}2^{(k+l/2)}\mathrm{min}(J_7, J_8),
\end{aligned}
\end{equation}
where
\begin{subequations}\label{6-4-2}
\begin{align}
J_7:=&\|R_{k_{l_{\mathrm{min}}}}S_{l_{\mathrm{min}}}U_{\pm}(s)\|_{L^{\infty}(\mathbb{R}^{2}\times\mathbb{T})}
\|R_{k_{l_{\mathrm{med}}}}S_{l_{\mathrm{med}}}U_{\pm}(s)\|_{L^{2}(\mathbb{R}^{2}\times\mathbb{T})}
\|R_{k_{l_{\mathrm{max}}}}S_{l_{\mathrm{max}}}U_{\pm}(s)\|_{L^{2}(\mathbb{R}^{2}\times\mathbb{T})},\label{6-94A}\\
J_8:=&\|R_{k_{\mathrm{min}}}S_{l_{k_{\mathrm{min}}}}U_{\pm}(s)\|_{L^{\infty}(\mathbb{R}^{2}\times\mathbb{T})}
\|R_{k_{\mathrm{med}}}S_{l_{k_{\mathrm{med}}}}U_{\pm}(s)\|_{L^{2}(\mathbb{R}^{2}\times\mathbb{T})}
\|R_{k_{\mathrm{max}}}S_{l_{k_{\mathrm{max}}}}U_{\pm}(s)\|_{L^{2}(\mathbb{R}^{2}\times\mathbb{T})}.\label{6-95A}
\end{align}
\end{subequations}

It is noted that when $j\geq M(k,s)$,  $(k_{1},k_{2},k_{3})\in \mathcal{Y}_{k}$ and $\mathrm{max}(k_{1},k_{2},k_{3})\geq j/10$
imply $(k_{1},k_{2},k_{3})\in \mathcal{Y}_{k}^{2}$, which yields $k_{\mathrm{med}}\geq k_{\mathrm{max}}-4$. Then,
by \eqref{6-10A} and \eqref{6-12A}, we arrive at
\begin{equation*}\label{6-99A}
\begin{aligned}
J_7&\lesssim \varepsilon_{0}(1+s)^{-1}2^{-7(k_{l_{\mathrm{min}}}+l_{\mathrm{min}})}
\varepsilon_{0}(1+s)^{\kappa}2^{-N(k_{l_{\mathrm{med}}}+l_{\mathrm{med}})/2}
\varepsilon_{0}(1+s)^{\kappa}2^{-N(k_{l_{\mathrm{max}}}+l_{\mathrm{max}})/2}
\\&\lesssim \varepsilon_{0}^{3}2^{-N(k_{\mathrm{max}}+l_{\mathrm{max}})/2},\\
J_8&\lesssim \varepsilon_{0}(1+s)^{-1}2^{-7(k_{\mathrm{min}}+l_{k_{\mathrm{min}}})}
\varepsilon_{0}(1+s)^{\kappa}2^{-N(k_{\mathrm{med}}+l_{k_{\mathrm{med}}})/2}
\varepsilon_{0}(1+s)^{\kappa}2^{-N(k_{\mathrm{max}}+l_{k_{\mathrm{max}}})/2}
\\&\lesssim \varepsilon_{0}^{3}2^{-7(l_{1}+l_{2}+l_{3})}2^{-Nk_{\mathrm{max}}}.
\end{aligned}
\end{equation*}
Hence,
\begin{equation}\label{6-101A}
\begin{aligned}
\mathrm{min}(J_7, J_8)&\leq J_7^{1/2} J_8^{1/2}
\lesssim \varepsilon_{0}^{3}2^{-3Nk_{\mathrm{max}}/4}2^{-Nl_{\mathrm{max}}/4}2^{-3(l_{1}+l_{2}+l_{3})}.
\end{aligned}
\end{equation}
By \eqref{6-93A}, \eqref{6-101A} and \eqref{6-84A}, we get
\begin{equation}\label{6-102A}
J_5\lesssim \varepsilon_{0}^{3}.
\end{equation}

Therefore, in view of \eqref{6-aabb}, \eqref{6-81A}, \eqref{6-84A} and \eqref{6-102A}, in order to
prove \eqref{6-80AAA}, it suffices to establish
\begin{equation}\label{6-103AAH}
\begin{aligned}
J_6:=2^{9(k+l)}\int_{0}^{t}\Big(\sum\limits_{j\geq M(k,s)}2^{j}\Big\|\varphi_{j}(x)\cdot R_{k}S_{l}F^{\mu\sigma\lambda\omega}_{j}(s)\Big\|_{L^{2}(\mathbb{R}^{2}\times\mathbb{T})}\Big)ds
\lesssim \varepsilon_{0}^{3}+c_{0}\varepsilon_{0},
\end{aligned}
\end{equation}
where
\begin{equation}\label{6-103A}
\begin{aligned}
&\mathcal{F}_{x,y}(R_{k}S_{l}F^{\mu\sigma\lambda\omega}_{j})(s,\xi,n)
:=\sum_{\mbox{\tiny$\begin{array}{c}l_{1},l_{2},l_{3}\geq-1\\
(k_{1},k_{2},k_{3})\in \mathcal{Y}_{k},(l_{1},l_{2},l_{3})\in \mathcal{Y}_{l}\\
k_{1},k_{2},k_{3}\in [-1,j/10]\end{array}$}}
\mathcal{F}_{x,y}(F^{\mu\sigma\lambda\omega;k,l}_{k_{1},k_{2},k_{3},l_{1},l_{2},l_{3}})(s,\xi,n).
\end{aligned}
\end{equation}

The proof of \eqref{6-103AAH} is arranged in the following two steps.\vskip 0.2cm

\textbf{Step 3: Establish}
\begin{equation}\label{7-25AA1}
\begin{aligned}
&2^{9(k+l)}\int_{0}^{t}\Big(\sum\limits_{j\geq M'(k,s)}2^{j}\Big\|\varphi_{j}(x)\cdot
R_{k}S_{l}F^{\mu\sigma\lambda\omega}_{j}(s)\Big\|_{L^{2}(\mathbb{R}^{2}\times\mathbb{T})}\Big)ds
\lesssim \varepsilon_{0}^{3},
\end{aligned}
\end{equation}
where $M'(k,s):=100+\mathrm{max}(\mathrm{log}_{2}(1+s),10k)\geq M(k, s)$ with $M(k, s)$ defined in \eqref{6-82A}.

It follows from \eqref{6-103A} and \eqref{6-50AA-1-1} that
\begin{equation}\label{7-27AA3}
\begin{aligned}
&\mathcal{F}_{x,y}(R_{k}S_{l}F^{\mu\sigma\lambda\omega}_{j})(s,\xi,n)\\
:=&\psi_{k}(\xi)\psi_{l}(n)\cdot e^{is\Lambda_{n}(\xi)}
\sum_{\mbox{\tiny$\begin{array}{c}l_{1},l_{2},l_{3}\geq-1\\
(k_{1},k_{2},k_{3})\in \mathcal{Y}_{k},(l_{1},l_{2},l_{3})\in \mathcal{Y}_{l}\\
k_{1},k_{2},k_{3}\in [-1,j/10]\end{array}$}}
\sum_{m\in \mathbb{Z}}\sum_{l'\in \mathbb{Z}}
\int_{\mathbb{R}^{2}}\int_{\mathbb{R}^{2}}\\
&~~~~~~~~\Big(\frac{\mathcal{M}_{n,m}(\xi,\eta)}{\Phi_{\sigma,\mu}^{n,m}(\xi,\eta)}+
\frac{\mathcal{M}_{n,n-m}(\xi,\xi-\eta)}{\Phi_{\mu,\sigma}^{n,n-m}(\xi,\xi-\eta)}\Big)
\mathcal{M}_{m,l'}(\eta,\zeta)
\mathcal{F}_{x,y}(R_{k_{1}}S_{l_{1}}U_{\sigma})(s,\xi-\eta,n-m)\\
&~~~~~~~~~~~~~~~~~~~~~~~~~~
~~~~~~~~~~~\cdot\mathcal{F}_{x,y}(R_{k_{2}}S_{l_{2}}U_{\lambda})(s,\eta-\zeta,m-l')
\mathcal{F}_{x,y}(R_{k_{3}}S_{l_{3}}U_{\omega})(s,\zeta,l')d\zeta d\eta\\[2mm]
=& \psi_{k}(\xi)\psi_{l}(n)
\sum_{\mbox{\tiny$\begin{array}{c}l_{1},l_{2},l_{3}\geq-1\\
(k_{1},k_{2},k_{3})\in \mathcal{Y}_{k},(l_{1},l_{2},l_{3})\in \mathcal{Y}_{l}\\
k_{1},k_{2},k_{3}\in [-1,j/10]\end{array}$}}
\sum_{n_{2}\in \mathbb{Z}}\sum_{n_{3}\in \mathbb{Z}}\int_{\mathbb{R}^{2}}
\int_{\mathbb{R}^{2}}e^{is\Phi_{\sigma,\lambda,\omega}^{n,n_{2}+n_{3},n_{3}}(\xi,,\xi_{2}+\xi_{3},\xi_{3})}\\
&~~~~~~~~~~~~~~~
\Big(\frac{\mathcal{M}_{n,n_{2}+n_{3}}(\xi,\xi_{2}+\xi_{3})}{\Phi_{\sigma,\mu}^{n,n_{2}+n_{3}}(\xi,\xi_{2}+\xi_{3})}+
\frac{\mathcal{M}_{n,n-n_{2}-n_{3}}(\xi,\xi-\xi_{2}-\xi_{3})}{\Phi_{\mu,\sigma}^{n,n-n_{2}-n_{3}}(\xi,\xi-\xi_{2}-\xi_{3})}\Big)
\mathcal{M}_{n_{2}+n_{3},n_{3}}(\xi_{2}+\xi_{3},\xi_{3})\\
&~~~~~~~~~~~~~~~~~~~~~~~~~~~~~\cdot\mathcal{F}_{x,y}(R_{k_{1}}S_{l_{1}}V_{\sigma})(s,\xi-\xi_{2}-\xi_{3},n-n_{2}-n_{3})
\mathcal{F}_{x,y}(R_{k_{2}}S_{l_{2}}V_{\lambda})(s,\xi_{2},n_{2})\\
&~~~~~~~~~~~~~~~~~~~~~~~~~~~~~~~~~~~~~~~~~~~~~~~~~~~~~~~~~~~~~~~~~
~~~~~~~~~~~~~~~~~~~~\cdot\mathcal{F}_{x,y}(R_{k_{3}}S_{l_{3}}V_{\omega})(s,\xi_{3},n_{3})d\xi_{2} d\xi_{3}\\[2mm]
=& \psi_{k}(\xi)\psi_{l}(n)
\sum_{\mbox{\tiny$\begin{array}{c}l_{1},l_{2},l_{3}\geq-1\\
(k_{1},k_{2},k_{3})\in \mathcal{Y}_{k},(l_{1},l_{2},l_{3})\in \mathcal{Y}_{l}\\
k_{1},k_{2},k_{3}\in [-1,j/10]\end{array}$}}
\sum_{n_{3}\in \mathbb{Z}}\sum_{n'\in \mathbb{Z}}\int_{\mathbb{R}^{2}}
\int_{\mathbb{R}^{2}}e^{is \Phi_{\sigma,\lambda,\omega}^{n,n',n_{3}}(\xi,v,\xi_{3})}\\
&~~~~~~~~~~~~\Big(\frac{\mathcal{M}_{n,n'}(\xi,v)}{\Phi_{\sigma,\mu}^{n,n'}(\xi,v)}+
\frac{\mathcal{M}_{n,n-n'}(\xi,\xi-v)}{\Phi_{\mu,\sigma}^{n,n-n'}(\xi,\xi-v)}\Big)\mathcal{M}_{n',n_{3}}(v,\xi_{3})
\mathcal{F}_{x,y}(R_{k_{1}}S_{l_{1}}V_{\sigma})(s,\xi-v,n-n')\\
&~~~~~~~~~~~~~~~~~~
~~~~~~~~~~~~~~~~\cdot\mathcal{F}_{x,y}(R_{k_{2}}S_{l_{2}}V_{\lambda})(s,v-\xi_{3},n'-n_{3})
\mathcal{F}_{x,y}(R_{k_{3}}S_{l_{3}}V_{\omega})(s,\xi_{3},n_{3})d\xi_{3}dv,
\end{aligned}
\end{equation}
where
\begin{equation*}\label{7-29AA5}
\begin{aligned}
\Phi_{\sigma,\lambda,\omega}^{n,n_{2},n_{3}}(\xi,\xi_{2},\xi_{3})
:=\Lambda_{n}(\xi)-\Lambda_{\sigma;n-n_{2}}
(\xi-\xi_{2})-\Lambda_{\lambda;n_{2}-n_{3}}(\xi_{2}-\xi_{3})
-\Lambda_{\omega;n_{3}}(\xi_{3}).
\end{aligned}
\end{equation*}

\vskip 0.2 true cm

On the other hand, based on \eqref{7-27AA3}, we take the following decomposition
\begin{equation}\label{7-31AA6}
R_{k}S_{l}F^{\mu\sigma\lambda\omega}_{j}
=\sum_{\mbox{\tiny$\begin{array}{c}l_{1},l_{2},l_{3}\geq-1\\
(k_{1},k_{2},k_{3})\in \mathcal{Y}_{k},(l_{1},l_{2},l_{3})\in \mathcal{Y}_{l}\\
k_{1},k_{2},k_{3}\in [-1,j/10]\end{array}$}}
\sum_{k_{4}\in \mathbb{Z}}\sum_{l_{4}\geq-1}
T^{\mu\sigma\lambda\omega;k,l}_{k_{4},l_{4}}
[R_{k_{1}}S_{l_{1}}V_{\sigma},R_{k_{2}}S_{l_{2}}V_{\lambda},R_{k_{3}}S_{l_{3}}V_{\omega}],
\end{equation}
where for any $g_{1},g_{2},g_{3}\in L^{2}(\mathbb{R}^{2}\times\mathbb{T})$,
\begin{equation}\label{7-32AA7}
\begin{aligned}
&\mathcal{F}_{x,y}(T^{\mu\sigma\lambda\omega;k,l}_{k_{4},l_{4}}[g_{1},g_{2},g_{3}])(s,\xi,n):=
\sum_{n_{3}\in \mathbb{Z}}\sum_{n'\in \mathbb{Z}}\int_{\mathbb{R}^{2}}
\int_{\mathbb{R}^{2}}e^{is \Phi_{\sigma,\lambda,\omega}^{n,n',n_{3}}(\xi,v,\xi_{3})}
\varphi_{k_{4}}(v)\psi_{l_{4}}(n')\\
&~~~~~\cdot Q(\xi,n,v,n',\xi_{3},n_{3})\mathcal{F}_{x,y}(g_{1})(\xi-v,n-n')
\mathcal{F}_{x,y}(g_{2})(v-\xi_{3},n'-n_{3})
\mathcal{F}_{x,y}(g_{3})(\xi_{3},n_{3})d\xi_{3}dv
\end{aligned}
\end{equation}
and
\begin{equation*}\begin{aligned}
&Q(\xi,n,v,n',\xi_{3},n_{3}):=\Big(\frac{\mathcal{M}_{n,n'}(\xi,v)}{\Phi_{\sigma,\mu}^{n,n'}(\xi,v)}+
\frac{\mathcal{M}_{n,n-n'}(\xi,\xi-v)}{\Phi_{\mu,\sigma}^{n,n-n'}(\xi,\xi-v)}\Big)\mathcal{M}_{n',n_{3}}(v,\xi_{3})\\
&~~~~~~~~~~~~~~~~~~~~~~~~~\cdot\psi_{k}(\xi)\psi_{[k_{1}-2,k_{1}+2]}(\xi-v)\psi_{[k_{2}-2,k_{2}+2]}(v-\xi_{3})
\psi_{[k_{3}-2,k_{3}+2]}(\xi_{3})\\
&~~~~~~~~~~~~~~~~~~~~~~~~~\cdot\psi_{l}(n)\psi_{[l_{1}-2,l_{1}+2]}(n-n')
\psi_{[l_{2}-2,l_{2}+2]}(n'-n_{3})\psi_{[l_{3}-2,l_{3}+2]}(n_{3}).
\end{aligned}
\end{equation*}

On the other hand, for later use, the
operator $T^{\mu\sigma\lambda\omega;k,l}_{k_{4},l_{4}}[\cdot]$ in \eqref{7-32AA7} has such a form in the physical space
\begin{equation}\label{7-33AA8}
\begin{aligned}
T^{\mu\sigma\lambda\omega;k,l}_{k_{4},l_{4}}[g_{1},g_{2},g_{3}](s,x,y)&=
\int_{\mathbb{R}^{2}\times\mathbb{T}}\int_{\mathbb{R}^{2}\times\mathbb{T}}\int_{\mathbb{R}^{2}\times\mathbb{T}}
g_{1}(x_{1},y_{1})g_{2}(x_{2},y_{2})g_{3}(x_{3},y_{3})\\
&\cdot R^{\mu\sigma\lambda\omega;k,l}_{k_{4},l_{4}}(s,x,y,x_{1},y_{1},x_{2},y_{2},x_{3},y_{3})
dx_{1}dy_{1}dx_{2}dy_{2}dx_{3}dy_{3},
\end{aligned}
\end{equation}
where
\begin{equation*}
\begin{aligned}
&R^{\mu\sigma\lambda\omega;k,l}_{k_{4},l_{4}}(s,x,y,x_{1},y_{1},x_{2},y_{2},x_{3},y_{3}):=
\sum_{n_{3}\in \mathbb{Z}}\sum_{n'\in \mathbb{Z}}\sum_{n\in \mathbb{Z}}
e^{i[n\cdot(y-y_{1})+n'\cdot(y_{1}-y_{2})+n_{3}\cdot(y_{2}-y_{3})]}\\
&\cdot\int_{\mathbb{R}^{2}}\int_{\mathbb{R}^{2}}\int_{\mathbb{R}^{2}}e^{i(x-x_{1})\cdot\xi}
e^{i(x_{1}-x_{2})\cdot v}e^{i(x_{2}-x_{3})\cdot \xi_{3}}e^{is\Phi_{\sigma,\lambda,\omega}^{n,n',n_{3}}(\xi,v,\xi_{3})}
\varphi_{k_{4}}(v)\psi_{l_{4}}(n')Q(\xi,n,v,n',\xi_{3},n_{3})d\xi dv d\xi_{3}.
\end{aligned}
\end{equation*}

By the definition of $Q(\xi,n,v,n',\xi_{3},n_{3})$ in \eqref{7-32AA7}, Lemma \ref{HC-11} and the
definitions of $\mathcal{M}_{n,n'}(\xi,v)$,\\$\mathcal{M}_{n,n-n'}(\xi,\xi-v)$ in~\eqref{5-7-0}-\eqref{5-7-1},
one has
\begin{equation}\label{7-34AA9}
|Q(\xi,n,v,n',\xi_{3},n_{3})|\lesssim (2^{k}+2^{l})(2^{k_{\mathrm{max}}}+2^{l_{\mathrm{max}}})^{2}.
\end{equation}
Then  it comes from \eqref{7-32AA7} and \eqref{7-34AA9} that
\begin{equation}\label{7-35AA10}
\begin{aligned}
&\Big\|\mathcal{F}_{x,y}(T^{\mu\sigma\lambda\omega;k,l}_{k_{4},l_{4}}[g_{1},g_{2},g_{3}])(s,\xi,n)\Big\|_{L^{2}_{\xi}l^{2}_{n}}\\
\lesssim& (2^{k}+2^{l})(2^{k_{\mathrm{max}}}+2^{l_{\mathrm{max}}})^{2}\Big\|\sum_{n'\in \mathbb{Z}}\int_{\mathbb{R}^{2}}
|\varphi_{k_{4}}(v)\psi_{l_{4}}(n')||\mathcal{F}_{x,y}(g_{1})(\xi-v,n-n')|\\
&~~~~~~~~~~~~~~~~~~~~~~~~~~~~~~~\cdot\Big(\sum_{n_{3}\in \mathbb{Z}}\int_{\mathbb{R}^{2}}|\mathcal{F}_{x,y}(g_{2})(v-\xi_{3},n'-n_{3})|
|\mathcal{F}_{x,y}(g_{3})(\xi_{3},n_{3})|d\xi_{3}\Big)dv\Big\|_{L^{2}_{\xi}l^{2}_{n}}\\
\lesssim& (2^{k}+2^{l})(2^{k_{\mathrm{max}}}+2^{l_{\mathrm{max}}})^{2}\|\mathcal{F}_{x,y}(g_{2})\|_{L^{2}_{\xi}l^{2}_{n}}
\|\mathcal{F}_{x,y}(g_{3})\|_{L^{2}_{\xi}l^{2}_{n}}\\
&~~~~~~~~~~~~~~~~~~~~~~
~~~~~~~~~~~~~~~~~~~~~~~~~\cdot\Big\|\sum_{n'\in \mathbb{Z}}\int_{\mathbb{R}^{2}}
|\varphi_{k_{4}}(v)\psi_{l_{4}}(n')||\mathcal{F}_{x,y}(g_{1})(\xi-v,n-n')|dv\Big\|_{L^{2}_{\xi}l^{2}_{n}}\\
\lesssim& (2^{k}+2^{l})(2^{k_{\mathrm{max}}}+2^{l_{\mathrm{max}}})^{2}\|\mathcal{F}_{x,y}(g_{2})\|_{L^{2}_{\xi}l^{2}_{n}}
\|\mathcal{F}_{x,y}(g_{3})\|_{L^{2}_{\xi}l^{2}_{n}}\|\varphi_{k_{4}}(\xi)\psi_{l_{4}}(n)\|_{L^{1}_{\xi}l^{1}_{n}}
\|\mathcal{F}_{x,y}(g_{1})\|_{L^{2}_{\xi}l^{2}_{n}}\\
\lesssim& (2^{k}+2^{l})(2^{k_{\mathrm{max}}}+2^{l_{\mathrm{max}}})^{2}2^{2k_{4}}2^{l_{4}}
\|\mathcal{F}_{x,y}(g_{1})\|_{L^{2}_{\xi}l^{2}_{n}}
\|\mathcal{F}_{x,y}(g_{2})\|_{L^{2}_{\xi}l^{2}_{n}}\|\mathcal{F}_{x,y}(g_{3})\|_{L^{2}_{\xi}l^{2}_{n}}.
\end{aligned}
\end{equation}

Under the preparations above, \eqref{7-25AA1} is reduced to establish the related estimates of
each term in the right hand side in \eqref{7-31AA6}. These estimates are shown in the
following {\bf Step 3-A} to {\bf Step 3-D} when $j\geq M'(k, s)$:
\vskip 0.1 true cm

\textbf{Step 3-A}: Estimate $\|\varphi_{j}(x)\cdot T^{\mu\sigma\lambda\omega;k,l}_{k_{4},l_{4}}[R_{k_{1}}S_{l_{1}}V_{\sigma},R_{k_{2}}
S_{l_{2}}V_{\lambda},R_{k_{3}}S_{l_{3}}V_{\omega}]\|_{L^{2}(\mathbb{R}^{2}\times\mathbb{T})}$ when $l_{\mathrm{max}}\geq 9j/N$.\vskip 0.2cm

Using \eqref{7-35AA10} and \eqref{6-12A}, we have
\begin{equation}\label{7-36AA11}
\begin{aligned}
&\sum_{\mbox{\tiny$\begin{array}{c}l_{1},l_{2},l_{3}\geq-1,l_{\mathrm{max}}\geq 9j/N\\
(l_{1},l_{2},l_{3})\in \mathcal{Y}_{l}\end{array}$}}
\sum_{l_{4}\geq-1} 2^{j}\Big\|\varphi_{j}(x)\cdot T^{\mu\sigma\lambda\omega;k,l}_{k_{4},l_{4}}[R_{k_{1}}S_{l_{1}}V_{\sigma},R_{k_{2}}
S_{l_{2}}V_{\lambda},R_{k_{3}}S_{l_{3}}V_{\omega}]\Big\|_{L^{2}(\mathbb{R}^{2}\times\mathbb{T})}\\
\lesssim& \sum_{\mbox{\tiny$\begin{array}{c}l_{1},l_{2},l_{3}\geq-1,l_{\mathrm{max}}\geq 9j/N\\
(l_{1},l_{2},l_{3})\in \mathcal{Y}_{l}\end{array}$}}
\sum_{-1\leq l_{4}\leq \mathrm{max}(l_{2},l_{3})+8}
2^{j}\Big\|\mathcal{F}_{x,y}(T^{\mu\sigma\lambda\omega;k,l}_{k_{4},l_{4}}
[R_{k_{1}}S_{l_{1}}V_{\sigma},R_{k_{2}}S_{l_{2}}V_{\lambda},R_{k_{3}}S_{l_{3}}V_{\omega}])\Big\|_{L^{2}_{\xi}l^{2}_{n}}\\
\lesssim& (2^{k}+2^{l})\sum_{\mbox{\tiny$\begin{array}{c}l_{1},l_{2},l_{3}\geq-1,l_{\mathrm{max}}\geq 9j/N\\
(l_{1},l_{2},l_{3})\in \mathcal{Y}_{l}\end{array}$}}
\sum_{-1\leq l_{4}\leq \mathrm{max}(l_{2},l_{3})+8}
2^{j}(2^{k_{\mathrm{max}}}+2^{l_{\mathrm{max}}})^{2}2^{2k_{4}}2^{l_{4}}
\|\mathcal{F}_{x,y}(R_{k_{1}}S_{l_{1}}V_{\sigma})\|_{L^{2}_{\xi}l^{2}_{n}}\\
&~~~~~~~~~~~~~~~~~~~~~~~~~~~
~~~~~~~~~~~~~~~~~~~~~~~~~~~~~~~\cdot\|\mathcal{F}_{x,y}(R_{k_{2}}S_{l_{2}}V_{\lambda})\|_{L^{2}_{\xi}l^{2}_{n}}
\|\mathcal{F}_{x,y}(R_{k_{3}}S_{l_{3}}V_{\omega})\|_{L^{2}_{\xi}l^{2}_{n}}\\
\lesssim& (2^{k}+2^{l})\sum_{\mbox{\tiny$\begin{array}{c}l_{1},l_{2},l_{3}\geq-1,l_{\mathrm{max}}\geq 9j/N\\
(l_{1},l_{2},l_{3})\in \mathcal{Y}_{l}\end{array}$}}
\sum_{-1\leq l_{4}\leq \mathrm{max}(l_{2},l_{3})+8}
2^{j}(2^{k_{\mathrm{max}}}+2^{l_{\mathrm{max}}})^{2}2^{2k_{4}}2^{l_{4}}
\varepsilon_{0}(1+s)^{\kappa}2^{-N(k_{1}+l_{1})/2}\\
&~~~~~~~~~~~~~~~~~~~~~~~~~~~~~~~~~~~~~~~~~~~~~~~~~~~~
~~~~~~\cdot\varepsilon_{0}(1+s)^{\kappa}2^{-N(k_{2}+l_{2})/2}
\varepsilon_{0}(1+s)^{\kappa}2^{-N(k_{3}+l_{3})/2}\\
\lesssim& \varepsilon_{0}^{3}(1+s)^{3\kappa}2^{-5j/4}2^{k}2^{l}2^{2k_{\mathrm{max}}}
2^{2k_{4}}2^{-N(k_{1}+k_{2}+k_{3})/2}2^{-(N/4-3)l}.
\end{aligned}
\end{equation}
Since $M'(k,s):=100+\mathrm{max}(\mathrm{log}_{2}(1+s),10k)$ and $k_{4}\leq \mathrm{max}(k_{2},k_{3})+8$ holds
in the support of  the Fourier transformation of $T_{k_4, l_4}^{\mu\sigma\lambda\omega; k, l}[\cdot]$ (see \eqref{7-32AA7}),
then we derive from \eqref{7-36AA11} that
\begin{equation}\label{7-37AA12}
\begin{aligned}
\int_{0}^{t}\Big(&\sum\limits_{j\geq M'(k,s)}2^{j}
\sum_{\mbox{\tiny$\begin{array}{c}l_{1},l_{2},l_{3}\geq-1, l_{\mathrm{max}}\geq 9j/N\\
(l_{1},l_{2},l_{3})\in \mathcal{Y}_{l}\end{array}$}}
\sum_{\mbox{\tiny$\begin{array}{c}
k_{1},k_{2},k_{3}\in [-1,j/10]\\(k_{1},k_{2},k_{3})\in \mathcal{Y}_{k}\end{array}$}}
\sum_{k_{4}\in \mathbb{Z}}
\sum_{l_{4}\geq-1}
\\&\Big\|\varphi_{j}(x)\cdot T^{\mu\sigma\lambda\omega;k,l}_{k_{4},l_{4}}
[R_{k_{1}}S_{l_{1}}V_{\sigma},R_{k_{2}}S_{l_{2}}V_{\lambda},R_{k_{3}}S_{l_{3}}V_{\omega}]
\Big\|_{L^{2}(\mathbb{R}^{2}\times\mathbb{T})}\Big)ds
\lesssim \varepsilon_{0}^{3}2^{-9(k+l)}.
\end{aligned}
\end{equation}
Thus, in order to show \eqref{7-25AA1}, it remains to prove
\begin{equation}\label{7-38AA13}
\begin{aligned}
\int_{0}^{t}&\Big(\sum\limits_{j\geq M'(k,s)}2^{j}
\sum_{\mbox{\tiny$\begin{array}{c}l_{1},l_{2},l_{3}\geq-1,l_{\mathrm{max}}\leq 9j/N\\
(l_{1},l_{2},l_{3})\in \mathcal{Y}_{l}\end{array}$}}
\sum_{\mbox{\tiny$\begin{array}{c}
k_{1},k_{2},k_{3}\in [-1,j/10]\\(k_{1},k_{2},k_{3})\in \mathcal{Y}_{k}\end{array}$}}
\sum_{k_{4}\in \mathbb{Z}}
\sum_{l_{4}\geq-1}
\\&\cdot\Big\|\varphi_{j}(x)\cdot T^{\mu\sigma\lambda\omega;k,l}_{k_{4},l_{4}}
[R_{k_{1}}S_{l_{1}}V_{\sigma},R_{k_{2}}S_{l_{2}}V_{\lambda},R_{k_{3}}S_{l_{3}}V_{\omega}]
\Big\|_{L^{2}(\mathbb{R}^{2}\times\mathbb{T})}\Big)ds
\lesssim \varepsilon_{0}^{3}2^{-9(k+l)}.
\end{aligned}
\end{equation}

\vskip 0.1 true cm

\textbf{Step 3-B}: Estimate $\|\varphi_{j}(x)\cdot T^{\mu\sigma\lambda\omega;k,l}_{k_{4},l_{4}}
[R_{k_{1}}S_{l_{1}}V_{\sigma}-f^{\sigma}_{jk_{1}l_{1}},R_{k_{2}}
S_{l_{2}}V_{\lambda},R_{k_{3}}S_{l_{3}}V_{\omega}]\|_{L^{2}(\mathbb{R}^{2}\times\mathbb{T})}$ when
$l_{\mathrm{max}}\leq 9j/N$.\vskip 0.2cm

Note that $k, k_{1}, k_{2}, k_{3}\in [-1, j/10]$, it follows from the integration by parts
on $R_{k_4, l_4}^{\mu\sigma\lambda\omega; k, l}(s, \cdot)$ in \eqref{7-33AA8} only in the $\xi$ variable  that
when $|x-x_{1}|\geq 2^{j-10}$ and $l_{\mathrm{max}}\leq 9j/N$,
\begin{equation}\label{7-39AA14}
\begin{aligned}
|R^{\mu\sigma\lambda\omega;k,l}_{k_{4},l_{4}}(s,x,y,x_{1},y_{1},x_{2},y_{2},x_{3},y_{3})|\lesssim 2^{2k_{4}}2^{l_{4}}2^{-50j}.
\end{aligned}
\end{equation}

Furthermore, for $f^{\sigma}_{jk_{1}l_{1}}(s,x,y):=\varphi_{[j-4,j+4]}(x)\cdot R_{k_{1}}S_{l_{1}}V_{\sigma}$
defined in \eqref{6-43AA}, one derives from \eqref{7-33AA8}, \eqref{7-39AA14} and \eqref{2-1-0} that
\begin{equation}\label{7-41AA16}
\begin{aligned}
&\int_{0}^{t}\Big(\sum\limits_{j\geq M'(k,s)}2^{j}
\sum_{\mbox{\tiny$\begin{array}{c}l_{1},l_{2},l_{3}\geq-1,l_{\mathrm{max}}\leq 9j/N\\
(l_{1},l_{2},l_{3})\in \mathcal{Y}_{l}\end{array}$}}
\sum_{\mbox{\tiny$\begin{array}{c}
k_{1},k_{2},k_{3}\in [-1,j/10]\\(k_{1},k_{2},k_{3})\in \mathcal{Y}_{k}\end{array}$}}
\sum_{k_{4}\in \mathbb{Z}}
\sum_{l_{4}\geq-1}\\
&~~~~~~~~~~~~~~~~~~~~~~~~~~~~~~~~~~~~~\cdot\Big\|\varphi_{j}(x)\cdot T^{\mu\sigma\lambda\omega;k,l}_{k_{4},l_{4}}
[R_{k_{1}}S_{l_{1}}V_{\sigma}-f^{\sigma}_{jk_{1}l_{1}},R_{k_{2}}S_{l_{2}}V_{\lambda},R_{k_{3}}S_{l_{3}}V_{\omega}]
\Big\|_{L^{2}(\mathbb{R}^{2}\times\mathbb{T})}\Big)ds\\
\lesssim& \int_{0}^{t}\Big(\sum\limits_{j\geq M'(k,s)}2^{2j}
\sum_{\mbox{\tiny$\begin{array}{c}l_{1},l_{2},l_{3}\geq-1,l_{\mathrm{max}}\leq 9j/N\\
(l_{1},l_{2},l_{3})\in \mathcal{Y}_{l}\end{array}$}}
\sum_{\mbox{\tiny$\begin{array}{c}
k_{1},k_{2},k_{3}\in [-1,j/10]\\(k_{1},k_{2},k_{3})\in \mathcal{Y}_{k}\end{array}$}}
\sum_{\mbox{\tiny$\begin{array}{c}
k_{4}\leq \mathrm{max}(k_{2},k_{3})+8\end{array}$}}
\sum_{\mbox{\tiny$\begin{array}{c}
-1\leq l_{4}\leq \mathrm{max}(l_{2},l_{3})+8\end{array}$}}\\
&~~~~~~~~~~~~~~~~~~~~~~~~~~~~~~~~~~~~~\cdot\Big\|\varphi_{j}(x)\cdot T^{\mu\sigma\lambda\omega;k,l}_{k_{4},l_{4}}
[R_{k_{1}}S_{l_{1}}V_{\sigma}-f^{\sigma}_{jk_{1}l_{1}},R_{k_{2}}S_{l_{2}}V_{\lambda},R_{k_{3}}S_{l_{3}}V_{\omega}]
\Big\|_{L^{\infty}(\mathbb{R}^{2}\times\mathbb{T})}\Big)ds\\
\lesssim& \int_{0}^{t}\Big(\sum\limits_{j\geq M'(k,s)}2^{2j-50j}
\sum_{\mbox{\tiny$\begin{array}{c}l_{1},l_{2},l_{3}\geq-1,l_{\mathrm{max}}\leq 9j/N\\
(l_{1},l_{2},l_{3})\in \mathcal{Y}_{l}\end{array}$}}
\sum_{\mbox{\tiny$\begin{array}{c}
k_{1},k_{2},k_{3}\in [-1,j/10]\\(k_{1},k_{2},k_{3})\in \mathcal{Y}_{k}\end{array}$}}
\sum_{\mbox{\tiny$\begin{array}{c} k_{4}\leq \mathrm{max}(k_{2},k_{3})+8\end{array}$}}
\sum_{\mbox{\tiny$\begin{array}{c} -1\leq l_{4}\leq \mathrm{max}(l_{2},l_{3})+8\end{array}$}}\\
&~~~~~~~~~~~~~~~~~~~~
~~\cdot2^{2k_{4}}2^{l_{4}}\|R_{k_{1}}S_{l_{1}}V_{\sigma}-f^{\sigma}_{jk_{1}l_{1}}\|_{L^{1}(\mathbb{R}^{2}\times\mathbb{T})}
\|R_{k_{2}}S_{l_{2}}V_{\lambda}\|_{L^{1}(\mathbb{R}^{2}\times\mathbb{T})}
\|R_{k_{3}}S_{l_{3}}V_{\omega}\|_{L^{1}(\mathbb{R}^{2}\times\mathbb{T})}\Big)ds\\
\lesssim& \int_{0}^{t}\Big(\sum\limits_{j\geq M'(k,s)}2^{-48j}
\sum_{\mbox{\tiny$\begin{array}{c}l_{1},l_{2},l_{3}\geq-1,l_{\mathrm{max}}\leq 9j/N\\
(l_{1},l_{2},l_{3})\in \mathcal{Y}_{l}\end{array}$}}
\sum_{\mbox{\tiny$\begin{array}{c}
k_{1},k_{2},k_{3}\in [-1,j/10]\\(k_{1},k_{2},k_{3})\in \mathcal{Y}_{k}\end{array}$}}
2^{2\mathrm{max}(k_{2},k_{3})}2^{\mathrm{max}(l_{2},l_{3})}\\
&~~~~~~~~~~~~~~~~~~~~~~~~~~~\cdot2^{-9(k_{1}+l_{1})}\|V_{\sigma}\|_{Z}
2^{-9(k_{2}+l_{2})}\|V_{\lambda}\|_{Z}2^{-9(k_{3}+l_{3})}\|V_{\omega}\|_{Z}\Big)ds\\
\lesssim& \varepsilon_{0}^{3}2^{-9(k+l)}.
\end{aligned}
\end{equation}

\vskip 0.1 true cm

\textbf{Step 3-C}: Estimate $\|\varphi_{j}(x)\cdot T^{\mu\sigma\lambda\omega;k,l}_{k_{4},l_{4}}
[f^{\sigma}_{jk_{1}l_{1}},R_{k_{2}}S_{l_{2}}V_{\lambda},R_{k_{3}}S_{l_{3}}V_{\omega}]
\|_{L^{2}(\mathbb{R}^{2}\times\mathbb{T})}$ when $l_{\mathrm{max}}\leq 9j/N$\\
and $k_{4}\leq -j/(1+\beta)$ with $\beta:=10^{-3}$.\vskip 0.2cm

For $j\geq M'(k,s)$, utilizing \eqref{7-35AA10}, the definition of $Z$-norm and \eqref{6-12A}, we obtain
\begin{equation*}\label{7-42AA17}
\begin{aligned}
&2^{j}\sum_{\mbox{\tiny$\begin{array}{c}l_{1},l_{2},l_{3}\geq-1,l_{\mathrm{max}}\leq 9j/N\\
(l_{1},l_{2},l_{3})\in \mathcal{Y}_{l}\end{array}$}}
\sum_{\mbox{\tiny$\begin{array}{c}
k_{1},k_{2},k_{3}\in [-1,j/10]\\(k_{1},k_{2},k_{3})\in \mathcal{Y}_{k}\end{array}$}}
\sum_{\mbox{\tiny$\begin{array}{c} k_{4}\leq -j/(1+\beta)\end{array}$}}
\sum_{l_{4}\geq-1}\\
&~~~~~~~~~~~~~~~~~~~~~~~~~~~~~~~~~~~~~~~~~~~
~~~~~~~~\Big\|\varphi_{j}(x)\cdot T^{\mu\sigma\lambda\omega;k,l}_{k_{4},l_{4}}
[f^{\sigma}_{jk_{1}l_{1}},R_{k_{2}}S_{l_{2}}V_{\lambda},R_{k_{3}}S_{l_{3}}V_{\omega}]
\Big\|_{L^{2}(\mathbb{R}^{2}\times\mathbb{T})}\\
\lesssim& \sum_{\mbox{\tiny$\begin{array}{c}l_{1},l_{2},l_{3}\geq-1,l_{\mathrm{max}}\leq 9j/N\\
(l_{1},l_{2},l_{3})\in \mathcal{Y}_{l}\end{array}$}}
\sum_{\mbox{\tiny$\begin{array}{c}
k_{1},k_{2},k_{3}\in [-1,j/10] \\(k_{1},k_{2},k_{3})\in \mathcal{Y}_{k}\end{array}$}}
\sum_{\mbox{\tiny$\begin{array}{c} k_{4}\leq -j/(1+\beta)\end{array}$}}
\sum_{\mbox{\tiny$\begin{array}{c} -1\leq l_{4}\leq \mathrm{max}(l_{2},l_{3})+8\end{array}$}}
(2^{k}+2^{l})(2^{k_{\mathrm{max}}}+2^{l_{\mathrm{max}}})^{2}2^{2k_{4}}2^{l_{4}}\\
&~~~~~~~~~~~\cdot\Big(\sum\limits_{j\geq M'(k,s)}2^{j}
\|\varphi_{[j-4,j+4]}(x)\cdot R_{k_{1}}S_{l_{1}}V_{\sigma}\|_{L^{2}(\mathbb{R}^{2}\times\mathbb{T})}\Big)
\|R_{k_{2}}S_{l_{2}}V_{\lambda}\|_{L^{2}(\mathbb{R}^{2}\times\mathbb{T})}
\|R_{k_{3}}S_{l_{3}}V_{\omega}\|_{L^{2}(\mathbb{R}^{2}\times\mathbb{T})}\\
\lesssim& 2^{-2j/(1+\beta)}\sum_{\mbox{\tiny$\begin{array}{c}l_{1},l_{2},l_{3}\geq-1,l_{\mathrm{max}}\leq 9j/N\\
(l_{1},l_{2},l_{3})\in \mathcal{Y}_{l}\end{array}$}}
\sum_{\mbox{\tiny$\begin{array}{c}
k_{1},k_{2},k_{3}\in [-1,j/10]\\(k_{1},k_{2},k_{3})\in \mathcal{Y}_{k}\end{array}$}}
\sum_{\mbox{\tiny$\begin{array}{c}-1\leq l_{4}\leq \mathrm{max}(l_{2},l_{3})+8\end{array}$}}
2^{k}2^{l}2^{2k_{\mathrm{max}}}2^{2l_{\mathrm{max}}}2^{l_{4}}\\
&~~~~~~~~~~~~~~~~~~~~~~~~~~~~~~
~~~~~~~~~~~~~~\cdot2^{-9(k_{1}+l_{1})}\|V_{\sigma}\|_{Z}
\varepsilon_{0}2^{-N_{0}(k_{2}+l_{2})/2}
\varepsilon_{0}2^{-N_{0}(k_{3}+l_{3})/2}\\
\lesssim& \varepsilon_{0}^{3}2^{-2j/(1+\beta)}2^{3j/10}2^{36j/N}2^{-9(k+l)}.
\end{aligned}
\end{equation*}
Hence,
\begin{equation}\label{7-42AA666}
\begin{aligned}
&\int_{0}^{t}\Big(\sum\limits_{j\geq M'(k,s)}2^{j}
\sum_{\mbox{\tiny$\begin{array}{c}l_{1},l_{2},l_{3}\geq-1,l_{\mathrm{max}}\leq 9j/N\\
(l_{1},l_{2},l_{3})\in \mathcal{Y}_{l}\end{array}$}}
\sum_{\mbox{\tiny$\begin{array}{c}
k_{1},k_{2},k_{3}\in [-1,j/10]\\(k_{1},k_{2},k_{3})\in \mathcal{Y}_{k}\end{array}$}}
\sum_{\mbox{\tiny$\begin{array}{c} k_{4}\leq -j/(1+\beta)\end{array}$}}
\sum_{l_{4}\geq-1}\\
&~~~~~~~~~~~~~~~~~~~~~~~~~~~~~~~~~~~~~~~~\cdot\Big\|\varphi_{j}(x)\cdot T^{\mu\sigma\lambda\omega;k,l}_{k_{4},l_{4}}
[f^{\sigma}_{jk_{1}l_{1}},R_{k_{2}}S_{l_{2}}V_{\lambda},R_{k_{3}}S_{l_{3}}V_{\omega}]
\Big\|_{L^{2}(\mathbb{R}^{2}\times\mathbb{T})}\Big)ds\\
\lesssim& \int_{0}^{t}\Big(\sum\limits_{j\geq M'(k,s)}\varepsilon_{0}^{3}2^{-2j/(1+\beta)}2^{3j/10}2^{36j/N}2^{-9(k+l)}\Big)ds\\
\lesssim& \varepsilon_{0}^{3}2^{-9(k+l)}\cdot\int_{0}^{t}(1+s)^{-11/10}\Big(\sum\limits_{j\geq 100}2^{-j/20}\Big)ds\lesssim \varepsilon_{0}^{3}2^{-9(k+l)}.
\end{aligned}
\end{equation}

\vskip 0.1 true cm

\textbf{Step 3-D}: Estimate $\|\varphi_{j}(x)\cdot T^{\mu\sigma\lambda\omega;k,l}_{k_{4},l_{4}}
[f^{\sigma}_{jk_{1}l_{1}},R_{k_{2}}S_{l_{2}}V_{\lambda},R_{k_{3}}S_{l_{3}}V_{\omega}]
\|_{L^{2}(\mathbb{R}^{2}\times\mathbb{T})}$ when $l_{\mathrm{max}}\leq 9j/N$\\
and $k_{4}\geq -j/(1+\beta)$ with $\beta:=10^{-3}$.\vskip 0.2cm

In this case, for $R^{\mu\sigma\lambda\omega;k,l}_{k_{4},l_{4}}(s, \cdot)$ in \eqref{7-33AA8}, we firstly assert that
when $|x-x_{1}|+|x_{1}-x_{2}|+|x_{2}-x_{3}|\geq 2^{j-10}$ and $l_{\mathrm{max}}\leq 9j/N$,
\begin{equation}\label{7-43AA18}
|R^{\mu\sigma\lambda\omega;k,l}_{k_{4},l_{4}}(s,x,y,x_{1},y_{1},x_{2},y_{2},x_{3},y_{3})|
\lesssim 2^{2k_{4}}2^{l_{4}}2^{-50j}.
\end{equation}

The estimate of \eqref{7-43AA18} can be obtained via the method of non-stationary phase. Indeed, with the abbreviation $\Phi:=\Phi_{\sigma,\lambda,\omega}^{n,n',n_{3}}(\xi,v,\xi_{3})$, we set
\begin{equation*}
\begin{aligned}
\mathcal{L}:=\frac{\sum\limits_{d=1}^{2}\Big(s\partial_{\xi_{d}}\Phi+x^{d}-x_{1}^{d}\Big)\partial_{\xi_{d}}
+\sum\limits_{d=1}^{2}\Big(s\partial_{v_{d}}\Phi+x_{1}^{d}-x_{2}^{d}\Big)\partial_{\nu_{d}}
+\sum\limits_{d=1}^{2}\Big(s\partial_{\xi^{d}_{3}}\Phi+x_{2}^{d}-x_{3}^{d}\Big)\partial_{\xi^{d}_{3}}}
{i\sum\limits_{d=1}^{2}\Big(s\partial_{\xi_{d}}\Phi+x^{d}-x_{1}^{d}\Big)^{2}
+i\sum\limits_{d=1}^{2}\Big(s\partial_{v_{d}}\Phi+x_{1}^{d}-x_{2}^{d}\Big)^{2}
+i\sum\limits_{d=1}^{2}\Big(s\partial_{\xi^{d}_{3}}\Phi+x_{2}^{d}-x_{3}^{d}\Big)^{2}},
\end{aligned}
\end{equation*}
then one has
\begin{equation*}
\begin{aligned}
\mathcal{L}(e^{i[(x-x_{1})\cdot\xi+(x_{1}-x_{2})\cdot v+(x_{2}-x_{3})\cdot \xi_{3}]}e^{is \Phi})
=e^{i[(x-x_{1})\cdot\xi+(x_{1}-x_{2})\cdot v+(x_{2}-x_{3})\cdot \xi_{3}]}e^{is \Phi}.
\end{aligned}
\end{equation*}

With $j\geq M'(k, s)$ and the notations in \eqref{7-27AA3}, \eqref{2-10}, one has
\begin{equation*}
2^{j}\geq 2^{100+\mathrm{max}(\mathrm{log}_{2}(1+s),10k)}
\end{equation*}
and
\begin{equation*}
s|\partial_{\xi_{d}}\Phi|+s|\partial_{v_{d}}\Phi|
+s|\partial_{\xi^{d}_{3}}\Phi|\leq 6(1+s)~~~(d=1,2).
\end{equation*}
Combining these two inequalities yields that when $|x-x_{1}|+|x_{1}-x_{2}|+|x_{2}-x_{3}|\geq 2^{j-10}$,
\begin{equation*}
\begin{aligned}
&\sum\limits_{d=1}^{2}\Big(s\partial_{\xi_{d}}\Phi+x^{d}-x_{1}^{d}\Big)^{2}
+\sum\limits_{d=1}^{2}\Big(s\partial_{v_{d}}\Phi+x_{1}^{d}-x_{2}^{d}\Big)^{2}
+\sum\limits_{d=1}^{2}\Big(s\partial_{\xi^{d}_{3}}\Phi+x_{2}^{d}-x_{3}^{d}\Big)^{2}\\
\gtrsim& \Big(|x-x_{1}|+|x_{1}-x_{2}|+|x_{2}-x_{3}|\Big)^{2}.
\end{aligned}
\end{equation*}
Therefore, by a similar
argument as in \eqref{6-42A}, \eqref{7-43AA18} can be obtained.

As a consequence of \eqref{7-43AA18}, along the same method to derive \eqref{7-41AA16}, we arrive at
\begin{equation}\label{7-44A24}
\begin{aligned}
&\int_{0}^{t}\Big(\sum\limits_{j\geq M'(k,s)}2^{j}
\sum_{\mbox{\tiny$\begin{array}{c}l_{1},l_{2},l_{3}\geq-1,l_{\mathrm{max}}\leq 9j/N\\
(l_{1},l_{2},l_{3})\in \mathcal{Y}_{l}\end{array}$}}
\sum_{\mbox{\tiny$\begin{array}{c}
k_{1},k_{2},k_{3}\in [-1,j/10] \\(k_{1},k_{2},k_{3})\in \mathcal{Y}_{k}\end{array}$}}
\sum_{\mbox{\tiny$\begin{array}{c} k_{4}\geq -j/(1+\beta)\end{array}$}}
\sum_{l_{4}\geq-1}
\\&~~~~~~~~~~~~~~~~~~~~
~~~~~~~~~~~~~~~\Big\|\varphi_{j}(x)\cdot T^{\mu\sigma\lambda\omega;k,l}_{k_{4},l_{4}}
[f^{\sigma}_{jk_{1}l_{1}},R_{k_{2}}S_{l_{2}}V_{\lambda}-f^{\lambda}_{jk_{2}l_{2}},R_{k_{3}}S_{l_{3}}V_{\omega}]
\Big\|_{L^{2}(\mathbb{R}^{2}\times\mathbb{T})}\Big)ds
\\&+\int_{0}^{t}\Big(\sum\limits_{j\geq M'(k,s)}2^{j}
\sum_{\mbox{\tiny$\begin{array}{c}l_{1},l_{2},l_{3}\geq-1,l_{\mathrm{max}}\leq 9j/N\\
(l_{1},l_{2},l_{3})\in \mathcal{Y}_{l}\end{array}$}}
\sum_{\mbox{\tiny$\begin{array}{c}
k_{1},k_{2},k_{3}\in [-1,j/10] \\(k_{1},k_{2},k_{3})\in \mathcal{Y}_{k}\end{array}$}}
\sum_{\mbox{\tiny$\begin{array}{c} k_{4}\geq -j/(1+\beta)\end{array}$}}
\sum_{l_{4}\geq-1}
\\&~~~~~~~~~~~~~~~~~~~~
~~~~~~~~~~~~~~~~~~~\Big\|\varphi_{j}(x)\cdot T^{\mu\sigma\lambda\omega;k,l}_{k_{4},l_{4}}
[f^{\sigma}_{jk_{1}l_{1}},f^{\lambda}_{jk_{2}l_{2}},R_{k_{3}}S_{l_{3}}V_{\omega}-f^{\omega}_{jk_{3}l_{3}}]
\Big\|_{L^{2}(\mathbb{R}^{2}\times\mathbb{T})}\Big)ds
\\&\lesssim \varepsilon_{0}^{3}2^{-9(k+l)}.
\end{aligned}
\end{equation}
On the other hand, by \eqref{7-35AA10} and the definition of $Z$-norm, one has
\begin{equation*}
\begin{aligned}
&2^{j}\sum_{\mbox{\tiny$\begin{array}{c}l_{1},l_{2},l_{3}\geq-1,l_{\mathrm{max}}\leq 9j/N\\
(l_{1},l_{2},l_{3})\in \mathcal{Y}_{l}\end{array}$}}
\sum_{\mbox{\tiny$\begin{array}{c}
k_{1},k_{2},k_{3}\in [-1,j/10] \\(k_{1},k_{2},k_{3})\in \mathcal{Y}_{k}\end{array}$}}
\sum_{\mbox{\tiny$\begin{array}{c}k_{4}\geq -j/(1+\beta)\end{array}$}}
\sum_{l_{4}\geq-1}\\
&~~~~~~~~~~~~~~~~~~~~~~~~~~~~~~~~~~~~~
~~~~~~~~~~~~~~~\cdot\Big\|\varphi_{j}(x)\cdot T^{\mu\sigma\lambda\omega;k,l}_{k_{4},l_{4}}
[f^{\sigma}_{jk_{1}l_{1}},f^{\lambda}_{jk_{2}l_{2}},f^{\omega}_{jk_{3}l_{3}}]
\Big\|_{L^{2}(\mathbb{R}^{2}\times\mathbb{T})}\\
\lesssim& 2^{j}\sum_{\mbox{\tiny$\begin{array}{c}l_{1},l_{2},l_{3}\geq-1,l_{\mathrm{max}}\leq 9j/N\\
(l_{1},l_{2},l_{3})\in \mathcal{Y}_{l}\end{array}$}}
\sum_{\mbox{\tiny$\begin{array}{c}
k_{1},k_{2},k_{3}\in [-1,j/10] \\(k_{1},k_{2},k_{3})\in \mathcal{Y}_{k}\end{array}$}}
\sum_{\mbox{\tiny$\begin{array}{c} k_{4}\geq -j/(1+\beta)\\k_{4}\leq \mathrm{max}(k_{2},k_{3})+8\end{array}$}}
\sum_{\mbox{\tiny$\begin{array}{c} -1\leq l_{4}\leq \mathrm{max}(l_{2},l_{3})+8\end{array}$}}\\
&~~~~~~~~~~~~~~~~~~~~~~~~~~~~~~~~~~~~~
~~~~~~~~~~~~~~~\cdot\Big\|\mathcal{F}_{x,y}(T^{\mu\sigma\lambda\omega;k,l}_{k_{4},l_{4}}
[f^{\sigma}_{jk_{1}l_{1}},f^{\lambda}_{jk_{2}l_{2}},f^{\omega}_{jk_{3}l_{3}}])
\Big\|_{L^{2}_{\xi}l^{2}_{n}}\\
\lesssim& 2^{j}
\sum_{\mbox{\tiny$\begin{array}{c}l_{1},l_{2},l_{3}\geq-1,l_{\mathrm{max}}\leq 9j/N\\
(l_{1},l_{2},l_{3})\in \mathcal{Y}_{l}\end{array}$}}
\sum_{\mbox{\tiny$\begin{array}{c}
k_{1},k_{2},k_{3}\in [-1,j/10]\\(k_{1},k_{2},k_{3})\in \mathcal{Y}_{k}\end{array}$}}
\sum_{\mbox{\tiny$\begin{array}{c} k_{4}\geq -j/(1+\beta)\\k_{4}\leq \mathrm{max}(k_{2},k_{3})+8\end{array}$}}
\sum_{\mbox{\tiny$\begin{array}{c}-1\leq l_{4}\leq \mathrm{max}(l_{2},l_{3})+8\end{array}$}}
(2^{k}+2^{l})\\
&~~~~~~~~~\cdot(2^{k_{\mathrm{max}}}+2^{l_{\mathrm{max}}})^{2}2^{2k_{4}}2^{l_{4}}
\|\mathcal{F}_{x,y}(f^{\sigma}_{jk_{1}l_{1}})\|_{L^{2}_{\xi}l^{2}_{n}}
\|\mathcal{F}_{x,y}(f^{\lambda}_{jk_{2}l_{2}})\|_{L^{2}_{\xi}l^{2}_{n}}
\|\mathcal{F}_{x,y}(f^{\omega}_{jk_{3}l_{3}})\|_{L^{2}_{\xi}l^{2}_{n}}\\
\lesssim& \sum_{\mbox{\tiny$\begin{array}{c}l_{1},l_{2},l_{3}\geq-1,l_{\mathrm{max}}\leq 9j/N\\
(l_{1},l_{2},l_{3})\in \mathcal{Y}_{l}\end{array}$}}
\sum_{\mbox{\tiny$\begin{array}{c}
k_{1},k_{2},k_{3}\in [-1,j/10]\\(k_{1},k_{2},k_{3})\in \mathcal{Y}_{k}\end{array}$}}
\sum_{\mbox{\tiny$\begin{array}{c} k_{4}\geq -j/(1+\beta)\\k_{4}\leq \mathrm{max}(k_{2},k_{3})+8\end{array}$}}
\sum_{\mbox{\tiny$\begin{array}{c} -1\leq l_{4}\leq \mathrm{max}(l_{2},l_{3})+8\end{array}$}}
2^{k}2^{l}2^{2k_{\mathrm{max}}}2^{2l_{\mathrm{max}}}2^{2k_{4}}2^{l_{4}}\\
&~~~~~~~~~~~~~~~~~~~~~~~~~~~~
~~~~~~~~~~~~\cdot\Big(2^{j}\|f^{\sigma}_{jk_{1}l_{1}}\|_{L^{2}(\mathbb{R}^{2}\times\mathbb{T})}\Big)
\|f^{\lambda}_{jk_{2}l_{2}}\|_{L^{2}(\mathbb{R}^{2}\times\mathbb{T})}
\|f^{\omega}_{jk_{3}l_{3}}\|_{L^{2}(\mathbb{R}^{2}\times\mathbb{T})}\\
\lesssim& \sum_{\mbox{\tiny$\begin{array}{c}l_{1},l_{2},l_{3}\geq-1,l_{\mathrm{max}}\leq 9j/N\\
(l_{1},l_{2},l_{3})\in \mathcal{Y}_{l}\end{array}$}}
\sum_{\mbox{\tiny$\begin{array}{c}
k_{1},k_{2},k_{3}\in [-1,j/10] \\(k_{1},k_{2},k_{3})\in \mathcal{Y}_{k}\end{array}$}}
\sum_{\mbox{\tiny$\begin{array}{c} k_{4}\geq -j/(1+\beta)\\k_{4}\leq \mathrm{max}(k_{2},k_{3})+8\end{array}$}}
\sum_{\mbox{\tiny$\begin{array}{c}-1\leq l_{4}\leq \mathrm{max}(l_{2},l_{3})+8\end{array}$}}
2^{k}2^{l}2^{2k_{\mathrm{max}}}2^{2l_{\mathrm{max}}}2^{2k_{4}}2^{l_{4}}\\
&~~~~~~~~~~~~~~~~~~~~~~~~~
~~~~~~~~~~~~~~\cdot2^{-9(k_{1}+l_{1})}\|V_{\sigma}\|_{Z}\cdot2^{-j}2^{-9(k_{2}+l_{2})}\|V_{\lambda}\|_{Z}
\cdot2^{-j}2^{-9(k_{3}+l_{3})}\|V_{\omega}\|_{Z}\\
\lesssim& \varepsilon_{0}^{3}2^{-2j}\cdot
\sum_{\mbox{\tiny$\begin{array}{c}
l_{1},l_{2},l_{3}\geq-1,l_{\mathrm{max}}\leq 9j/N\\
(l_{1},l_{2},l_{3})\in \mathcal{Y}_{l}\end{array}$}}
\sum_{\mbox{\tiny$\begin{array}{c}
k_{1},k_{2},k_{3}\in [-1,j/10] \\(k_{1},k_{2},k_{3})\in \mathcal{Y}_{k}\end{array}$}}
2^{k}2^{l}2^{4k_{\mathrm{max}}}2^{3l_{\mathrm{max}}}2^{-9(k_{1}+l_{1})}2^{-9(k_{2}+l_{2})}2^{-9(k_{3}+l_{3})}\\
\lesssim& \varepsilon_{0}^{3}2^{-2j}2^{j/2}2^{36j/N}2^{-9(k+l)}.
\end{aligned}
\end{equation*}
Hence,
\begin{equation}\label{7-46A27}
\begin{aligned}
&\int_{0}^{t}\Big(\sum\limits_{j\geq M'(k,s)}2^{j}\sum_{\mbox{\tiny$\begin{array}{c}
l_{1},l_{2},l_{3}\geq-1,l_{\mathrm{max}}\leq 9j/N\\
(l_{1},l_{2},l_{3})\in \mathcal{Y}_{l}\end{array}$}}
\sum_{\mbox{\tiny$\begin{array}{c}
k_{1},k_{2},k_{3}\in [-1,j/10]\\(k_{1},k_{2},k_{3})\in \mathcal{Y}_{k}\end{array}$}}
\sum_{\mbox{\tiny$\begin{array}{c} k_{4}\geq -j/(1+\beta)\end{array}$}}
\sum_{l_{4}\geq-1}
\\&~~~~~~~~~~~
\cdot\Big\|\varphi_{j}(x)\cdot T^{\mu\sigma\lambda\omega;k,l}_{k_{4},l_{4}}
[f^{\sigma}_{jk_{1}l_{1}},f^{\lambda}_{jk_{2}l_{2}},f^{\omega}_{jk_{3}l_{3}}]
\Big\|_{L^{2}(\mathbb{R}^{2}\times\mathbb{T})}\Big)ds\lesssim \varepsilon_{0}^{3}2^{-9(k+l)}.
\end{aligned}
\end{equation}
Collecting \eqref{7-37AA12}, \eqref{7-41AA16}, \eqref{7-42AA666}, \eqref{7-44A24} and \eqref{7-46A27} yields \eqref{7-25AA1}.

\vskip 0.2 true cm

\textbf{Step 4}: Treat the case of $M(k,s)\leq j\leq M'(k,s)$ in \eqref{6-103AAH}.\vskip 0.2 cm

In view of \eqref{7-25AA1}, we may assume $M(k,s)< M'(k,s)$. Namely,
$\mathrm{log}_{2}(1+s)\geq 10k$ and $M'(k,s)=100+\mathrm{log}_{2}(1+s)$.
In the following, we just need to show
\begin{equation}\label{7-47aaBBCA}
\begin{aligned}
&2^{9(k+l)}\int_{0}^{t}\Big(\sum\limits_{M(k,s)\leq j\leq 100+\mathrm{log}_{2}(1+s)}2^{j}\Big\|\varphi_{j}(x)\cdot R_{k}S_{l}F^{\mu\sigma\lambda\omega}_{j}(s)\Big\|_{L^{2}(\mathbb{R}^{2}\times\mathbb{T})}\Big)ds
\lesssim c_{0}\varepsilon_{0},
\end{aligned}
\end{equation}
where
\begin{equation}\label{7-48aaAA3}
\begin{aligned}
\mathcal{F}_{x,y}(R_{k}S_{l}F^{\mu\sigma\lambda\omega}_{j})(s,\xi,n)=
\sum_{\mbox{\tiny$\begin{array}{c}l_{1},l_{2},l_{3}\geq-1\\
(k_{1},k_{2},k_{3})\in \mathcal{Y}_{k},(l_{1},l_{2},l_{3})\in \mathcal{Y}_{l}\\
k_{1},k_{2},k_{3}\in [-1,j/10] \end{array}$}}
\mathcal{F}_{x,y}(F^{\mu\sigma\lambda\omega;k,l}_{k_{1},k_{2},k_{3},l_{1},l_{2},l_{3}})(s,\xi,n).
\end{aligned}
\end{equation}
In addition, applying Lemma \ref{HC-9} with $r=2$ to \eqref{7-48aaAA3} derives
\begin{equation}\label{7-48aaas}
\Big\|\mathcal{F}_{x,y}(F^{\mu\sigma\lambda\omega;k,l}_{k_{1},k_{2},k_{3},l_{1},l_{2},l_{3}})(s,\xi,n)\Big\|_{L^{2}_{\xi}l^{2}_{n}}
\lesssim 2^{k+k_{1}+\mathrm{max}(k_{2},k_{3})+k_{3}}2^{3l+2\mathrm{max}(l_{2},l_{3})+l_{3}}\mathrm{min}(J_1, J_2),
\end{equation}
where $J_1$ and $J_2$ have been defined in \eqref{6-67A} and \eqref{6-68A}.

It follows from \eqref{6-10A} and \eqref{6-12A} that
\begin{equation*}
\begin{aligned}
J_1&\lesssim \varepsilon_{0}(1+s)^{-1}2^{-7(k_{l_{\mathrm{min}}}+l_{\mathrm{min}})}
\varepsilon_{0}(1+s)^{-1}2^{-7(k_{l_{\mathrm{med}}}+l_{\mathrm{med}})}
\varepsilon_{0}2^{-N_{0}(k_{l_{\mathrm{max}}}+l_{\mathrm{max}})/2}
\\&\lesssim \varepsilon_{0}^{3}(1+s)^{-2}2^{-7(k_{1}+k_{2}+k_{3})}
2^{-7(l_{1}+l_{2}+l_{3})}2^{-(N_{0}/2-7)l_{\mathrm{max}}},\\[2mm]
J_2&\lesssim \varepsilon_{0}(1+s)^{-1}2^{-7(k_{\mathrm{min}}+l_{k_{\mathrm{min}}})}
\varepsilon_{0}(1+s)^{-1}2^{-7(k_{\mathrm{med}}+l_{k_{\mathrm{med}}})}
\varepsilon_{0}2^{-N_{0}(k_{\mathrm{max}}+l_{k_{\mathrm{max}}})/2}
\\&\lesssim \varepsilon_{0}^{3}(1+s)^{-2}2^{-7(l_{1}+l_{2}+l_{3})}
2^{-7(k_{1}+k_{2}+k_{3})}2^{-(N_{0}/2-7)k_{\mathrm{max}}},
\end{aligned}
\end{equation*}
and then
\begin{equation}\label{7-56baddbaa}
\mathrm{min}(J_1, J_2)\leq J_1^{1/2}J_2^{1/2}
\lesssim \varepsilon_{0}^{3}(1+s)^{-2}2^{-7(l_{1}+l_{2}+l_{3})}
2^{-7(k_{1}+k_{2}+k_{3})}2^{-(N_{0}/4-7/2)(k_{\mathrm{max}}+l_{\mathrm{max}})}.
\end{equation}
Combining \eqref{7-48aaAA3}-\eqref{7-48aaas} with \eqref{7-56baddbaa} shows \eqref{7-47aaBBCA}, \emph{i.e.},
\begin{equation*}
\begin{aligned}
&2^{9(k+l)}\int_{0}^{t}\Big(\sum\limits_{M(k,s)\leq j\leq 100+\mathrm{log}_{2}(1+s)}2^{j}\Big\|\varphi_{j}(x)\cdot R_{k}S_{l}F^{\mu\sigma\lambda\omega}_{j}(s)\Big\|_{L^{2}(\mathbb{R}^{2}\times\mathbb{T})}\Big)ds\\
\lesssim& \varepsilon_{0}^{3}\int_{0}^{t}
\sum\limits_{M(k,s)\leq j\leq 100+\mathrm{log}_{2}(1+s)}2^{j}(1+s)^{-2}ds\\
\lesssim& \varepsilon_{0}^{3}\int_{0}^{e^{c_{0}/\varepsilon_{0}^{2}}}(1+s)^{-1}ds\lesssim c_{0}\varepsilon_{0}.
\end{aligned}
\end{equation*}
Based on all the above analysis, Lemma 6.2 is shown.
\end{proof}

\appendix

\renewcommand{\appendixname}{}

\section{Bilinear and trilinear estimates}\label{A}

\begin{lem}\label{HC-8}
{\it For $\mu,\nu\in \{+,-\}$,~$k,k_{1},k_{2},l,l_{1},l_{2} \geq -1$,~and
$p,q,r\in[1,\infty],~r\geq2,$ $1/p+1/q+1/r=1$, there holds
\begin{equation}\label{7-18AA}
\begin{aligned}
&\Big\|\int_{\mathbb{R}^{2}}\sum_{m\in \mathbb{Z}}
P^{\mu\nu}_{k,k_{1},k_{2},l,l_{1},l_{2}}(\xi,n,\eta,m)
\widehat{f_{n-m}}(\xi-\eta)\widehat{g_{m}}(\eta)d\eta\Big\|_{L^{2}_{\xi}l^{2}_{n}}\\
\lesssim& 2^{l+l_{2}}(2^{k}+2^{l})(2^{k_{1}}+2^{k_{2}}+2^{l_{2}})2^{(k+l/2)(1-2/r)}
\|f\|_{L^{p}(\mathbb{R}^{2}\times \mathbb{T})}\|g\|_{L^{q}(\mathbb{R}^{2}\times \mathbb{T})},
\end{aligned}
\end{equation}
where $P^{\mu\nu}_{k,k_{1},k_{2},l,l_{1},l_{2}}(\xi,n,\eta,m)$ is defined in \eqref{6-18AAA}.}
\end{lem}

\begin{proof}
By duality, the left-hand side of \eqref{7-18AA} can be dominated by
\begin{equation}\label{7-19AA}
\begin{aligned}
&\sup\limits_{\|h(x,y)\|_{L^{2}(\mathbb{R}^{2}\times \mathbb{T})}=1}\Big|\int_{\mathbb{R}^{2}\times \mathbb{T}}h(x,y)
\int_{\mathbb{R}^{2}}e^{ix\cdot \xi}\sum_{n\in \mathbb{Z}}e^{iny}\int_{\mathbb{R}^{2}}\sum_{m\in \mathbb{Z}}
P^{\mu\nu}_{k,k_{1},k_{2},l,l_{1},l_{2}}(\xi,n,\eta,m)\\
&~~~~~~~~~~~~~~~~~~~~~~~~~~~~~~~~~~~~~~~~
~~~~~~~~~~~~~~~~~~~~~~~~~~~~~~\cdot\widehat{f_{n-m}}(\xi-\eta)\widehat{g_{m}}(\eta)d\eta d\xi dxdy\Big|\\
=& \sup\limits_{\|h(x,y)\|_{L^{2}(\mathbb{R}^{2}\times \mathbb{T})}=1}\Big|\int_{\mathbb{R}^{2}\times \mathbb{T}}
\int_{\mathbb{R}^{2}\times \mathbb{T}}\int_{\mathbb{R}^{2}\times \mathbb{T}}
(R_{[k-2,k+2]}S_{[l-2,l+2]}h)(x,y)f(x_{1},y_{1})g(x_{2},y_{2})\\
&~~~~~~~~~~~~~~~~~~~~~~~~~~~~~~~
~~~~~~~~~~~\cdot Q^{\mu\nu}_{k,k_{1},k_{2},l,l_{1},l_{2}}(x,y,x_{1},y_{1},x_{2},y_{2})dxdydx_{1}dy_{1}dx_{2}dy_{2}\Big|,
\end{aligned}
\end{equation}
where
\begin{equation}\label{7-20ABC}
\begin{aligned}
&P^{\mu\nu}_{k,k_{1},k_{2},l,l_{1},l_{2}}(\xi,n,\eta,m)\\
:=& \frac{\mathcal{M}_{n,m}(\xi,\eta)}{\Phi_{\mu,\nu}^{n,m}(\xi,\eta)}
\psi_{k}(\xi)\psi_{[k_{1}-2,k_{1}+2]}(\xi-\eta)\psi_{[k_{2}-2,k_{2}+2]}(\eta)\\
&~~~~~~~~~~~~~~~~
~~~~~~~~~~~~\cdot\psi_{l}(n)\psi_{[l_{1}-2,l_{1}+2]}(n-m)\psi_{[l_{2}-2,l_{2}+2]}(m),\\[2mm]
&Q^{\mu\nu}_{k,k_{1},k_{2},l,l_{1},l_{2}}(x,y,x_{1},y_{1},x_{2},y_{2})\\
:=& \sum_{n\in \mathbb{Z}}\sum_{m\in \mathbb{Z}}e^{i[n(y-y_{1})+m(y_{1}-y_{2})]}
\int_{\mathbb{R}^{2}\times\mathbb{R}^{2}}e^{i[(x-x_{1})\cdot\xi+(x_{1}-x_{2})\cdot\eta]}
P^{\mu\nu}_{k,k_{1},k_{2},l,l_{1},l_{2}}(\xi,n,\eta,m)d\xi d\eta.
\end{aligned}
\end{equation}

Without loss of generality, we assume that $k_{2}\leq k_{1}$ since the case of $k_{1}\leq k_{2}$ can be analogously treated.
By the forms of $\mathcal{M}_{n,m}(\xi,\eta)$ in \eqref{5-7-0}-\eqref{5-7-1},
$\mathcal{M}_{n,m}(\xi,\eta)\psi_{k}(\xi)\psi_{[k_{1}-2,k_{1}+2]}(\xi-\eta)\psi_{[k_{2}-2,k_{2}+2]}(\eta)$ are the sums of
such symbols
\begin{equation}\label{7-21A}
\begin{aligned}
(2^{k_{1}}+2^{k_{2}}+2^{l_{2}})\psi_{k}(\xi)a(n-m,\xi-\eta)b(m,\eta),
\end{aligned}
\end{equation}
where $\sup\limits_{n_{1}\in \mathbb{Z}}\sup\limits_{n_{2}\in \mathbb{Z}}\sup\limits_{w\in \mathbb{R}^{2}}
\sum\limits_{m\in [0,10]}\Big[2^{km}|\partial^{m}_{w}(\psi_{k}(w))|+2^{k_{1}m}|\partial^{m}_{w}(a(n_{1},w))|
+2^{k_{2}m}|\partial^{m}_{w}(b(n_{2},w))|\Big]\lesssim 1.$

Integrating by parts in $\xi$ and $\eta$ in \eqref{7-20ABC} via the method of non-stationary phase and applying
\eqref{7-21A}, \eqref{7-16BAB} and Lemma \ref{HC-11}, we get
\begin{equation}\label{7-23AB}
\begin{aligned}
&\Big|Q^{\mu\nu}_{k,k_{1},k_{2},l,l_{1},l_{2}}(x,y,x_{1},y_{1},x_{2},y_{2})\Big|\\
\lesssim& 2^{l+l_{2}}(2^{k}+2^{l})(2^{k_{1}}+2^{k_{2}}+2^{l_{2}})
\Big(2^{2k_{2}}(1+2^{2k_{2}}|x_{1}-x_{2}|^{2})^{-2}\Big)
\Big(2^{2k}(1+2^{2k}|x-x_{1}|^{2})^{-2}\Big).
\end{aligned}
\end{equation}
Based on \eqref{7-19AA}, \eqref{7-23AB} and Lemma \ref{HC-7}, the left-hand
side of \eqref{7-18AA} can be dominated by
\begin{equation*}
\begin{aligned}
&\Big|\int_{\mathbb{R}^{2}\times \mathbb{T}}
\int_{\mathbb{R}^{2}\times \mathbb{T}}\int_{\mathbb{R}^{2}\times \mathbb{T}}(R_{[k-2,k+2]}S_{[l-2,l+2]}h)(x,y)f(x_{1},y_{1})g(x_{2},y_{2})\\
&~~~~~~~~~~~~~~~~~~~~~~~~~~~~~~
~~~~~~~~~~~~~~~~~~~~~~~~~~~~~~\cdot Q^{\mu\nu}_{k,k_{1},k_{2},l,l_{1},l_{2}}(x,y,x_{1},y_{1},x_{2},y_{2})dxdydx_{1}dy_{1}dx_{2}dy_{2}\Big|\\
\lesssim& \int_{\mathbb{R}^{2}\times \mathbb{T}}
\int_{\mathbb{R}^{2}\times \mathbb{T}}\int_{\mathbb{R}^{2}\times \mathbb{T}}\Big|f(x_{1},y_{1})\Big|\Big|Q^{\mu\nu}_{k,k_{1},k_{2},l,l_{1},l_{2}}\Big|^{1/p}
\Big|g(x_{2},y_{2})\Big|\Big|Q^{\mu\nu}_{k,k_{1},k_{2},l,l_{1},l_{2}}\Big|^{1/q}\\
&~~~~~~~~~~~~~~~~~
~~~~~~~~~~~~~~~~~~~~\cdot\Big|(R_{[k-2,k+2]}S_{[l-2,l+2]}h)(x,y)
\Big|\Big|Q^{\mu\nu}_{k,k_{1},k_{2},l,l_{1},l_{2}}\Big|^{1/r}dxdydx_{1}dy_{1}dx_{2}dy_{2}\\
\lesssim& \Big\|f(x_{1},y_{1})|Q^{\mu\nu}_{k,k_{1},k_{2},l,l_{1},l_{2}}|^{1/p}\Big\|_{L^{p}((\mathbb{R}^{2}\times \mathbb{T})^{3})}
\Big\|g(x_{2},y_{2})|Q^{\mu\nu}_{k,k_{1},k_{2},l,l_{1},l_{2}}|^{1/q}\Big\|_{L^{q}((\mathbb{R}^{2}\times \mathbb{T})^{3})}\\
&~~~~~~~~~~~~~~~~~~~~~~~~
~~~~~~~~~~~~~~~~~~~~~~~~~\cdot\Big\|(R_{[k-2,k+2]}S_{[l-2,l+2]}h)(x,y)
|Q^{\mu\nu}_{k,k_{1},k_{2},l,l_{1},l_{2}}|^{1/r}\Big\|_{L^{r}((\mathbb{R}^{2}\times \mathbb{T})^{3})}\\
\end{aligned}
\end{equation*}

\begin{equation*}
\begin{aligned}
\lesssim& \|f\|_{L^{p}(\mathbb{R}^{2}\times \mathbb{T})}
[2^{l+l_{2}}(2^{k}+2^{l})(2^{k_{1}}+2^{k_{2}}+2^{l_{2}})]^{1/p}\\
&~~~~~~~~~~~~~~~~~~\cdot\|g\|_{L^{q}(\mathbb{R}^{2}
\times \mathbb{T})}[2^{l+l_{2}}(2^{k}+2^{l})(2^{k_{1}}+2^{k_{2}}+2^{l_{2}})]^{1/q}\\
&~~~~~~~~~~~~~~~~~~~~~~~~\cdot\|R_{[k-2,k+2]}S_{[l-2,l+2]}h\|_{L^{r}(\mathbb{R}^{2}\times \mathbb{T})}
[2^{l+l_{2}}(2^{k}+2^{l})(2^{k_{1}}+2^{k_{2}}+2^{l_{2}})]^{1/r}\\
\lesssim& 2^{l+l_{2}}(2^{k}+2^{l})(2^{k_{1}}+2^{k_{2}}+2^{l_{2}})
\|f\|_{L^{p}(\mathbb{R}^{2}\times \mathbb{T})}\|g\|_{L^{q}(\mathbb{R}^{2}\times \mathbb{T})}
\|R_{[k-2,k+2]}S_{[l-2,l+2]}h\|_{L^{2}(\mathbb{R}^{2}\times \mathbb{T})}\\
\lesssim& 2^{l+l_{2}}(2^{k}+2^{l})(2^{k_{1}}+2^{k_{2}}+2^{l_{2}})
2^{(k+l/2)(1-2/r)}\|f\|_{L^{p}(\mathbb{R}^{2}\times \mathbb{T})}\|g\|_{L^{q}(\mathbb{R}^{2}\times \mathbb{T})}.
\end{aligned}
\end{equation*}
This completes the proof of the Lemma \ref{HC-8}.
\end{proof}

\begin{lem}\label{HC-9}
{\it For $k, k_{1},
k_{2}, k_{3}, l, l_{1}, l_{2}, l_{3}\geq -1$, we have from \eqref{6-50AA-1-1} that
\begin{equation}\label{7-13ADE}
\begin{aligned}
&\Big\|\mathcal{F}_{x,y}(F^{\mu\sigma\lambda\omega;k,l}_{k_{1},k_{2},k_{3},l_{1},l_{2},l_{3}})(s,\xi,n)\Big\|_{L^{2}_{\xi}l^{2}_{n}}
\lesssim 2^{k+k_{1}+\mathrm{max}(k_{2},k_{3})+k_{3}}
2^{3l+2\mathrm{max}(l_{2},l_{3})+l_{3}}2^{(k+l/2)(1-2/r)}\\
&~~~~~~~~~~~~~~~~~~~~~~~~~~~~~~~~~~~\cdot\|R_{k_{1}}S_{1}U_{\sigma}\|_{L^{p_{1}}(\mathbb{R}^{2}\times \mathbb{T})}
\|R_{k_{2}}S_{2}U_{\lambda}\|_{L^{p_{2}}(\mathbb{R}^{2}\times \mathbb{T})}
\|R_{k_{3}}S_{3}U_{\omega}\|_{L^{p_{3}}(\mathbb{R}^{2}\times\mathbb{T})},
\end{aligned}
\end{equation}
where $p_{1},p_{2},p_{3},r\in [2,\infty]$ with
$\frac{1}{p_{1}}+\frac{1}{p_{2}}+\frac{1}{p_{3}}+\frac{1}{r}=1.$}
\end{lem}

\begin{proof}
At first, we consider the more general operator $T$ as
\begin{equation}\label{7-14AE212}
\begin{aligned}
&T[f,g,h](\xi,n)\\
=& \psi_{k}(\xi)\psi_{l}(n)\sum_{m\in \mathbb{Z}}\sum_{l'\in \mathbb{Z}}\int_{\mathbb{R}^{2}}
\int_{\mathbb{R}^{2}}\Big(\frac{\mathcal{M}_{n,m}(\xi,\eta)}{\Phi_{\sigma,\mu}^{n,m}(\xi,\eta)}+
\frac{\mathcal{M}_{n,n-m}(\xi,\xi-\eta)}{\Phi_{\mu,\sigma}^{n,n-m}(\xi,\xi-\eta)}\Big)
\mathcal{M}_{m,l'}(\eta,\zeta)\\
&~~~~~~~~~~~~~~\cdot\mathcal{F}_{x,y}(R_{k_{1}}S_{l_{1}}f)(\xi-\eta,n-m)
\mathcal{F}_{x,y}(R_{k_{2}}S_{l_{2}}g)(\eta-\zeta,m-l')
\mathcal{F}_{x,y}(R_{k_{3}}S_{l_{3}}h)(\zeta,l')d\zeta d\eta\\
=& \psi_{k}(\xi)\psi_{l}(n)\sum_{n_{2}\in \mathbb{Z}}\sum_{n_{3}\in \mathbb{Z}}\int_{\mathbb{R}^{2}}
\int_{\mathbb{R}^{2}}\Big(\frac{\mathcal{M}_{n,n_{2}+n_{3}}(\xi,\xi_{2}+\xi_{3})}{\Phi_{\sigma,\mu}^{n,n_{2}+n_{3}}(\xi,\xi_{2}+\xi_{3})}+
\frac{\mathcal{M}_{n,n-n_{2}-n_{3}}(\xi,\xi-\xi_{2}-\xi_{3})}{\Phi_{\mu,\sigma}^{n,n-n_{2}-n_{3}}(\xi,\xi-\xi_{2}-\xi_{3})}\Big)\\
&~~~~~~~~~~~~~~~~~~~~~~~~~~~~~~~~~~~~~~~~~~~~~~~\cdot\mathcal{M}_{n_{2}+n_{3},n_{3}}(\xi_{2}+\xi_{3},\xi_{3})
\mathcal{F}_{x,y}(R_{k_{1}}S_{l_{1}}f)(\xi-\xi_{2}-\xi_{3},n-n_{2}-n_{3})\\
&~~~~~~~~~~~~~~~~~~~~~~~~~~~~~~~~
~~~~~~~~~~~~~~~~~~~~~~~~~~~~~~~~~~\cdot\mathcal{F}_{x,y}(R_{k_{2}}S_{l_{2}}g)(\xi_{2},n_{2})
\mathcal{F}_{x,y}(R_{k_{3}}S_{l_{3}}h)(\xi_{3},n_{3})d\xi_{2} d\xi_{3}\\
=& \psi_{k}(\xi)\psi_{l}(n)\sum_{n_{3}\in \mathbb{Z}}\sum_{n'\in \mathbb{Z}}\int_{\mathbb{R}^{2}}
\int_{\mathbb{R}^{2}}\Big(\frac{\mathcal{M}_{n,n'}(\xi,v)}{\Phi_{\sigma,\mu}^{n,n'}(\xi,v)}+
\frac{\mathcal{M}_{n,n-n'}(\xi,\xi-v)}{\Phi_{\mu,\sigma}^{n,n-n'}(\xi,\xi-v)}\Big)\mathcal{M}_{n',n_{3}}(v,\xi_{3})\\
&~~~~~~~~~~\cdot\mathcal{F}_{x,y}(R_{k_{1}}S_{l_{1}}f)(\xi-v,n-n')
\mathcal{F}_{x,y}(R_{k_{2}}S_{l_{2}}g)(v-\xi_{3},n'-n_{3})
\mathcal{F}_{x,y}(R_{k_{3}}S_{l_{3}}h)(\xi_{3},n_{3})d\xi_{3}dv\\
=& \psi_{k}(\xi)\psi_{l}(n)\sum_{n'\in \mathbb{Z}}
\int_{\mathbb{R}^{2}}\Big(\frac{\mathcal{M}_{n,n'}(\xi,v)}{\Phi_{\sigma,\mu}^{n,n'}(\xi,v)}+
\frac{\mathcal{M}_{n,n-n'}(\xi,\xi-v)}{\Phi_{\mu,\sigma}^{n,n-n'}(\xi,\xi-v)}\Big)
\mathcal{F}_{x,y}(R_{k_{1}}S_{l_{1}}f)(\xi-v,n-n')
\\&~~~~~~~~~~~~~~~~~~~~~~~~~~~~~~~~~~~~~~~~~~~~~~~~~~~
~~~~~~~~~~~~~~~~~~~~~~~~~~~~~~~~~~~~~~~~~~~~~~~~~~~~~\cdot\mathcal{F}_{x,y}(G)(v,n')dv,
\end{aligned}
\end{equation}
where
\begin{equation*}
\begin{aligned}
\mathcal{F}_{x,y}(G)(v,n')&=\sum_{n_{3}\in \mathbb{Z}}
\int_{\mathbb{R}^{2}}\mathcal{M}_{n',n_{3}}(v,\xi_{3})
\mathcal{F}_{x,y}(P_{k_{2}}S_{l_{2}}g)(v-\xi_{3},n'-n_{3})
\mathcal{F}_{x,y}(P_{k_{3}}S_{l_{3}}h)(\xi_{3},n_{3})d\xi_{3}.
\end{aligned}
\end{equation*}

Next, we prove that
\begin{equation}\label{7-13ADE-1}
\begin{aligned}
&\Big\|T[f,g,h](\xi,n)\Big\|_{L^{2}_{\xi}l^{2}_{n}}
\lesssim 2^{k+k_{1}+\mathrm{max}(k_{2},k_{3})+k_{3}}2^{3l+2\mathrm{max}(l_{2},l_{3})+l_{3}}
2^{(k+l/2)(1-2/r)}\\
&~~~~~~~~~~~~~~~~~~~~~~~~~~~~~~~~~~~
~~~~~~~~~~~~~~~\cdot\|f\|_{L^{p_{1}}(\mathbb{R}^{2}\times \mathbb{T})}
\|g\|_{L^{p_{2}}(\mathbb{R}^{2}\times \mathbb{T})}
\|h\|_{L^{p_{3}}(\mathbb{R}^{2}\times\mathbb{T})}.
\end{aligned}
\end{equation}
Obviously, once \eqref{7-13ADE-1} is established, then \eqref{7-13ADE} follows immediately
when $f=U_{\sigma}, g=U_{\lambda}, h=U_{\omega}$ in \eqref{7-13ADE-1}.

In view of \eqref{7-14AE212}, we taking such a decomposition
\begin{equation}\label{A-A-9}
\begin{aligned}
T=\sum_{k_{4}\geq-1}\sum_{l_{4}\geq-1}T_{k_{4},l_{4}},
\end{aligned}
\end{equation}
where
\begin{equation}\label{7-16Ahh}
\begin{aligned}
&T_{k_{4},l_{4}}[f,g,h](\xi,n)\\
:=& \psi_{k}(\xi)\psi_{l}(n)\sum_{n'\in \mathbb{Z}}
\int_{\mathbb{R}^{2}}\Big(\frac{\mathcal{M}_{n,n'}(\xi,v)}{\Phi_{\sigma,\mu}^{n,n'}(\xi,v)}+
\frac{\mathcal{M}_{n,n-n'}(\xi,\xi-v)}{\Phi_{\mu,\sigma}^{n,n-n'}(\xi,\xi-v)}\Big)\\
&~~~~~~~~~~~~~~~~~~~~~~~~~~~~~~~~~~~~\cdot
\mathcal{F}_{x,y}(R_{k_{1}}S_{l_{1}}f)(\xi-v,n-n')
\psi_{k_{4}}(v)\psi_{l_{4}}(n')\mathcal{F}_{x,y}(G)(v,n')dv\\
=& \psi_{k}(\xi)\psi_{l}(n)\sum_{n'\in \mathbb{Z}}
\int_{\mathbb{R}^{2}}\Big(\frac{\mathcal{M}_{n,n'}(\xi,v)}{\Phi_{\sigma,\mu}^{n,n'}(\xi,v)}+
\frac{\mathcal{M}_{n,n-n'}(\xi,\xi-v)}{\Phi_{\mu,\sigma}^{n,n-n'}(\xi,\xi-v)}\Big)\\
&~~~~~~~~~~~~~~~~~~~~~~~~~~~~~~~~~~~~~~~~~~~~~~~~~~\cdot
\mathcal{F}_{x,y}(R_{k_{1}}S_{l_{1}}f)(\xi-v,n-n')
\mathcal{F}_{x,y}(R_{k_{4}}S_{l_{4}}G(v,n')dv.
\end{aligned}
\end{equation}
By Lemma \ref{HC-8}, we have
\begin{equation}\label{A-A-10}
\begin{aligned}
&\|T_{k_{4},l_{4}}[f,g,h]\|_{L^{2}_{\xi}l^{2}_{n}}\\
\lesssim& 2^{l+l_{4}}(2^{k}+2^{l})(2^{k_{1}}+2^{k_{4}}+2^{\mathrm{max}(l_{1},l_{4})})2^{(k+l/2)(1-2/r)}
\|R_{k_{1}}S_{l_{1}}f\|_{L^{p_{1}}(\mathbb{R}^{2}\times\mathbb{T})}
\|R_{k_{4}}S_{l_{4}}G\|_{L^{p_{2}p_{3}/(p_{2}+p_{3})}(\mathbb{R}^{2}\times\mathbb{T})}\\
\lesssim& 2^{l+l_{4}}(2^{k}+2^{l})(2^{k_{1}}+2^{k_{4}}+2^{\mathrm{max}(l_{1},l_{4})})2^{(k+l/2)(1-2/r)}
\|f\|_{L^{p_{1}}(\mathbb{R}^{2}\times\mathbb{T})}\|G\|_{L^{p_{2}p_{3}/(p_{2}+p_{3})}(\mathbb{R}^{2}\times\mathbb{T})},
\end{aligned}
\end{equation}
In addition, it follows from Lemma \ref{HC-6} and Lemma \ref{HC-10} that
\begin{equation}\label{A-A-11}
\begin{aligned}
&\|G\|_{L^{p_{2}p_{3}/(p_{2}+p_{3})}(\mathbb{R}^{2}\times\mathbb{T})}\\
\lesssim& \Big\|\Big[R_{k_{2}}S_{l_{2}}\Big(\sum\limits_{n\in \mathbb{Z}}
\frac{(\partial_{x}~\mathrm{or}~n)}{\sqrt{1-\Delta_{x}+n^{2}}}g_{n}(x)e^{iny}\Big)\Big]\\
&~~~~~~~~~~~~~~~~~~~~~~~~~~~~\cdot \Big[(1,\partial_{x},\partial_{y})R_{k_{3}}S_{l_{3}}
\Big(\sum\limits_{n\in \mathbb{Z}}\frac{(\partial_{x}~\mathrm{or}~n)}{\sqrt{1-\Delta_{x}+n^{2}}}h_{n}(x)e^{iny}
\Big)\Big]\Big\|_{L^{p_{2}p_{3}/(p_{2}+p_{3})}(\mathbb{R}^{2}\times\mathbb{T})}\\
\lesssim& \Big\|R_{k_{2}}S_{l_{2}}\Big(\sum\limits_{n\in \mathbb{Z}}\frac{(\partial_{x}~\mathrm{or}~n)}{\sqrt{1-\Delta_{x}+n^{2}}}g_{n}(x)e^{iny}\Big)\Big\|_{L^{p_{2}}(\mathbb{R}^{2}\times\mathbb{T})}\\
&~~~~~~~~~~~~~~~~~~~~~~~~~~~~~~~~~~~~~~~~\cdot(2^{k_{3}}+2^{l_{3}})\Big\|R_{k_{3}}S_{l_{3}}\Big(\sum\limits_{n\in \mathbb{Z}}\frac{(\partial_{x}~\mathrm{or}~n)}{\sqrt{1-\Delta_{x}+n^{2}}}h_{n}(x)e^{iny}\Big)\Big\|_{L^{p_{3}}(\mathbb{R}^{2}\times\mathbb{T})}\\
\lesssim& (2^{k_{3}}+2^{l_{3}})\|g\|_{L^{p_{2}}(\mathbb{R}^{2}\times\mathbb{T})}\|h\|_{L^{p_{3}}(\mathbb{R}^{2}\times\mathbb{T})}.
\end{aligned}
\end{equation}
On the other hand, on the support of the integrand in \eqref{7-16Ahh}, the following restrictions hold
\begin{equation}\label{A-A-12}
\begin{aligned}
(k_{2},k_{3})\in\chi_{k_{4}},~~~~~(l,l_{4})\in\chi_{l_{1}},~~~~~(l_{2},l_{3})\in\chi_{l_{4}}.
\end{aligned}
\end{equation}
Using \eqref{A-A-9} and \eqref{A-A-10}-\eqref{A-A-12}, the left-hand side of \eqref{7-13ADE-1} can be dominated by
\begin{equation*}
\begin{aligned}
&\sum_{\mbox{\tiny$\begin{array}{c} -1\leq k_{4}\leq \mathrm{max}(k_{2},k_{3})+8\end{array}$}}
\sum_{\mbox{\tiny$\begin{array}{c}-1\leq l_{4}\leq \mathrm{max}(l_{2},l_{3})+8\end{array}$}}
\|T_{k_{4},l_{4}}[f,g,h]\|_{L^{2}_{\xi}l^{2}_{n}}\\
\lesssim& \sum_{\mbox{\tiny$\begin{array}{c} -1\leq k_{4}\leq \mathrm{max}(k_{2},k_{3})+8\end{array}$}}
\sum_{\mbox{\tiny$\begin{array}{c} -1\leq l_{4}\leq \mathrm{max}(l_{2},l_{3})+8\end{array}$}}
2^{l+l_{4}}(2^{k}+2^{l})(2^{k_{1}}+2^{k_{4}}+2^{\mathrm{max}(l_{1},l_{4})})2^{(k+l/2)(1-2/r)}\\
&~~~~~~~~~~~~~~~~~~~~~~~~~~~~~~~~~~~~~~~~~~~~~~~~~~~~~~~~~~~~~~~~~~~~~
~~~~~~~~~~\cdot\|f\|_{L^{p_{1}}(\mathbb{R}^{2}\times\mathbb{T})}\|G\|_{L^{p_{2}p_{3}/(p_{2}+p_{3})}(\mathbb{R}^{2}\times\mathbb{T})}\\
\end{aligned}
\end{equation*}

\begin{equation*}
\begin{aligned}
\lesssim& 2^{k+k_{1}+\mathrm{max}(k_{2},k_{3})}2^{3l+2\mathrm{max}(l_{2},l_{3})}
2^{(k+l/2)(1-2/r)}(2^{k_{3}}+2^{l_{3}})
\|f\|_{L^{p_{1}}(\mathbb{R}^{2}\times\mathbb{T})}\|g\|_{L^{p_{2}}(\mathbb{R}^{2}\times\mathbb{T})}
\|h\|_{L^{p_{3}}(\mathbb{R}^{2}\times\mathbb{T})}\\
\lesssim& 2^{k+k_{1}+\mathrm{max}(k_{2},k_{3})+k_{3}}
2^{3l+2\mathrm{max}(l_{2},l_{3})+l_{3}}2^{(k+l/2)(1-2/r)}
\|f\|_{L^{p_{1}}(\mathbb{R}^{2}\times \mathbb{T})}
\|g\|_{L^{p_{2}}(\mathbb{R}^{2}\times \mathbb{T})}
\|h\|_{L^{p_{3}}(\mathbb{R}^{2}\times\mathbb{T})}.
\end{aligned}
\end{equation*}
This shows \eqref{7-13ADE-1} and then Lemma \ref{HC-9} is proved.
\end{proof}

\section{Temporal decay estimate for the 3D linear Klein-Gordon problem}\label{B}

In this Appendix, we establish the time-decay rate of smooth solutions to the 3D homogeneous Klein-Gordon
equation on $\mathbb{R}^{2}\times \mathbb{T}$ with compactly supported initial data. The related result can be stated as:
\begin{thm}\label{8-1}
{\it Let $V(t,x,y)$ be the smooth solution to the following problem
\begin{equation}\label{8-1}\begin{cases}
(\partial_{t}^2-\Delta_{x,y}+1)V=0,\\[2mm]
(V(0),\dot{V}(0))=(V_{0},V_{1})(x,y)\in C_{0}^{\infty}(\mathbb{R}^{2}\times \mathbb{T}),
\end{cases}
\end{equation}
then one has
\begin{equation}\label{8-2}
\begin{aligned}
\sup\limits_{|\rho|\leq2}\|\partial_{x,y}^{\rho}V(t)\|_{L^{\infty}(\mathbb{R}^{2}\times \mathbb{T})}
+\sup\limits_{|\rho|\leq1}\|\partial_{x,y}^{\rho}\dot{V}(t)\|_{L^{\infty}(\mathbb{R}^{2}\times \mathbb{T})}\lesssim(1+t)^{-1}.
\end{aligned}
\end{equation}}
\end{thm}

\begin{proof}
We express $V,V_{0},V_{1}$ in problem \eqref{8-1} as follows
\begin{equation}\label{8-3}
\begin{aligned}
V=\sum\limits_{n\in \mathbb{Z}}
V_{n}(t,x)e^{iny},~~V_{0}=\sum\limits_{n\in \mathbb{Z}}
V_{0;n}(x)e^{iny},~~V_{1}=\sum\limits_{n\in \mathbb{Z}}
V_{1;n}(x)e^{iny}.
\end{aligned}
\end{equation}
Substituting \eqref{8-3} into problem \eqref{8-1} yields that for each $n\in \mathbb{Z}$,
\begin{equation}\label{8-6}\begin{cases}
(\partial_{t}^2-\Delta_{x}+n^{2}+1)V_{n}(t, x)=0,\\[2mm]
(V_{n}(0,x), \dot{V_{n}}(0, x))=(V_{0; n}(x),V_{1; n}(x)).
\end{cases}
\end{equation}
Taking the Fourier transformation in the $x$-variables for problem \eqref{8-6} obtains
\begin{equation}\label{8-7}
\begin{aligned}
\widehat{V_{n}}(t,\xi)=\frac{(\widehat{V_{0;n}}(\xi)- i\widehat{V_{1;n}}(\xi)/\Lambda_{n}(\xi))e^{ it\Lambda_{n}(\xi)}}{2}
+\frac{(\widehat{V_{0;n}}(\xi)+i\widehat{V_{1;n}}(\xi)/\Lambda_{n}(\xi))e^{-it\Lambda_{n}(\xi)}}{2}.
\end{aligned}
\end{equation}
Due to
\begin{equation}\label{8-8}
\begin{aligned}
V(t,x,y)=\frac{1}{(2\pi)^{2}}\sum_{n\in \mathbb{Z}}\int_{\mathbb{R}^{2}}e^{i(ny+x\cdot\xi)}\widehat{V_{n}}(t,\xi)d\xi,
\end{aligned}
\end{equation}
then substituting \eqref{8-7} into \eqref{8-8} obtains the explicit expression of $V$.

To obtain \eqref{8-2}, it suffices to prove
\begin{equation}\label{8-9}
\|V(t)\|_{L^{\infty}(\mathbb{R}^{2}\times \mathbb{T})}\lesssim (1+t)^{-1}
\end{equation}
since the treatments of $\|\partial_{x,y}^{\rho}(V, \dot{V})(t)\|_{L^{\infty}(\mathbb{R}^{2}\times \mathbb{T})}$  ($|\rho|\ge 1$)  are completely analogous. Moreover, we only need to show \eqref{8-9} for $t\geq 1$ since \eqref{8-9}  is trivial for $0\leq t \leq 1$.
In view of \eqref{8-7} and \eqref{8-8}, we have
\begin{equation}\label{8-AA}
\|V(t)\|_{L^{\infty}(\mathbb{R}^{2}\times \mathbb{T})}\lesssim K_1+K_2,
\end{equation}
where
\begin{equation*}\label{8-10}
K_1=\Big\|\sum_{n\in \mathbb{Z}}\int_{\mathbb{R}^{2}}e^{i(ny+x\cdot\xi)}e^{\pm it\Lambda_{n}(\xi)}
\widehat{V_{0; n}}(\xi)d\xi\Big\|_{L^{\infty}(\mathbb{R}^{2}\times \mathbb{T})}
\end{equation*}
and
\begin{equation*}\label{8-11}
K_2=\Big\|\sum_{n\in \mathbb{Z}}\int_{\mathbb{R}^{2}}e^{i(ny+x\cdot\xi)}e^{\pm it\Lambda_{n}(\xi)}
\cdot\frac{\widehat{V_{1; n}}(\xi)}{\Lambda_{n}(\xi)}d\xi\Big\|_{L^{\infty}(\mathbb{R}^{2}\times \mathbb{T})}.
\end{equation*}

We next focus on the estimate of $K_1$ since the analogous argument can be done for $K_2$. Note that
\begin{equation}\label{8-1A}
K_1\leq K_{11}+K_{12},
\end{equation}
where
\begin{equation*}\label{8-12}
K_{11}=\Big\|\int_{\mathbb{R}^{2}}e^{ix\cdot\xi}e^{\pm it\sqrt{1+|\xi|^{2}}}
\widehat{V_{0; 0}}(\xi)d\xi\Big\|_{L^{\infty}(\mathbb{R}^{2})}
\end{equation*}
and
\begin{equation*}\label{8-13}
K_{12}=\sum_{n\in \mathbb{Z},|n|\geq1}\Big\|\int_{\mathbb{R}^{2}}e^{i(ny+x\cdot\xi)}e^{\pm it\Lambda_{n}(\xi)}
\widehat{V_{0; n}}(\xi)d\xi\Big\|_{L^{\infty}(\mathbb{R}^{2}\times \mathbb{T})}.
\end{equation*}

For $K_{11}$ in \eqref{8-1A}, by the definition of $\varphi_l, \psi_{-1}$ in Section 2, we have
\begin{equation}\label{8-14}
K_{11}=\Big\|\int_{\mathbb{R}^2}e^{ix\cdot \xi}e^{\pm it\sqrt{1+|\xi|^2}}\Big(\psi_{-1}(\xi)+\sum\limits_{l=1}^{+\infty}\varphi_l(\xi)\Big)\widehat{V_{0; 0}}(\xi)d\xi\Big\|_{L^{\infty}(\mathbb{R}^2)}.
\end{equation}
It follows from Young's inequality and Lemma~\ref{HC-1} that
\begin{equation}\label{8-15}
\begin{aligned}
&\Big\|\int_{\mathbb{R}^{2}}e^{ix\cdot\xi}e^{\pm it\sqrt{1+|\xi|^{2}}}
\psi_{-1}(\xi)\widehat{V_{0; 0}}(\xi)
d\xi\Big\|_{L^{\infty}(\mathbb{R}^{2})}\\
=&\Big\|\mathcal{F}^{-1}(e^{\pm it\sqrt{1+|\xi|^{2}}}
\psi_{-1}(\xi))\ast\mathcal{F}^{-1}(\widehat{V_{0;0}})\Big\|_{L^{\infty}(\mathbb{R}^{2})}\\
\lesssim& \Big\|\mathcal{F}^{-1}(e^{\pm it\sqrt{1+|\xi|^{2}}}
\psi_{-1}(\xi))\Big\|_{L^{\infty}(\mathbb{R}^{2})}\cdot\Big\|
\mathcal{F}^{-1}(\widehat{V_{0;0}})\Big\|_{L^{1}(\mathbb{R}^{2})}\\
=&\Big\|\mathcal{F}^{-1}(e^{\pm it\sqrt{1+|\xi|^{2}}}
\psi_{-1}(\xi))\Big\|_{L^{\infty}(\mathbb{R}^{2})}\cdot\|V_{0; 0}\|_{L^{1}(\mathbb{R}^{2})}\\
\lesssim& (1+t)^{-1}\|V_{0}\|_{L^{1}(\mathbb{R}^{2}\times \mathbb{T})}\lesssim (1+t)^{-1}.
\end{aligned}
\end{equation}
Similarly, by Young's inequality and Lemma \ref{HC-2}, one has
\begin{equation}\label{8-16}
\begin{aligned}
&\sum\limits_{l=1}^{+\infty}\Big\|\int_{\mathbb{R}^{2}}e^{ix\cdot\xi}e^{\pm it\sqrt{1+|\xi|^{2}}}
\varphi_{l}(\xi)\widehat{V_{0;0}}(\xi)
d\xi\Big\|_{L^{\infty}(\mathbb{R}^{2})}\\
=&\sum\limits_{l=1}^{+\infty}\Big\|\mathcal{F}^{-1}(e^{\pm it\sqrt{1+|\xi|^{2}}}
\varphi_{l}(\xi))\ast\mathcal{F}^{-1}(\varphi_{[l-2,l+2]}(\xi)\widehat{V_{0;0}}(\xi))\Big\|_{L^{\infty}(\mathbb{R}^{2})}\\
\lesssim& \sum\limits_{l=1}^{+\infty}\Big\|\mathcal{F}^{-1}(e^{\pm it\sqrt{1+|\xi|^{2}}}
\varphi_{l}(\xi))\Big\|_{L^{\infty}(\mathbb{R}^{2})}\cdot\Big\|
\mathcal{F}^{-1}(\varphi_{[l-2,l+2]}(\xi)\widehat{V_{0;0}}(\xi))\Big\|_{L^{1}(\mathbb{R}^{2})}\\
\lesssim& \sum\limits_{l=1}^{+\infty}2^{2l}(1+t)^{-1}\cdot\Big\|
\mathcal{F}^{-1}\Big(\varphi_{[l-2,l+2]}(\xi)(1+|\xi|^{2})^{-2}(1+|\xi|^{2})^{2}
\widehat{V_{0;0}}(\xi)\Big)\Big\|_{L^{1}(\mathbb{R}^{2})}\\
\lesssim& \sum\limits_{l=1}^{+\infty}2^{2l}(1+t)^{-1}\Big\|
\mathcal{F}^{-1}\Big(\varphi_{[l-2,l+2]}(\xi)(1+|\xi|^{2})^{-2}\Big)\Big\|_{L^{1}(\mathbb{R}^{2})}
\Big\|\mathcal{F}^{-1}\Big((1+|\xi|^{2})^{2}\widehat{V_{0;0}}(\xi)\Big)\Big\|_{L^{1}(\mathbb{R}^{2})}\\
\lesssim& \sum\limits_{l=1}^{+\infty}2^{-2l}(1+t)^{-1}\Big\|(1-\Delta_{x})^{2}V_{0;0}(x)\Big\|_{L^{1}(\mathbb{R}^{2})}
\lesssim (1+t)^{-1}.
\end{aligned}
\end{equation}
Collecting \eqref{8-14}-\eqref{8-16} yields
\begin{equation}\label{8-17}
K_{11}\lesssim (1+t)^{-1}.
\end{equation}

Next, we deal with $K_{12}$ in \eqref{8-1A}. Note that
\begin{equation}\label{8-18}
\begin{aligned}
K_{12}&=\sum_{n\in \mathbb{Z},|n|\geq1}\Big\|\int_{\mathbb{R}^{2}}e^{ix\cdot\xi}e^{\pm it\Lambda_{n}(\xi)}
\widehat{V_{0;n}}(\xi)d\xi\Big\|_{L^{\infty}(\mathbb{R}^{2})}
\\&=\sum_{n\in \mathbb{Z},|n|\geq1}\Big\|\int_{\mathbb{R}^{2}}e^{ix\cdot\xi}e^{\pm it\Lambda_{n}(\xi)}
\Big[\psi_{-1}(\xi)+\sum\limits_{l=1}^{+\infty}\varphi_{l}(\xi)\Big]
\widehat{V_{0;n}}(\xi)d\xi\Big\|_{L^{\infty}(\mathbb{R}^{2})}.
\end{aligned}
\end{equation}
Similarly to \eqref{8-15}, we have
\begin{equation}\label{8-19}
\begin{aligned}
&\sum_{n\in \mathbb{Z},|n|\geq1}\Big\|\int_{\mathbb{R}^{2}}e^{ix\cdot\xi}e^{\pm it\Lambda_{n}(\xi)}
\psi_{-1}(\xi)\widehat{V_{0;n}}(\xi)d\xi\Big\|_{L^{\infty}(\mathbb{R}^{2})}\\
=&\sum_{n\in \mathbb{Z},|n|\geq1}n^{-4}\Big\|\mathcal{F}^{-1}(e^{\pm it\Lambda_{n}(\xi)}
\psi_{-1}(\xi))\ast\mathcal{F}^{-1}\Big(n^{4}\cdot\widehat{V_{0;n}}\Big)\Big\|_{L^{\infty}(\mathbb{R}^{2})}\\
\lesssim& \sum_{n\in \mathbb{Z},|n|\geq1}n^{-4}\Big\|\mathcal{F}^{-1}(e^{\pm it\Lambda_{n}(\xi)}
\psi_{-1}(\xi))\Big\|_{L^{\infty}(\mathbb{R}^{2})}\cdot\Big\|
\mathcal{F}^{-1}\Big(\widehat{(\partial_{y}^{4}V_{0})_{n}}(\xi)\Big)\Big\|_{L^{1}(\mathbb{R}^{2})}\\
\lesssim& \sum_{n\in \mathbb{Z},|n|\geq1}n^{-4}|n|(1+t)^{-1}\cdot\|(\partial_{y}^{4}V_{0})_{n}\|_{L^{1}(\mathbb{R}^{2})}\\
\lesssim& (1+t)^{-1}\|\partial_{y}^{4}V_{0}\|_{L^{1}(\mathbb{R}^{2}\times \mathbb{T})}\lesssim (1+t)^{-1}.
\end{aligned}
\end{equation}
On the other hand, analogous to \eqref{8-16},~one has
\begin{equation}\label{8-20}
\begin{aligned}
&\sum_{n\in \mathbb{Z},|n|\geq1}\sum\limits_{l=1}^{+\infty}\Big\|\int_{\mathbb{R}^{2}}e^{ix\cdot\xi}e^{\pm it\Lambda_{n}(\xi)}
\varphi_{l}(\xi)\widehat{V_{0;n}}(\xi)
d\xi\Big\|_{L^{\infty}(\mathbb{R}^{2})}\\
=&\sum_{n\in \mathbb{Z},|n|\geq1}\sum\limits_{l=1}^{+\infty}n^{-4}\Big\|\mathcal{F}^{-1}(e^{\pm it\Lambda_{n}(\xi)}
\varphi_{l}(\xi))\ast\mathcal{F}^{-1}(\varphi_{[l-2,l+2]}(\xi)n^{4}\cdot\widehat{V_{0;n}}(\xi))\Big\|_{L^{\infty}(\mathbb{R}^{2})}\\
\lesssim& \sum_{n\in \mathbb{Z},|n|\geq1}\sum\limits_{l=1}^{+\infty}n^{-4}\Big\|\mathcal{F}^{-1}(e^{\pm it\Lambda_{n}(\xi)}
\varphi_{l}(\xi))\Big\|_{L^{\infty}(\mathbb{R}^{2})}\cdot\Big\|
\mathcal{F}^{-1}(\varphi_{[l-2,l+2]}(\xi)\widehat{(\partial_{y}^{4}V_{0})_{n}}(\xi))\Big\|_{L^{1}(\mathbb{R}^{2})}\\
\lesssim& \sum_{n\in \mathbb{Z},|n|\geq1}\sum\limits_{l=1}^{+\infty}|n|^{-3} 2^{2l}(1+t)^{-1}\cdot\Big\|
\mathcal{F}^{-1}\Big(\varphi_{[l-2,l+2]}(\xi)(1+|\xi|^{2})^{-2}(1+|\xi|^{2})^{2}\widehat{(\partial_{y}^{4}V_{0})_{n}}(\xi)\Big)
\Big\|_{L^{1}(\mathbb{R}^{2})}\\
\lesssim& \sum_{n\in \mathbb{Z},|n|\geq1}\sum\limits_{l=1}^{+\infty}n^{-2}2^{2l}(1+t)^{-1}\Big\|
\mathcal{F}^{-1}\Big(\varphi_{[l-2,l+2]}(\xi)(1+|\xi|^{2})^{-2}\Big)\Big\|_{L^{1}(\mathbb{R}^{2})}\times\\
&\qquad\qquad\qquad\qquad\Big\|\mathcal{F}^{-1}\Big((1+|\xi|^{2})^{2}\widehat{(\partial_{y}^{4}V_{0})_{n}}(\xi)\Big)
\Big\|_{L^{1}(\mathbb{R}^{2})}\\
\lesssim& \sum_{n\in\mathbb{Z},|n|\geq1}\sum\limits_{l=1}^{+\infty}n^{-2}2^{-2l}(1+t)^{-1}
\Big\|(1-\Delta_{x})^{2}(\partial_{y}^{4}V_{0})_{n}(x)\Big\|_{L^{1}(\mathbb{R}^{2})}\\
\lesssim& \sum_{n\in\mathbb{Z},|n|\geq1}\sum\limits_{l=1}^{+\infty} n^{-2}2^{-2l}(1+t)^{-1}\Big\|(1-\Delta_{x})^{2}\partial_{y}^{4}V_{0}(x,y)\Big\|_{L^{1}(\mathbb{R}^{2}\times \mathbb{T})}\lesssim (1+t)^{-1}.
\end{aligned}
\end{equation}
By \eqref{8-18}-\eqref{8-20},~then
\begin{equation}\label{8-22}
\begin{aligned}
K_{12}\lesssim (1+t)^{-1}.
\end{aligned}
\end{equation}
Collecting \eqref{8-1A}, \eqref{8-17} and \eqref{8-22} derives $K_1\lesssim (1+t)^{-1}$. Then the proof of the Theorem B.1
is completed.
\end{proof}

\bibliographystyle{amsplain}

\end{document}